\documentclass[a4paper]{amsart}                % American Mathematical Society document class for articles
\usepackage[latin1]{inputenc}                  % Codificación europea del teclado (acentos, símbolos teclado ...)
\usepackage[T1]{fontenc}                       % Acentos de palabras y división silábica
\usepackage[spanish,english]{babel}            % Idioma principal el ingles
\usepackage{graphicx}                  % Incorporar gráficos
\usepackage{amsmath,amssymb,amsthm}            % Paquetes de la American Mathematical Society
\usepackage{latexsym}                          % Símbolos
\usepackage{delarray}                          % Delimiter array: permite integrar delimitadores en las opciones del entorno array
\usepackage{bbm}                               % Fuentes para los números reales, racionales, ...
\usepackage{hyperref}                          % Enlaces de "hipertexto"
\usepackage{calrsfs,eufrak,mathrsfs,upgreek}   % Tipo de letra
\usepackage{bbding,trfsigns}                   % Símbolos
\usepackage{wasysym}                           % Variar el tamaño de los símbolos
\usepackage{datetime}                          % Imprimir fecha y hora

\DeclareMathAlphabet{\mathpzc}{OT1}{pzc}{m}{it} % Tipo de letra \mathpzc

\newtheorem{Th}{Theorem}[section]              % Enumera los teoremas de acuerdo con la sección (Theorem 1.1, Theorem 1.2 , ...)

\newtheorem{Rm}{Remark}[section]

\newtheorem{Prop}{Proposition}[section]
\newtheorem{Lema}{Lemma}[section]

\title[Harmonic analysis operators in the Schr\"odinger setting]{$L^p$-boundedness properties of variation operators in the Schr\"odinger setting}
\author{J.J. Betancor}
\address{Departamento de Análisis Matemático\\
Universidad de la Laguna\\
Campus de Anchieta, Avda. Astrofísico Francisco Sánchez, s/n\\
38271 La Laguna (Sta. Cruz de Tenerife), Spain}
\email{jbetanco@ull.es}

\author{J.C. Fari\~na}
\address{Departamento de Análisis Matemático\\
Universidad de la Laguna\\
Campus de Anchieta, Avda. Astrofísico Francisco Sánchez, s/n\\
38271 La Laguna (Sta. Cruz de Tenerife), Spain}
\email{jcfarina@ull.es}

\author{E. Harboure}
\address{Instituto de Matemática Aplicada del Litoral, IMAL\\
Universidad Nacional del Litoral\\
C/ G\"emes\\
Santa Fe, Argentina}
\email{harboure@santafe-conicet.gov.ar}

\author{L. Rodr\'{\i}guez-Mesa}
\address{Departamento de Análisis Matemático\\
Universidad de la Laguna\\
Campus de Anchieta, Avda. Astrofísico Francisco Sánchez, s/n\\
38271 La Laguna (Sta. Cruz de Tenerife), Spain}
\email{lrguez@ull.es}

\thanks{This paper is partially supported by MTM2007/65609.}

\begin{document}

  \maketitle                                  % Si no se activa esta opción no se pone ni el título ni los autores en el encabezado de cada página

  \begin{abstract}
    In this paper we establish the $L^p$-boundedness properties of the variation operators associated with the heat semigroup, Riesz transforms and commutator between Riesz transforms and multiplication by $BMO(\mathbb{R}^n)$-functions in the Schr\"odinger setting.
  \end{abstract}

  \section{Introduction and main results}

  We consider the time independent Schr\"odinger operator on $\mathbb{R}^n$, $n\ge 3$, defined by
  $$
  \mathcal{L}=-\Delta+V,
  $$
  where the potential $V$ is nonzero, nonnegative and belongs, for some $q\geq n/2$, to the reverse H\"older class $B_q$, that is, there exists $C>0$ such that
  $$
  \Big(\frac{1}{|B|}\int_B V(x)^qdx\Big)^{1/q}\le \frac{C}{|B|}\int_B V(x)dx,
  $$
  for every ball $B$ in $\mathbb{R}^n$. Since any nonnegative polynomial belongs to $B_q$ for every  $1<q<\infty$, the Hermite operator $-\Delta+|x|^2$ falls under our considerations.

  Harmonic analysis associated with the operator $\mathcal{L}$ has been studied by several authors in the last decade. Most of them had, as starting point, the paper of Shen \cite{Sh}. This author investigated $L^p$-boundedness properties of the Riesz transforms formally defined in the $\mathcal{L}$-setting by
  \begin{equation}\label{R1}
  R^\mathcal{L}=\nabla \mathcal{L}^{-1/2},
  \end{equation}
  where as usual $\nabla$ denotes the gradient operator and the negative square root $\mathcal{L}^{-1/2}$ of $\mathcal{L}$ is given by the functional calculus as follows
  \begin{equation}\label{R2}
  \mathcal{L}^{-1/2}f(x)=\int_{\mathbb{R}^n}K(x,y)f(y)dy,
  \end{equation}
  being
  $$
  K(x,y)=-\frac{1}{2\pi}\int_\mathbb{R}
(-i\tau)^{-1/2}\Gamma(x,y,\tau)d\tau.
$$
Here, for every $\tau\in \mathbb{R}$, $\Gamma(x,y,\tau)$, $x,y\in \mathbb{R}^n$, represents the fundamental solution for the operator $\mathcal{L}+i\tau$.

Bongioanni, Harboure and Salinas \cite{BHS3} studied the behavior in $L^p$ spaces of the commutator operator $[R^\mathcal{L},b]$ defined by
$$
[R^\mathcal{L},b]f=bR^\mathcal{L}(f)-R^\mathcal{L}(bf),
$$
where $b$ is in an appropriate class containing the space $BMO$ of bounded mean oscillation functions.

The heat semigroup $\{W_t^\mathcal{L}\}_{t>0}$ generated by the operator $-\mathcal{L}$ can be written on $L^2(\mathbb{R}^n)$ as the following integral operator
$$
W_t^\mathcal{L}(f)(x)=\int_{\mathbb{R}^n} W_t^\mathcal{L}(x,y)f(y)dy,\,\,\,f\in L^2(\mathbb{R}^n).
$$
The semigroup $\{W_t^\mathcal{L}\}_{t>0}$ is $C_0$ in $L^p(\mathbb{R}^n)$, $1\le p<\infty$, but it is not Markovian. The main properties of the kernel $W_t^\mathcal{L}(x,y)$, $x,y\in \mathbb{R}^n$, $t>0$, can be encountered in \cite{DGMTZ}.

Other operators associated with the Schr\"odinger operator $\mathcal{L}$ have been studied on $L^p(\mathbb{R}^n)$ and on other kind of function spaces also in \cite{BHS1}, \cite{BHS2}, \cite{DGMTZ}, \cite{DZ}, \cite{GLP}, \cite{TD}, and \cite{Z}.

The variation operators were introduced in ergodic theory
(\cite{JKRW}) in order to measure the speed of convergence. Suppose
that $\{T_t\}_{t>0}$ is an uniparametric system of linear operators
bounded in $L^p(\mathbb{R}^n)$, for some $1\le p<\infty$,  such that
$ \lim_{t\to 0^+}T_tf $ exists in $L^p(\mathbb{R}^n)$.
%If
%$\{t_j\}_{j\in \mathbb{N}}$ is a real decreasing sequence that
%converges to zero, the oscillation operator $O(T_t;\{t_j\}_{j\in
%\mathbb{N}})$ is defined by
%$$
%O(T_t;\{t_j\}_{j\in \mathbb{N}})(f)(x)=\Big(\sum_{j=0}^\infty \sup_{t_{j+1}\le \varepsilon_{j+1}<\varepsilon_j\le t_j}|T_{\varepsilon_j}f(x)-T_{\varepsilon_{j+1}}f(x)|^2\Big)^{1/2},\quad f\in L^p(\mathbb{R}^n).
%$$
%We consider the space $E(\{t_j\}_{j\in \mathbb{N}})$ that consists of all those functions $w:(0,\infty)\longrightarrow \mathbb{R}$, such that
%$$
%\|w\|_{E(\{t_j\}_{j\in \mathbb{N}})}=\Big(\sum_{j=0}^\infty \sup_{t_{j+1}\le \varepsilon_{j+1}<\varepsilon_j\le t_j}|w({\varepsilon_j})-w({\varepsilon_{j+1}})|^2\Big)^{1/2}<\infty.
%$$
%$\|\cdot\|_{E(\{t_j\}_{j\in \mathbb{N}})}$ is a seminorm on $E(\{t_j\}_{j\in \mathbb{N}})$. We have that
%$$
%O(T_t;\{t_j\}_{j\in \mathbb{N}})(f)=\|T_tf\|_{E(\{t_j\}_{j\in \mathbb{N}})}.
%$$

If $\rho>2$, the variation operator $V_\rho(T_t)$ is given by
$$
V_\rho(T_t)(f)(x)=\sup_{\{t_j\}_{j\in \mathbb{N}}\searrow 0}\Big(\sum_{j=0}^\infty |T_{t_j}f(x)-T_{t_{j+1}}f(x)|^\rho\Big)^{1/\rho},\,\,\,f\in L^p(\mathbb{R}^n).
$$
Here the supremum is taken over all the real decreasing sequences $\{t_j\}_{j\in \mathbb{N}}$ that converge to zero. By $E_\rho$ we denote the space that includes all the functions $w:(0,\infty)\longrightarrow \mathbb{R}$, such that
$$
\|w\|_{E_\rho}= \sup_{\{t_j\}_{j\in \mathbb{N}}\searrow 0}\Big(\sum_{j=0}^\infty |w(t_j)-w(t_{j+1})|^\rho\Big)^{1/\rho}<\infty.
$$
$\|\cdot\|_{E_\rho}$ is a seminorm on $E_\rho$. It can be written
$$
V_\rho(T_t)(f)=\|T_tf\|_{E_\rho}.
$$

Variation operators for $C_0$-semigroups of operators and singular
integrals have been analyzed on $L^p$-spaces by Campbell et al.
(\cite{CJRW1}, \cite{CJRW2} and \cite{JKRW}). More recently, in
\cite{CMMTV} and \cite{HMTV} the authors have studied variation
operators on $L^p$-spaces for semigroups and Riesz transforms in the
Ornstein-Uhlenbeck and Hermite settings.

As it was mentioned, for every $1\le p<\infty$, the semigroup $\{W_t^\mathcal{L}\}_{t>0}$ is $C_0$ in $L^p(\mathbb{R}^n)$, that is, for every $f\in L^p(\mathbb{R}^n)$, $W_t^\mathcal{L}(f)\to f$, as $t\to 0^+$, in $L^p(\mathbb{R}^n)$. The $L^p$-boundedness properties of the oscillation and variation operators for $\{W_t^\mathcal{L}\}_{t>0}$ are established in the following.

\begin{Th}\label{LVarLp} Let $\rho > 2$ . Then,
the variation operator $V_\rho(W_t^{\mathcal{L}})$ is bounded from
$L^p(\mathbb R^n)$ into itself, for every $1<p<\infty$, and from
$L^1(\mathbb R^n)$ into $L^{1,\infty}(\mathbb R^n)$.
\end{Th}

Note that, since $\{W_t^\mathcal{L}\}_{t>0}$ is not Markovian, none part of Theorem \ref{LVarLp} can be deduced from \cite[Theorem 3.3]{JR}.

According to standard ideas, Shen in \cite{Sh} actually defined
(although he did not write it in this way), for every $\ell
=1,\cdots,n$, the $\ell $-th Riesz transform in the
$\mathcal{L}$-context by
\begin{equation}\label{R3}
R^\mathcal{L}_\ell (f)(x)=\lim_{\varepsilon\to 0^+}\int_{|x-y|>\varepsilon}R^\mathcal{L}_\ell (x,y)f(y)dy,\,\,\,\mbox{a.e.}\,\,\,x\in \mathbb{R}^n,
\end{equation}
provided that $f\in L^p(\mathbb{R}^n)$ and either

(i) $1\le p<\infty$ and $V\in B_n$; or

(ii) $1\le p<p_0$, $\frac{1}{p_0}=\frac{1}{q}-\frac{1}{n}$, and $V\in B_q$, $n/2\le q<n$.

Here, for every $x,y\in \mathbb{R}^n$, $x\neq y$,
$$
R^\mathcal{L}_\ell (x,y)=-\frac{1}{2\pi}\int_{\mathbb{R}}(-i\tau)^{-1/2}\frac{\partial}{\partial x_\ell }\Gamma(x,y,\tau)d\tau.
$$
In the sequel we complete Shen's result proving that actually
the limit in (\ref{R3}) exists and the Riesz transform
$R^\mathcal{L}=(R_1^\mathcal{L},\cdots,R_n^\mathcal{L})$ can be
represented by (\ref{R1}) on $C_c^\infty(\mathbb{R}^n)$, the space
of $C^\infty$-functions in $\mathbb{R}^n$ that have compact support.

\begin{Prop} \label{Pvriesz} Let $\ell =1,\cdots,n$. Suppose that one of the following two conditions holds:

(i) $f\in L^p(\mathbb{R}^n)$, $1\le p<\infty$, and $V\in B_n$;

(ii) $f\in L^p(\mathbb{R}^n)$, $1\le p<p_0$, where $\frac{1}{p_0}=\frac{1}{q}-\frac{1}{n}$, and $V\in B_q$, $n/2\le q<n$.

Then, there exists the following limit
$$
\lim_{\varepsilon\to 0^+}\int_{|x-y|>\varepsilon}R^\mathcal{L}_\ell (x,y)f(y)dy,\,\,\,\mbox{a.e.}\,\,\,x\in \mathbb{R}^n.
$$
Moreover, if $f\in C_c^\infty(\mathbb{R}^n)$, $\mathcal{L}^{-1/2}f$, as defined in (\ref{R2}), admits partial derivative with respect to $x_\ell$ almost everywhere in $\mathbb{R}^n$ and
\begin{equation}\label{repre}
\frac{\partial}{\partial x_\ell}\mathcal{L}^{-1/2}f(x)=\lim_{\varepsilon\to 0^+}\int_{|x-y|>\varepsilon}R^\mathcal{L}_\ell (x,y)f(y)dy,\,\,\,\mbox{a.e.}\,\,\,x\in \mathbb{R}^n.
\end{equation}
\end{Prop}

%The Riesz transform in the $\mathcal{L}$-context is defined formally by
%$$
%R^\mathcal{L}=\nabla\mathcal{L}^{-1/2}.
%$$
%We denote, for every $\ell =1,\cdots,n$, the $\ell $-th Riesz transform by
%$$
%R^\mathcal{L}_\ell =\frac{\partial}{\partial x_\ell }\mathcal{L}^{-1/2}.
%$$
%In \cite[Theorems 0.5 and 0.8]{Sh} it was proved that, for every $\ell =1,\cdots,n$, the operator $R^\mathcal{L}_\ell $ is bounded
%
%(i) from $L^p(\mathbb{R}^n)$ into itself, $1<p<\infty$, and from $L^1(\mathbb{R}^n)$ into $L^{1,\infty}(\mathbb{R}^n)$, provided that $V\in B_n$.
%
%(ii) from $L^p(\mathbb{R}^n)$ into itself, for every $1<p<p_0$, where $\frac{1}{p_0}=\frac{1}{q}-\frac{1}{n}$, and $V\in B_q$, $n/2\le q<n$.
%
%We complete the Shen's results proving that the Riesz transforms $R^\mathcal{L}_\ell $, $\ell =1,\cdots,n$, can be pointwise represented by a principal value integral.
%
%\begin{Prop} \label{Pvriesz} Let $\ell =1,\cdots,n$. Suppose that one of the following two conditions holds:
%
%(i) $f\in L^p(\mathbb{R}^n)$, $1\le p<\infty$, and $V\in B_n$;
%
%(ii) $f\in L^p(\mathbb{R}^n)$, $1<p<p_0$, where $\frac{1}{p_0}=\frac{1}{q}-\frac{1}{n}$, and $V\in B_q$, $n/2\le q<n$.
%
%Then,
%$$
%R^\mathcal{L}_\ell (f)(x)=\lim_{\varepsilon\to 0^+}\int_{|x-y|>\varepsilon}R^\mathcal{L}_\ell (x,y)f(y)dy,\,\,\,\mbox{a.e.}\,\,\,x\in \mathbb{R}^n,
%$$
%where
%$$
%R^\mathcal{L}_\ell (x,y)=-\frac{1}{2\pi}\int_{\mathbb{R}}(-i\tau)^{-1/2}\frac{\partial}{\partial x_\ell }\Gamma(x,y,\tau)d\tau,\,\,\,x,y\in \mathbb{R}^n,\,\,x\neq y.
%$$
%\end{Prop}

For every $\varepsilon>0$, the $\varepsilon$-truncation $R^{\mathcal{L},\varepsilon}_\ell $ of $R^\mathcal{L}_\ell $ is defined as usual by
$$
R^{\mathcal{L},\varepsilon}_\ell (f)(x)=\int_{|x-y|>\varepsilon}R^\mathcal{L}_\ell (x,y)f(y)dy,\,\,\,\ell =1,\cdots,n.
$$

The behavior on $L^p$ spaces of the variation operators associated
with the family of truncations $\{R^{\mathcal{L},\varepsilon}_\ell
\}_{\varepsilon>0}$, $\ell =1,\cdots,n$, is contained in the following result.

\begin{Th} \label{VarRiesz} Let $\ell =1,\cdots,n$. Assume that $\rho>2$ and that $V\in B_q$, with $q\ge n/2$, and $1<p<p_0$, where $\frac{1}{p_0}=\Big(\frac{1}{q}-\frac{1}{n}\Big)_+$. Then, the variation operator $V_\rho(R^{\mathcal{L},\varepsilon}_\ell
)$ is bounded from $L^p(\mathbb{R}^n)$ into itself. Moreover,
$V_\rho(R^{\mathcal{L},\varepsilon}_\ell)$ is bounded from
$L^1(\mathbb{R}^n)$ into $L^{1,\infty}(\mathbb{R}^n)$.

%Then the variation operator $V_\rho(R^{\mathcal{L},\varepsilon}_\ell
%)$ is bounded
%
%(i) from $L^p(\mathbb{R}^n)$ into itself, $1<p<\infty$, and from $L^1(\mathbb{R}^n)$ into $L^{1,\infty}(\mathbb{R}^n)$, provided that $V\in B_n$.
%
%(ii) from $L^p(\mathbb{R}^n)$ into itself, when $1<p<p_0$, where $\frac{1}{p_0}=\frac{1}{q}-\frac{1}{n}$, and from $L^1(\mathbb{R}^n)$ into $L^{1,\infty}(\mathbb{R}^n)$, when $V\in B_q$, $n/2\le q<n$.
\end{Th}

By $BMO(\mathbb{R}^n)$ we denote the usual John-Nirenberg space. A locally integrable function $b$ on $\mathbb{R}^n$ is in $BMO(\mathbb{R}^n)$ if and only if there exists $C>0$ such that
$$
\frac{1}{|B|}\int_B|b(x)-b_B|dx\le C,
$$
for every ball $B$ in $\mathbb{R}^n$. Here
$b_B=\frac{1}{|B|}\int_Bb(x)dx$, where $B$ is a ball in
$\mathbb{R}^n$. For $f\in BMO(\mathbb{R}^n)$ we define
$$
\|f\|_{BMO(\mathbb{R}^n)}=\sup_{B}\frac{1}{|B|}\int_B|b(x)-b_B|dx,
$$
where the supremum is taken over all the balls $B$ in
$\mathbb{R}^n$.

For every $V\in B_{n/2}$, we consider the function $\gamma$ defined
by
$$
\gamma(x)=\sup\Big\{r>0:\frac{1}{r^{n-2}}\int_{B(x,r)}V(y)dy\le 1\Big\},\,\,\,x\in \mathbb{R}^n.
$$
Under our assumptions it is not hard to see that
$0<\gamma(x)<\infty$, for all $x\in \mathbb{R}^n$. This function
$\gamma$ was introduced in \cite{Sh0} when the potential $V$
satisfies that
$$
\max_{x\in B}V(x)\le C\frac{1}{|B|}\int_B V(y)dy,
$$
for every ball $B$ in $\mathbb{R}^n$, to study the Neumann problem
for the operator $\mathcal{L}$ in the region above a Lipschitz
graph. The main properties  of $\gamma$ were showed in \cite[Section
1]{Sh} (see also \cite{Sh0}). Here, the function $\gamma$ plays a
crucial role.

In \cite{BHS3} the space $BMO_\theta(\gamma)$, $\theta\ge 0$, was
defined as follows. Let $\theta\ge 0$. A locally integrable function
$b$ in $\mathbb{R}^n$ is in $BMO_\theta(\gamma)$ provided that
$$
\frac{1}{|B(x,r)|}\int_{B(x,r)}|b(y)-b_{B(x,r)}|dy\le
C\Big(1+\frac{r}{\gamma(x)}\Big)^\theta,
$$
for all $x\in \mathbb{R}^n$ and $r>0$. We denote for $b\in
BMO_\theta(\gamma)$
$$
\|b\|_{BMO_\theta(\gamma)}=\sup_{x\in \mathbb{R}^n, \,\,r>0}
\frac{1}{|B(x,r)|}\int_{B(x,r)}|b(y)-b_{B(x,r)}|dy\Big(1+\frac{r}{\gamma(x)}\Big)^{-\theta}.
$$
Note that $BMO(\mathbb{R}^n)=BMO_0(\gamma)\subset
BMO_\theta(\gamma)\subset BMO_{\theta'}(\gamma)$, when $0\le
\theta\le \theta'$. We set $BMO_\infty (\gamma)=\bigcup_{\theta>0}
BMO_\theta(\gamma)$. As it is pointed out in \cite{BHS3},
$BMO_\infty(\gamma)$ is in general larger than $BMO(\mathbb{R}^n)$.

For $b\in BMO_\infty(\gamma)$ and $\ell =1,\cdots,n$, the commutator
operator $C_{b,\ell }^\mathcal{L}$ is defined by
$$
C_{b,\ell }^\mathcal{L}(f)=bR^\mathcal{L}_\ell (f)-R^\mathcal{L}_\ell (bf),\,\,\,f\in C_c^\infty(\mathbb{R}^n).
$$
Note that $bf\in L^1(\mathbb{R}^n)$, for every $f\in C_c^\infty(\mathbb{R}^n)$ and $b\in BMO_\infty(\gamma)$.

In \cite[Theorem 1]{BHS3} it was shown that, for every $b\in
BMO_\infty(\gamma)$ and $\ell =1,\cdots,n$, the operator $C_{b,\ell
}^\mathcal{L}$ is bounded from $L^p(\mathbb{R}^n)$ into itself,
provided that $1<p<p_0$, where
$\frac{1}{p_0}=\Big(\frac{1}{q}-\frac{1}{n}\Big)_+$ and $V\in B_q$,
$q\ge n/2$.

In the next result we obtain a pointwise representation of the
commutator operator by a principal value integral.

\begin{Prop} \label{PvComm} Let $\ell =1,\cdots,n$. If $b\in BMO_\infty(\gamma)$, $V\in B_q$, with $q\ge n/2$, and $f\in L^p(\mathbb{R}^n)$, where $1<p<p_0$ and $\frac{1}{p_0}=\Big(\frac{1}{q}-\frac{1}{n}\Big)_+$, then
$$
C_{b,\ell }^\mathcal{L}(f)(x)=\lim_{\varepsilon\to 0^+}\int_{|x-y|>\varepsilon}(b(x)-b(y))R^\mathcal{L}_\ell (x,y)f(y)dy,\,\,\,\mbox{a.e.}\,\,x\in \mathbb{R}^n.
$$
\end{Prop}

For every $b\in BMO_\infty(\gamma)$, $\varepsilon>0$, and $\ell =1,\cdots,n$, we define the $\varepsilon$-truncation $C_{b,\ell }^{\mathcal{L},\varepsilon}$ of $C_{b,\ell }^\mathcal{L}$ by
$$
C_{b,\ell }^{\mathcal{L},\varepsilon}(f)(x)=\int_{|x-y|>\varepsilon}(b(x)-b(y))R^\mathcal{L}_\ell (x,y)f(y)dy,\quad x\in\mathbb{R}^n.
$$
The $L^p$-boundedness properties of the variation operators
associated with the family of truncations $\{C_{b,\ell
}^{\mathcal{L},\varepsilon}\}_{\varepsilon>0}$ are contained in the
following.

\begin{Th} \label{VarComm} Let $\ell =1,\cdots,n$ and $b\in BMO_\infty(\gamma)$. Assume that $V\in B_q$, with $q\ge n/2$, and $1<p<p_0$, where $\frac{1}{p_0}=\Big(\frac{1}{q}-\frac{1}{n}\Big)_+$. Then, if $\rho>2$, the variation operator $V_\rho(C_{b,\ell }^{\mathcal{L},\varepsilon})$
is bounded from $L^p(\mathbb{R}^n)$ into itself.
\end{Th}
In \cite[p. 516]{Sh} it was proved that if $V$ is a nonnegative polynomial, then $V\in B_q$, for every $1<q<\infty$. Then, as special cases of our results appear the corresponding ones to the Hermite operator $H=-\Delta+|x|^2$ (\cite{CMMTV} and \cite{CLMT}).

This paper is organized as follows. In Section 2 we describe a
general procedure that we shall use to prove our main results and we
present the $L^p$-properties of the variation operators associated
with the classical ($V\equiv 0$) heat semigroup $\{W_t\}_{t>0}$,
Riesz transforms $R_\ell $ and their commutators $C_{b,\ell }$, $\ell
=1,\cdots,n$, that will be very useful to our purposes. The proof of
Theorem \ref{LVarLp} is carried out in Section 3. We present proofs
of Proposition \ref{Pvriesz} and Theorem \ref{VarRiesz} in Section
4. Finally, in Section 5 we give proofs for Proposition \ref{PvComm}
and Theorem \ref{VarComm}.

Throughout this paper by $c$ and $C$ we will always denote positive constants that can change in each occurrence. If $1<p<\infty$, by $p'$ we represent the exponent conjugated of $p$, that is, $p'=\frac{p}{p-1}$.

\section{Procedure and auxiliary results}

%As we said in the previous section we consider the Schr\"odinger operator $\mathcal{L}=-\Delta+V$ in $\mathbb{R}^n$, $n\ge 3$, where $V\in B_q$, for some $q\ge n/2$.

%For every $V\in B_{n/2}$, we define the function $\gamma$ by
%$$
%\gamma(x)=\sup\Big\{r>0:\frac{1}{r^{n-2}}\int_{B(x,r)}V(y)dy\le 1\Big\},\,\,\,x\in \mathbb{R}^n.
%$$
%This function $\gamma$ was introduced in \cite{Sh0} when the potential $V$ satisfies that
%$$
%\max_{x\in B}V(x)\le C\frac{1}{|B|}\int_B V(y)dy,
%$$
%for every ball $B$ in $\mathbb{R}^n$, to study the Neumann problem for the operator $\mathcal{L}$ in the region above %a Lipschitz graph. The main properties  of $\gamma$ were showed in \cite[Section 1]{Sh}. Here, the function $\gamma$ %plays a crucial role.

In order to establish boundedness properties for harmonic analysis
operators (semigroup, maximal operators, Riesz transforms,
Littlewood-Paley functions,....) in the Schr\"{o}dinger setting it
is usual to exploit that $\mathcal{L}$ is actually a nice
perturbation of the Laplacian operator $-\Delta$. We now describe a
general procedure to analyze harmonic operators associated with the
Schr\"{o}dinger operator. Suppose that $T$ is a
$\mathcal{L}$-harmonic analysis operator and that $\mathcal{T}$ is
the corresponding $\Delta$-harmonic operator. According to the
function $\gamma$ described above, we split $\mathbb R^n \times
\mathbb R^n$ in two parts as follows
$$
A= \{(x,y) \in \mathbb R^n \times \mathbb R^n : |x-y| <\gamma(x)\},
$$
and
$$
B= (\mathbb R^n \times \mathbb R^n) \setminus A .
$$
The sets $A$ y $B$ are usually called local and global region
associated with $\mathcal{L},$ respectively. The local part of the
operator $T$ is defined by
$$
T_{{\rm loc}}(f)(x) = T(f \chi_{B(x,\gamma(x))}) (x),\,\,\,x\in
\mathbb{R}^n.
$$
In a similar way we consider the operator
$$
\mathcal{T}_{{\rm loc}}(f)(x)=
\mathcal{T}(f\chi_{B(x,\gamma(x))})(x),\,\,\,x\in \mathbb{R}^n.
$$
Then, we decompose the operator $T$ through
$$
T=(T_{{\rm loc}}-\mathcal{T}_{{\rm loc}}) + \mathcal{T}_{{\rm loc}} + (T-T_{{\rm loc}}).
$$
It is clear that $(T-T_{{\rm loc}})(f)(x) =
T(f\chi_{\mathbb{R}^n\setminus B(x,\gamma(x))})(x)$. Since the set
$\{(x,y): x \in \mathbb R^n, y \in \mathbb{R}^n\setminus
B(x,\gamma(x))\}$ is sufficiently far away from the diagonal (usual
line of singularities) $\{(x,x):x\in \mathbb R^n\}$, the operator
$T-T_{{\rm loc}}$ will be controlled by a positive and $L^p$-bounded
operator. We said that $\mathcal{L}$ is actually a nice perturbation
of the Laplacian operator $-\Delta$. That niceness leads to the
operators $T$ and $\mathcal{T}$ to have the same singularity in the
local region. Then, cancelation of singularities in $T_{{\rm loc}} -
\mathcal{T}_{{\rm loc}}$ takes place and $T_{{\rm
loc}}-\mathcal{T}_{{\rm loc}}$ is controlled by a positive and
$L^p$-bounded operator for the given range of $p$. In this way
$L^p$-boundedness of $T$ is reduced to the corresponding property
for the operator $\mathcal{T}_{{\rm loc}}$. Finally,
$L^p$-boundedness properties of the operator $\mathcal{T}_{{\rm
loc}}$ rely on well known properties for the classical harmonic
operator $\mathcal{T} $.

This procedure has been used in \cite{Sh} to establish $L^p$-boundedness properties for $\mathcal{L}$-Riesz transform.

We will employ this comparative method to describe the behavior in
$L^p$-spaces of the  variation operators for the heat semigroup
$\{W_t^{\mathcal{L}}\}_{t>0}$ generated by $-\mathcal{L}$, Riesz
transforms and commutators of Riesz transforms with the
multiplication by $BMO_\infty(\gamma)$-functions in the
Schr\"{o}dinger setting. Following this pattern we will need to know
 $L^p$-boundedness properties of the variation operators associated
with the classical heat semigroup, Riesz transforms and commutators
between Riesz transforms and multiplication by
$BMO(\mathbb{R}^n)$-functions.

In \cite[Theorem 3.3]{JR} it was established that that if
$\{T_t\}_{t>0}$ is a symmetric diffusion semigroup (in the sense of
\cite[p. 65]{Stein}) then the variation operator $V_\rho(T_t)$, with
$\rho >2$, is bounded from $L^p(\mathbb{R}^n)$ into itself for
every $1 < p < \infty$. This result applies to the symmetric
diffusion semigroup $\{W_t\}_{t>0}$ generated by the Euclidean
Laplacian $\Delta$. Recently, Crescimbeni, Mac{\'\i}as, Men\'arguez,
Torrea and Viviani (\cite{CMMTV}), by using vector valued Calder\'on-Zygmund
theory, have proved that the operators $V_\rho(W_t)$ map
$L^1(\mathbb R^n)$ into $L^{1,\infty}(\mathbb R^n)$, for each $\rho
>2$. These results are contained in the following.
\begin{Th} \label{VarLp}(\cite[Theorem 3.3]{JR} and \cite[Theorem 1.1]{CMMTV})
Let $\rho>2$. Then, the variation operator $V_\rho(W_t)$ is bounded
from $L^p(\mathbb R^n)$ into itself, for every $1<p<\infty$, and
from $L^1(\mathbb R^n)$ into $L^{1,\infty}(\mathbb R^n)$.
\end{Th}

For every $\ell =1,\cdots ,n$, the $\ell$-th Riesz transform $R_\ell $ is defined  by
$$
R_\ell f(x) = \lim_{\varepsilon \rightarrow 0^+} \int_{|x-y| > \varepsilon}\frac{x_\ell -y_\ell }{|x-y|^{n+1}} f(y) dy,\;\;\mbox{a.e.} \;\; x \in \mathbb R^n,
$$
for each $f \in L^p(\mathbb R^n)$, $1 \leq p < \infty$. For every $\varepsilon>0$ and $\ell =1,\ldots, n$, the $\varepsilon$-truncation of $R_\ell $ is given by
$$
R_\ell ^\varepsilon(f)(x)=\int_{|x-y| > \varepsilon}\frac{x_\ell -y_\ell }{|x-y|^{n+1}} f(y) dy,\quad x\in \mathbb{R}^n.
$$
We denote the kernel function of $R_\ell$ by
$R_\ell(z)=\frac{z_\ell}{|z|^{n+1}}$, $z=(z_1,\ldots,z_n)\in
\mathbb{R}^n\setminus \{0\}$. The variation operators for $R_\ell
^\varepsilon, \ell =1,\ldots,n$, were investigated in \cite{CJRW1}
and \cite{CJRW2} where the following results were proved.

\begin{Th} \label{RVarLp}(\cite[Theorems 1.1 and 1.2]{CJRW1} and \cite[Corollary 1.4]{CJRW2}). Let $\ell =1,\cdots,n$.
Assume that $\rho>2$. Then, the variation operator $V_\rho(R_\ell
^{\varepsilon})$ is bounded from $L^p(\mathbb R^n)$ into itself, for
every $1<p<\infty$, and from $L^1(\mathbb R^n)$ into
$L^{1,\infty}(\mathbb R^n)$.
\end{Th}

Let us mention that by using transference methods Gillespie and
Torrea (\cite[Theorem B]{GiTo}) have proved dimension free
$L^p(\mathbb R^n,|x|^\alpha dx)$ norm inequalities, for every $1<p<
\infty$ and $-1 < \alpha < p-1$, for variation operators of the
Riesz transform $R_\ell ,\;\ell =1,\cdots,n$.

Next let $b\in {\rm BMO}(\mathbb{R}^n)$. It is well known that, for
every $\ell =1,...,n$, the commutator operator $C_{b,\ell }$ defined
by
 $$
 C_{b,\ell }(f)=bR_\ell (f)-R_\ell (bf),
 $$
is bounded from $L^p(\mathbb{R}^n)$ into itself, and for each $f\in L^p(\mathbb{R}^n)$, $1<p<\infty$,
$$
C_{b,\ell }f(x)=\lim_{\varepsilon \rightarrow 0^+}C_{b,\ell }^\varepsilon (f)(x),\quad {\rm a.e. }\;\;x\in \mathbb{R}^n,
$$
where
$$
C_{b,\ell
}^\varepsilon(f)(x)=\int_{|x-y|>\varepsilon}(b(x)-b(y))R_\ell(x-y)f(y)dy, x\in \mathbb{R}^n,
$$
(\cite[Theorem I]{CRW}).

$L^p$-boundedness properties for the variation operators associated
with $C_{b,\ell }$, $\ell =1,\cdots,n$, are stated in the following.
To our knowledge the result is new, so we provide a proof.

\begin{Th}\label{CVarBMO} Let $b\in {\rm BMO}(\mathbb{R}^n)$ and $\ell =1,2,...,n$. Assume that $\rho>2$.
Then, the variation operator $V_{\rho}(C_{b,\ell }^\varepsilon)$
is bounded from $L^p(\mathbb{R}^n)$ into itself, for every
$1\!<\!p\!<\!\infty$.
\end{Th}

\begin{proof} Let $1<p<\infty$ and $f\in L^p(\mathbb{R}^n)$. Inspired in the ideas
developed in \cite{GiTo} our goal is to estimate the sharp maximal
function
$$
(V_\rho (C_{b,\ell }^\varepsilon )(f))^\# (x)=\sup
_{r>0}\frac{1}{|B(x,r)|}\int _{B(x,r)}|V_\rho (C_{b,\ell
}^\varepsilon)(f)(y)-c_{B(x,r)}|dy,\quad x\in \mathbb{R}^n,
$$
where, for every $x\in \mathbb{R}^n$ and $r>0$, $c_{B(x,r)}$ is a
constant that will be specified later.

Let $x_0\in \mathbb{R}^n$ and $r_0>0$ and denote by $B=B(x_0,r_0)$.
We decompose $f=f_1+f_2$, where $f_1=f\chi _{4B}$ and $f_2=f\chi
_{(4B)^c}$, and we write
\begin{eqnarray*}
C_{b,\ell }^\varepsilon f&=&(b-b_B)R_\ell ^\varepsilon (f)-R_\ell ^\varepsilon ((b-b_B)f_1)-R_\ell ^\varepsilon ((b-b_B)f_2)\\
&=&A_1^\varepsilon (f)+A_2^\varepsilon (f)+A_3^\varepsilon (f),\quad
\varepsilon >0.
\end{eqnarray*}

We have that
\begin{eqnarray*}
\int_{(4B)^c}\frac{|b(y)-b_B|}{|x-y|^n}|f(y)|dy&=&\sum_{k=2}^\infty
\int_{2^kr_0<|y-x_0|\leq 2^{k+1}r_0}\frac{|b(y)-b_B|}{|x-y|^n}|f(y)|dy\\
&\hspace{-6cm}\leq
&\hspace{-3cm}\sum_{k=2}^\infty\left(\int_{2^kr_0<|y-x_0|\leq
2^{k+1}r_0}\frac{|b(y)-b_B|^{p'}}{|x-y|^n}dy
\right)^{1/p'}\left(\int_{2^kr_0<|y-x_0|\leq 2^{k+1}r_0}\frac{|f(y)|^p}{|x-y|^n}dy\right)^{1/p}\\
&\hspace{-6cm}\leq&\hspace{-3cm}C\sum_{k=2}^\infty \left(\frac{1}{(2^{k+1}r_0)^n}\int_{|y-x_0|\leq 2^{k+1}r_0}|b(y)-b_B|^{p'}dy\right)^{1/p'}\frac{1}{(2^kr_0)^{n/p}}||f||_{L^p(\mathbb{R}^n)}\\
&\hspace{-6cm}\leq&\hspace{-3cm}C||b||_{{\rm
BMO}(\mathbb{R}^n)}||f||_{L^p(\mathbb{R}^n)}\sum_{k=2}^\infty
\frac{k}{(2^kr_0)^{n/p}},\quad x\in B.
\end{eqnarray*}
Then, we deduce that
\begin{multline*}
V_{\rho }(R_\ell ^\varepsilon)((b-b_B)f_2)(x)\\
\shoveleft{\hspace{21mm}=\sup_{\{\varepsilon_j\}_{j\in
\mathbb{N}}\searrow 0}\left(\sum_{j=0}^\infty
\left|\int_{\varepsilon _{j+1}<|x-y|<\varepsilon _j}R_\ell(x-y)(b(y)-b_B)f(y)\chi _{(4B)^c}(y)dy\right|^\rho \right)^{1/\rho}}\\
\shoveleft{\hspace{21mm}\leq \sup_{\{\varepsilon_j\}_{j\in
\mathbb{N}}\searrow 0}\sum_{j=0}^\infty
\int_{\varepsilon _{j+1}<|x-y|<\varepsilon _j}\frac{|b(y)-b_B||f(y)|}{|x-y|^n}\chi _{(4B)^c}(y)dy}\\
\leq\int_{(4B)^c}\frac{|b(y)-b_B||f(y)|}{|x-y|^n}dy\leq
Cr_0^{-n/p}||b||_{{\rm
BMO}(\mathbb{R}^n)}||f||_{L^p(\mathbb{R}^n)},\quad x\in
B.\hspace{10mm}
\end{multline*}

We denote $c_B=-V_\rho (R_\ell ^\varepsilon )((b-b_B)f_2)(x_0)$. It is clear that $c_B=V_\rho (A_3^\varepsilon
)(f)(x_0)$.

We have
\begin{eqnarray}\label{c1}
\lefteqn{\frac{1}{|B|}\int_B|V_\rho (C_{b,\ell }^\varepsilon )(f)(x)-c_B|dx } \\
&=&\frac{1}{|B|}\int_B|\,||A_1^\varepsilon (f)(x)+A_2^\varepsilon
(f)(x)+A_3^\varepsilon (f)(x)||_{E_\rho }
-||A_3^\varepsilon (f)(x_0)||_{E_\rho}\,|dx\nonumber\\
&\leq&\frac{1}{|B|}\int_B\,||A_1^\varepsilon (f)(x)+A_2^\varepsilon (f)(x)+A_3^\varepsilon (f)(x)-A_3^\varepsilon (f)(x_0)||_{E_\rho }dx \nonumber\\
&\leq& \frac{1}{|B|}\int_B||A_1^\varepsilon (f)(x)||_{E_\rho }dx+\frac{1}{|B|}\int_B||A_2^\varepsilon (f)(x)||_{E_\rho }dx\nonumber\\
&&\hspace{-3.5mm}+\hspace{1mm}\frac{1}{|B|}\int_B||A_3^\varepsilon
(f)(x)-A_3^\varepsilon (f)(x_0)||_{E_\rho }dx.\nonumber
\end{eqnarray}

We analyze each term. Firstly we obtain
\begin{eqnarray}\label{c2}
\frac{1}{|B|}\int_B||A_1^\varepsilon (f)(x)||_{E_\rho }dx&=&\frac{1}{|B|}\int_B|b(x)-b_B|V_\rho (R_\ell ^\varepsilon )(f)(x)dx\nonumber\\
&\leq& \left(\frac{1}{|B|}\int_B|b(x)-b_B|^{r'}dx\right)^{1/r'}\left(\frac{1}{|B|}\int_B\Big(V_\rho (R_\ell ^\varepsilon )(f)(x)\Big)^rdx\right)^{1/r}\nonumber\\
&\leq&C||b||_{{\rm BMO}(\mathbb{R}^n)}\mathcal{M}_r(V_\rho (R_\ell
^\varepsilon )(f))(z),\quad z\in B.
\end{eqnarray}
Here $1\le r<\infty$ and $\mathcal{M}_r$ is the $r$-maximal function
defined by
$$
\mathcal{M}_r(g)(x)=\sup_{x\in
B}\left(\frac{1}{|B|}\int_B|g(y)|^rdy\right)^{1/r},\,\,\,x\in
\mathbb{R}^n,
$$
for every measurable function $g$ on $\mathbb{R}^n$.

On the other hand, according to \cite[Theorem A]{CJRW2}, we have
that
\begin{eqnarray*}
\frac{1}{|B|}\int_B||A_2^\varepsilon (f)(x)||_{E_\rho }dx&\leq&\left(\frac{1}{|B|}\int_B|V_\rho (R_\ell ^\varepsilon )((b-b_B)f_1)(x)|^\beta dx\right)^{1/\beta }\\
&\leq& C\left(\frac{1}{|B|}\int_{4B}|b(x)-b_B|^{\beta }|f(x)|^\beta dx\right)^{1/\beta}\\
&\leq&C\left(\frac{1}{|B|}\int_{4B}|b(x)-b_B|^{s'\beta }dx\right)^{1/(s'\beta )}\left(\frac{1}{|B|}\int_{4B}|f(x)|^{s\beta }dx\right)^{1/(s\beta )}\\
&\leq&C||b||_{{\rm
BMO}(\mathbb{R}^n)}\left(\frac{1}{|B|}\int_{4B}|f(x)|^{s\beta
}dx\right)^{1/(s\beta )},
\end{eqnarray*}
where $1<s,\beta<\infty$. Then, for every $1<r<\infty$, we have
\begin{equation}\label{c3}
\frac{1}{|B|}\int_B||A_2^\varepsilon (f)(x)||_{E_\rho }dx\leq
C||b||_{{\rm BMO}(\mathbb{R}^n)}\mathcal{M}_r(f)(z),\quad z\in B.
\end{equation}
In order to study
$\|A_3^\varepsilon(f)(x)-A_3^\varepsilon(f)(x_0)\|_{E_\rho}$ we use
a procedure developed in \cite{GiTo}. We have that
\begin{equation}\label{B3}
\|R_\ell ^\varepsilon((b-b_B)f_2)(x)-R_\ell
^\varepsilon((b-b_B)f_2)(x_0)\|_{E_\rho} \leq H_1(x) + H_2(x),\;\; x
\in B,
\end{equation}
where, for every $x \in B$,
$$
H_1(x)=\left\|\int_{|x-y| > \varepsilon}(R_\ell (x-y)-R_\ell
(x_0-y))(b(y)-b_B)f_2(y)dy\right\|_{E_\rho}
$$
and
$$
H_2(x) =\left\|\int_{\mathbb{R}^n}\left(\chi _{\{|x-y|>\varepsilon
\}}(y)-\chi _{\{|x_0-y|>\varepsilon \}}(y)\right)R_\ell
(x_0-y)(b(y)-b_B)f_2(y)dy\right\|_{E_\rho}.
$$
By using Minkowski inequality and well known properties of the
function $R_\ell $ we get
\begin{eqnarray}\label{B4}
H_1(x) & \leq & \int_{\mathbb R^n}|R_\ell (x-y)-R_\ell (x_0-y)||b(y)-b_B||f(y)|\chi_{(4B)^c}(y) dy \nonumber\\
& \leq & C \sum_{k=2}^\infty \int_{2^kr_0 \leq |x_0 -y| \leq 2^{k+1}r_0}\frac{|x-x_0|}{|x_0-y|^{n+1}}|b(y)-b_B||f(y)| dy \nonumber \\
& \leq & C \sum_{k=2}^\infty
\frac{1}{2^k}\frac{1}{(2^kr_0)^n}\int_{2^{k+1}B}|b(y)-b_B||f(y)|dy ,
\;\; x \in B.
\end{eqnarray}

To analyze $H_2$ we split the
integral appearing in the norm in four terms as follows. Let
$\{\varepsilon_j\}_{j\in \mathbb{N}}$ be a real decreasing sequence
that converges to zero. We have
\begin{multline}\label{B5}
\int_{\mathbb R^n}\left|\chi_{\{\varepsilon_{j+1} < |x-y| < \varepsilon_j\}}(y) - \chi_{\{\varepsilon_{j+1} < |x_0-y| < \varepsilon_j\}}(y)\right| |R_\ell (x_0-y)||b(y)-b_B||f_2(y)|dy\\
\shoveleft{\hspace{20mm}\leq C\left(\int_{\mathbb R^n} \chi_{\{\varepsilon_{j+1} < |x-y| < \varepsilon_{j+1} +r_0\}}(y) \chi_{\{\varepsilon_{j+1} < |x-y| < \varepsilon_j\}}(y)\frac{1}{|x_0-y|^n}|b(y)-b_B||f_2(y)|dy\right.}\\
\shoveleft{\hspace{23mm}+ \int_{\mathbb R^n} \chi_{\{\varepsilon_{j} < |x_0-y| < \varepsilon_{j} +r_0\}}(y)  \chi_{\{\varepsilon_{j+1} < |x-y| < \varepsilon_j\}}(y)\frac{1}{|x_0-y|^n}|b(y)-b_B||f_2(y)|dy}\\
\shoveleft{\hspace{23mm}+ \int_{\mathbb R^n} \chi_{\{\varepsilon_{j+1} < |x_0-y| < \varepsilon_{j+1} +r_0\}}(y)  \chi_{\{\varepsilon_{j+1} < |x_0-y| < \varepsilon_j\}}(y)\frac{1}{|x_0-y|^n}|b(y)-b_B||f_2(y)|dy}\\
\shoveleft{\hspace{23mm}+\left. \int_{\mathbb R^n} \chi_{\{\varepsilon_{j} < |x-y| < \varepsilon_{j} +r_0\}}(y)  \chi_{\{\varepsilon_{j+1} < |x_0-y| < \varepsilon_j\}}(y)\frac{1}{|x_0-y|^n}|b(y)-b_B||f_2(y)|dy \right)}\\
\shoveleft{\hspace{20mm}= C(H^j_{2,1}(x) + H^j_{2,2}(x) +
H^j_{2,3}(x) + H^j_{2,4}(x)),\hspace{15mm}x\in
B\,\,\,\mbox{and}\,\,\,j\in \mathbb{N}.}
\end{multline}
We observe that $\frac{4}{3} |x-y| \geq |x_0-y| \geq
\frac{4}{5}|x-y|$, when $y \notin 4B$ and $x \in B$. Moreover, if
$x\in B$, then $H_{2,m}^j(x)=0$, for $m=1,3$, when $j\in \mathbb{N}$
and $r_0\ge \varepsilon_{j+1}$ and $H_{2,m}^j(x)=0$, when $m=2,4$,
$j\in \mathbb{N}$ and $r_0\ge \varepsilon_{j}$. For every $j\in
\mathbb{N}$, H\"{o}lder inequality leads to
$$
H_{2,1}^j(x) \leq C \left(\int_{\mathbb
R^n}\chi_{\{\varepsilon_{j+1} < |x-y| < \varepsilon_j\}}(y)
\frac{1}{|x-y|^{ns}}(|b(y)-b_B||f_2(y)|)^sdy\right)^{\frac{1}{s}}v_{j+1}^{\frac{1}{s'}},\;\;
x\in B,
$$
$$
H^j_{2,2}(x) \leq C \left(\int_{\mathbb
R^n}\chi_{\{\max\{\varepsilon_{j+1},
\frac{3}{4}\varepsilon_j\}<|x-y|< \varepsilon_j\}}(y)
\frac{1}{|x-y|^{ns}}(|b(y)-b_B||f_2(y)|)^sdy\right)^{\frac{1}{s}}v_j^{\frac{1}{s'}},
\;\; x\in B,
$$
$$
H_{2,3}^j(x) \leq C \left(\int_{\mathbb
R^n}\chi_{\{\varepsilon_{j+1} < |x_0-y| < \varepsilon_j\}}(y)
\frac{1}{|x_0-y|^{ns}}(|b(y)-b_B||f_2(y)|)^sdy\right)^{\frac{1}{s}}v_{j+1}^{\frac{1}{s'}},\;\;
x\in B,
$$
and
$$
H^j_{2,4}(x) \leq C \left(\int_{\mathbb
R^n}\chi_{\{\max\{\varepsilon_{j+1},
\frac{4}{5}\varepsilon_j\}<|x_0-y|< \varepsilon_j\}}(y)
\frac{1}{|x_0-y|^{ns}}(|b(y)-b_B||f_2(y)|)^sdy\right)^{\frac{1}{s}}v_j^{\frac{1}{s'}},
\;\; x\in B.
$$
Here we take $1 < s < \rho$, and $v_j=(\varepsilon_j+r_0)^n -
\varepsilon_j^n$, $j\in \mathbb{N}$. Note that $v_j \leq
C(\max\{r_0,\varepsilon_j\})^{n-1}r_0$, $j\in \mathbb{N}$, for a
certain $C > 0$ that does not depend on $j$.

We define the set $\mathcal{A}=\{j\in \mathbb{N}: r_0<\varepsilon
_j\}$. We have that
\begin{eqnarray*}
H_{2,1}^j(x)&\leq&C\frac{v_{j+1}^{1/s'}}{\varepsilon _{j+1}^{(n-1)/s'}}\left(\int_{\mathbb{R}^n}\chi _{\{\varepsilon _{j+1}<|x-y|<\varepsilon _j\}}(y)\frac{(|b(y)-b_B||f_2(y)|)^s}{|x-y|^{n+s-1}}dy\right)^{1/s}\\
&\leq&Cr_0^{1/s'}\left(\int_{\mathbb{R}^n}\chi _{\{\varepsilon
_{j+1}<|x-y|<\varepsilon
_j\}}(y)\frac{(|b(y)-b_B||f_2(y)|)^s}{|x-y|^{n+s-1}}dy\right)^{1/s},
\end{eqnarray*}
for every $x\in B$ and $j+1\in \mathcal{A}$. In a similar way we can
see that
$$
H_{2,2}^j(x)\leq Cr_0^{1/s'}\left(\int_{\mathbb{R}^n}\chi
_{\{\varepsilon _{j+1}<|x-y|<\varepsilon
_j\}}(y)\frac{(|b(y)-b_B||f_2(y)|)^s}{|x-y|^{n+s-1}}dy\right)^{1/s},\quad
x\in B\mbox{ and }j\in \mathcal{A},
$$
$$
H_{2,3}^j(x)\leq Cr_0^{1/s'}\left(\int_{\mathbb{R}^n}\chi
_{\{\varepsilon _{j+1}<|x_0-y|<\varepsilon
_j\}}(y)\frac{(|b(y)-b_B||f_2(y)|)^s}{|x_0-y|^{n+s-1}}dy\right)^{1/s},\quad
x\in B\mbox{ and }j+1\in \mathcal{A},
$$
and
$$
H_{2,4}^j(x)\leq Cr_0^{1/s'}\left(\int_{\mathbb{R}^n}\chi
_{\{\varepsilon _{j+1}<|x_0-y|<\varepsilon
_j\}}(y)\frac{(|b(y)-b_B||f_2(y)|)^s}{|x_0-y|^{n+s-1}}dy\right)^{1/s},\quad
x\in B\mbox{ and }j\in \mathcal{A}.
$$
Hence, we get by using Minkowski's inequality
\begin{eqnarray}\label{B6}
\Big(\sum _{j=0}^\infty |H_{2,1}^j(x)&+&H_{2,2}^j(x)|^\rho \Big)^{1/\rho}\leq C\left(\sum_{j+1\in \mathcal{A}}|H_{2,1}^j(x)|^\rho +\sum_{j\in \mathcal{A}}|H_{2,2}^j(x)|^\rho\right)^{1/\rho}\nonumber\\
&\leq&C\left(\sum_{j=0}^\infty \left(\int_{\mathbb{R}^n}\chi _{\{\varepsilon _{j+1}<|x-y|<\varepsilon _j\}}(y)\frac{(|b(y)-b_B||f_2(y)|)^s}{|x-y|^{n+s-1}}dy\right)^{\rho/s}r_0^{\rho /s'}\right)^ {1/\rho}\nonumber\\
&\leq&
C\left(\int_{\mathbb{R}^n}\frac{(|b(y)-b_B||f_2(y)|)^s}{|x-y|^{n+s-1}}dy\right)^{1/s}r_0^{1/s'}
\nonumber\\
&\leq& C\left(\sum_{k=2}^\infty
\frac{1}{(2^kr_0)^n}\int_{|x_0-y|<2^{k+1}r_0}(|b(y)-b_B||f(y)|)^sdy\frac{1}{2^{k(s-1)}}\right)^{1/s},\quad
x\in B.
\end{eqnarray}

In a similar way we get
\begin{eqnarray}\label{B7}
\Big(\sum_{j=0}^\infty |H_{2,3}^j(x)+ H_{2,4}^j(x)|^\rho \Big)^{1/\rho }\nonumber&&\\
&\hspace{-4cm}\le&\hspace{-2cm}C\left(\sum_{k=1}^\infty
\frac{1}{(2^kr_0)^n}\int_{|x_0-y|<2^{k+1}r_0}(|b(y)-b_B||f(y)|)^sdy\frac{1}{2^{k(s-1)}}\right)^{1/s},\,\,\,x\in
B.
\end{eqnarray}

By combining (\ref{B3}), (\ref{B4}), (\ref{B6}), and (\ref{B7}) it
follows that
\begin{eqnarray*}
||A_3^\varepsilon f(x)-A_3^\varepsilon f(x_0)||_{E_\rho }&=&||R_\ell ((b-b_B)f_2)(x)-R_\ell ((b-b_B)f_2)(x_0)||_{E_\rho }\\
&\leq&C\left(\sum_{k=1}^\infty \frac{2^{-k}}{(2^kr_0)^n}\int_{|x_0-y|<2^kr_0}|b(y)-b_B||f(y)|dy\right.\\
&&+\left.\left(\sum_{k=1}^\infty
\frac{2^{-k(s-1)}}{(2^kr_0)^n}\int_{|x_0-y|<2^kr_0}|(b(y)-b_B)f(y)|^sdy\right)^{1/s}\right),
\end{eqnarray*}
for $x\in B$. Then, H\"older's inequality implies that
\begin{eqnarray}\label{c4}
||A_3^\varepsilon (f)(x)-A_3^\varepsilon (f)(x_0)||_{E_\rho }&\leq&
C\left(\sum_{k=1}^\infty \frac{1}{2^k}\left(\frac{1}{(2^kr_0)^n}\int_{|x_0-y|<2^kr_0}|b(y)-b_B|^{r'}dy\right)^{1/r'}\right.\nonumber\\
&&\times \left(\frac{1}{(2^kr_0)^n}\int_{|x_0-y|<2^kr_0}|f(y)|^rdy\right)^{1/r}\nonumber\\
&&+\left(\sum_{k=1}^\infty \frac{1}{2^{k(s-1)}}\left(\frac{1}{(2^kr_0)^n}\int_{|x_0-y|<2^kr_0}|b(y)-b_B|^{st'}dy\right)^{1/t'}\right.\nonumber\\
&&\times\left.\left.\left(\frac{1}{(2^kr_0)^n}\int_{|x_0-y|<2^kr_0}|f(y)|^{st}dy\right)^{1/t}\right)^{1/s}\right)\nonumber\\
&\leq &C\left(\sum_{k=1}^\infty \frac{k}{2^k}||b||_{{\rm BMO}(\mathbb{R}^n)}\mathcal{M}_r(f)(z)\right.\nonumber\\
&&\left.+\left(\sum_{k=1}^\infty \frac{k^s}{2^{k(s-1)}}\right)^{1/s}||b||_{{\rm BMO}(\mathbb{R}^n)}\mathcal{M}_{ts}(f)(z)\right)\nonumber\\
&\leq&C||b||_{{\rm
BMO}(\mathbb{R}^n)}(\mathcal{M}_r(f)(z)+\mathcal{M}_{ts}(f)(z)),\quad
x,z\in B,
\end{eqnarray}
where $1<t,r<\infty$.

From (\ref{c4}) it follows
\begin{equation}\label{c5}
\frac{1}{|B|}\int_B||A_3^\varepsilon f(x)-A_3^\varepsilon
f(x_0)||_{E_\rho }dx\leq C||b||_{{\rm
BMO}(\mathbb{R}^n)}\mathcal{M}_r(f)(z),\quad z\in B,
\end{equation}
where $1<r<\infty$.

By combining (\ref{c1}), (\ref{c2}), (\ref{c3}) and (\ref{c5}) we
conclude that
$$
(V_\rho (C_{b,\ell }^\varepsilon )(f))^\# (x)\leq C||b||_{{\rm
BMO}(\mathbb{R}^n)}(\mathcal{M}_r(f)(x)+\mathcal{M}_r(V_\rho (R_\ell
^\varepsilon )(f))(x)), \quad x\in B,
$$
where $1<r<p$.

Since ${\mathcal{M}}_r$ is bounded from $L^p(\mathbb{R}^n)$ into
itself provided that $1<r<p<\infty$, \cite[Theorem A]{CJRW2} and
\cite[Corollary 1, p. 154]{Stein1} allow us to conclude that $V_\rho
(C_{b,\ell }^\varepsilon )$ is bounded from $L^p(\mathbb{R}^n)$ into
itself.
\end{proof}

In \cite{YYZ} Dachun Yang, Dongyong Yang and Yuan Zhou introduced
localized Riesz transform associated with the classical Laplacian
and Schr\"odinger operators. Here, we consider localized commutator
operators with the classical Riesz transforms.

Let $b\in BMO_\theta(\gamma)$, $\theta>0$. We define, for every
$\varepsilon>0$, the local truncation
$C_{b,\ell}^{\varepsilon,\rm{loc}}f$ of $f\in L^p(\mathbb{R}^n)$,
$1< p<\infty$, by
$$
C_{b,\ell}^{\varepsilon,\rm{loc}}(f)(x)=
\int_{\varepsilon<|x-y|<\gamma(x)}(b(x)-b(y))R_\ell(x-y)f(y)dy,\,\,\,
x\in \mathbb{R}^n.
$$

In order to define the local commutator $C_{b,\ell}^{\rm{loc}}$ on
$L^p(\mathbb{R}^n)$, $1<p<\infty$, we show the $L^p$ boundedness
properties for the maximal operator $C_{b,\ell}^{*,\rm{loc}}$
defined by
$$
C_{b,\ell}^{*,\rm{loc}}(f)(x)=\sup_{\varepsilon>0}|C_{b,\ell}^{\varepsilon,\rm{loc}}(f)(x)|,\,\,\,x\in
\mathbb{R}^n.
$$
Previously we state an auxiliary result that will be also useful in
other sections of this paper. According to \cite[Proposition
5]{DGMTZ} we choose a sequence $\{x_k\}_{k=1}^\infty \subset
\mathbb{R}^n$, such that if $Q_k=B(x_k,\gamma (x_k))$, $k\in
\mathbb{N}$, the following properties hold:

(i) $\cup_{k=1}^\infty Q_k=\mathbb{R}^n$;

(ii) For every $m\in \mathbb{N}$ there exist $C,\beta >0$ such that, for every $k\in \mathbb{N}$,
$$
\mbox{card }\{l\in \mathbb{N}:2^mQ_l\cap 2^mQ_k\not=\varnothing\}\leq C2^{m\beta }.
$$
The sequence $\{Q_k\}_{k\in \mathbb{N}}$ of balls will appear in
different occasions throughout this paper.

\begin{Lema} \label{add1} Let $b\in BMO_\theta(\gamma)$, $\theta\ge 0$, and $M>0$. We define the
operator $T_{b,M}$ by
$$
T_{b,M}(f)(x)=\frac{1}{\gamma(x)^n}\int_{|x-y|\le
M\gamma(x)}|b(x)-b(y)|f(y)dy,\,\,\,x\in \mathbb{R}^n.
$$
Then, for every $1<p<\infty$, $T_{b,M}$ is bounded from
$L^p(\mathbb{R}^n)$ into itself. Moreover,
$$
\|T_{b,M}(f)\|_{L^p(\mathbb{R}^n)}\le
C(1+MA)^{n+\theta'}\|b\|_{BMO_\theta(\gamma)}\|f\|_{L^p(\mathbb{R}^n)},\,\,\,f\in
L^p(\mathbb{R}^n).
$$
Here, $A,C>0$ and $\theta'>\theta$ do not depend on $M$.
\end{Lema}

\begin{proof} Let $1<p<\infty$ and $f\in L^p(\mathbb{R}^n)$ . We have that
$$
\|T_{b,M}(f)\|_{L^p(\mathbb{R}^n)}\le C\Big(\sum_{k=1}^\infty
\int_{Q_k}|T_{b,M}(f)(x)|^pdx\Big)^{1/p}.
$$
Since $\gamma(x)\sim \gamma(x_k)$, for every $x\in Q_k$, $k\in
\mathbb{N}$, by using H\"older and Minkowski inequalities and
\cite[Proposition 3]{BHS3}, it follows that, for certain $A>0$ and
$\theta'>\theta$,
\begin{eqnarray*}
&&\int_{Q_k}|T_{b,M}(f)(x)|^pdx\le
C\int_{Q_k}\Big(\frac{1}{\gamma(x_k)^n}\int_{|x_k-y|\le
(1+MA)\gamma(x_k)}|b(x)-b(y)||f(y)|dy\Big)^pdx\\
&\le&C\int_{Q_k}\frac{1}{\gamma(x_k)^{np}}\int_{|x_k-y|\le
(1+MA)\gamma(x_k)}|f(y)|^pdy\Big(\int_{|x_k-y|\le
(1+MA)\gamma(x_k)}|b(x)-b(y)|^{p'}dy\Big)^{p/p'}dx\\
&\le&C\int_{Q_k}\frac{1}{\gamma(x_k)^{np}}\int_{|x_k-y|\le
(1+MA)\gamma(x_k)}|f(y)|^pdy\\
&\times&\Big(\Big(\int_{|x_k-y|\le
(1+MA)\gamma(x_k)}|b(x)-b_{\widehat{Q_k}}|^{p'}dy\Big)^{p/p'}+\Big(\int_{|x_k-y|\le
(1+MA)\gamma(x_k)}|b(y)-b_{\widehat{Q_k}}|^{p'}dy\Big)^{p/p'}\Big)dx\\
&\le&C\int_{Q_k}\frac{1}{\gamma(x_k)^{np}}\int_{|x_k-y|\le
(1+MA)\gamma(x_k)}|f(y)|^pdy\\
&\times&
\Big(|b(x)-b_{\widehat{Q_k}}|^p+\|b\|^p_{BMO_\theta(\gamma)}(2+MA)^{p\theta'}\Big)((1+MA)\gamma(x_k))^{np/p'}dx\\
&\le& C\int_{|x_k-y|\le (1+MA)\gamma(x_k)}|f(y)|^pdy
\|b\|^p_{BMO_\theta(\gamma)}(1+MA)^{(n+\theta')p},    \,\,\,k\in
\mathbb{N},
\end{eqnarray*}
where $\widehat{Q_k}=B(x_k,(1+MA)\gamma(x_k))$, $k\in \mathbb{N}$.

Hence, by taking into account the properties of $\{Q_k\}_{k\in
\mathbb{N}}$, we get
\begin{eqnarray*}
\|T_{b,M}(f)\|_{L^p(\mathbb{R}^n)}&\le&
C\Big(\sum_{k=1}^\infty\int_{|x_k-y|\le (1+MA)\gamma(x_k)}|f(y)|^pdy
\|b\|^p_{BMO_\theta(\gamma)}(1+MA)^{(n+\theta')p}\Big)^{1/p}\\
&\le&C(1+MA)^{n+\theta'}\|b\|_{BMO_\theta(\gamma)}\|f\|_{L^p(\mathbb{R}^n)}.
\end{eqnarray*}
\end{proof}

Our next step is to advance on the boundedness of the corresponding
local operators. The case of heat semigroup and Riesz transforms are
not very difficult and the will be done while proving Theorems
\ref{LVarLp} and \ref{VarRiesz} in Sections 3 and 4, respectively.
As far the commutator, it involves an additional difficulty.
According to our procedure we start with the commutator $C_{b,\ell}$
with $b\in BMO_\theta(\gamma)$ and we reduce the problem to the
local classical commutator but now $b$ is not necessarily in $BMO$
in the classical sense. In the next lemma and proposition we show
how to overcome this problem.

\begin{Lema}\label{add2} Let $\theta\ge 0$ and $L>0$. There exists $C>0$ such
that, for every $k\in \mathbb{N}$ and $b\in BMO_\theta(\gamma)$, we
can find a function $\mathfrak{b}_k\in BMO(\mathbb{R}^n)$ for which
$\mathfrak{b}_k=b$ on $LQ_k=B(x_k,L\gamma(x_k))$ and
$\|\mathfrak{b}_k\|_{BMO(\mathbb{R}^n)}\le
C\|b\|_{BMO_\theta(\gamma)}$.
\end{Lema}

\begin{proof} Let $k\in \mathbb{N}$ and $b\in BMO_\theta(\gamma)$.
We define $b_k=b\chi_{LQ_k}$. Assume that $z_0\in \mathbb{R}^n$ and
$r_0>0$ such that $B(z_0,r_0)\subset LQ_k$. It is clear that
$|z_0-x_k|\le L\gamma(x_k)$. Then, for a certain $A>0$ that does not
depend on $k$, $\gamma(x_k)\le A\gamma(z_0)$.

We have that
\begin{eqnarray*}
\frac{1}{|B(z_0,r_0)|}\int_{B(z_0,r_0)}|b(x)-b_{B(z_0,r_0)}|dx&\le&
\|b\|_{BMO_\theta(\gamma)}\Big(1+\frac{r_0}{\gamma(z_0)}\Big)^\theta\\
&\le&C\|b\|_{BMO_\theta(\gamma)},
\end{eqnarray*}
where $C>0$ depends on $L$ and $\theta$ but does not depend on $k$.
Hence, $b\in BMO(LQ_k)$ and $\|b\|_{BMO(LQ_k)}\le
C\|b\|_{BMO_\theta(\gamma)}$.

We define $w_k=\rm{exp}(b_k)$. It is well known that $w_k\in
A_2(LQ_k)$, where $A_2(LQ_k)$ denotes the Muckenhoupt class of
weights. Moreover, the $A_2$- characteristic $[w_k]_{A_2(LQ_k)}$
satisfies that $ [w_k]_{A_2(LQ_k)}\le C\|b\|_{BMO(LQ_k)}, $ where
$C>0$ does not depend on $k$. According to \cite[Lemma 1]{BHS4} (see
also \cite{Har}) there exists $\widetilde{w_k}\in A_2(\mathbb{R}^n)$
such that $\widetilde{w_k}=w_k$ in $LQ_k$, and $
[\widetilde{w}_k]_{A_2(\mathbb{R}^n)}\le C [w_k]_{A_2(LQ_k)}, $
$C>0$ being independent of $k$. Then, there exists
$\mathfrak{b}_k\in BMO(\mathbb{R}^n)$ satisfying that
$\widetilde{w}_k=\rm{exp}(\mathfrak{b}_k)$ and
$\|\mathfrak{b}_k\|_{BMO(\mathbb{R}^n)}\le
C[\widetilde{w}_k]_{A_2(\mathbb{R}^n)}$, where $C>0$ does not depend
on $k$.

We conclude that $\mathfrak{b}_k=b$ on $LQ_k$ and
$\|\mathfrak{b}_k\|_{BMO(\mathbb{R}^n)}\le C
\|b\|_{BMO_\theta(\gamma)}$, for a certain $C>0$ independent of $k$.
\end{proof}

\begin{Prop}\label{master} Let $b\in BMO_\infty(\gamma)$ and $\ell=1,\cdots,n$. Then, the maximal operator $C_{b,\ell}^{*,\rm{loc}}$
is bounded from $L^p(\mathbb{R}^n)$ into itself, for every $1\!<\!p\!<\!\infty$.
\end{Prop}

\begin{proof} %We choose a sequence $\{x_k\}_{k=1}^\infty \subset \mathbb{R}^n$, such that if $Q_k=B(x_k,\gamma (x_k))$, $k\in \mathbb{N}$, the following properties hold (see \cite[Proposition 2.14]{DGMTZ}):
%
%(i) $\cup_{k=1}^\infty Q_k=\mathbb{R}^n$;
%
%(ii) For every $m\in \mathbb{N}$ there exist $C,\beta >0$ such that, for every $k\in \mathbb{N}$,
%$$
%\mbox{card }\{l\in \mathbb{N}:2^mQ_l\cap 2^mQ_k\not=\varnothing\}\leq C2^{m\beta }.
%$$

Assume that $b\in BMO_\theta(\gamma)$, with $\theta\ge 0$. Fix $k\in
\mathbb{N}$. According to \cite[Lemma 1.4]{Sh}
$\gamma(x)\sim\gamma(x_k)$, for every $x\in Q_k$. We choose $L>0$ independent of k
such that, for any $x\in Q_k$, $B(x,\gamma(x))\subset \mathbb{Q}_k$, where
$\mathbb{Q}_k=B(x_k,L\gamma(x_k))$. We can write
\begin{eqnarray}\label{VP1}
&&\Big|\int_{\varepsilon<|x-y|<\gamma(x)}(b(x)-b(y))R_\ell(x-y)f(y)dy\Big|\nonumber\\
&\le&\Big|\int_{\varepsilon<|x-y|<\gamma(x)}(b(x)-b(y))R_\ell(x-y)f(y)dy-\int_{y\in \mathbb{Q}_k,|x-y|>\varepsilon}(b(x)-b(y))R_\ell(x-y)f(y)dy\Big|\nonumber\\
&+&\Big|\int_{y\in \mathbb{Q}_k,|x-y|>\varepsilon}(b(x)-b(y))R_\ell(x-y)f(y)dy\Big|\nonumber\\
&&\hspace{10mm}=J_{1,k}(x,\varepsilon)+J_{2,k}(x,\varepsilon),\,\,\,x\in
Q_k,\,\,\,\varepsilon>0.
\end{eqnarray}

Let us analyze $J_{i,k}(x,\varepsilon)$, for $x\in Q_k$ and $\varepsilon>0$, $i=1,2$. Observe that, since
$\gamma(y)\sim \gamma(x_k)$, for every $y\in \mathbb{Q}_k$, we can
find $C,M>0$ that do not depend on $k\in \mathbb{N}$ such that
\begin{eqnarray*}
J_{1,k}(x,\varepsilon)&\le& C\int_{\mathbb{Q}_k\setminus
B(x,\gamma(x))}\frac{|b(x)-b(y)|}{|x-y|^n}|f(y)|dy\\
&\le& C\int_{\gamma(x)/M\le|x-y|\le
M\gamma(x)}\frac{|b(x)-b(y)|}{|x-y|^n}|f(y)|\chi_{\mathbb{Q}_k}(y)dy\\
&\le&\frac{C}{\gamma(x)^n}\int_{|y-x|\le
M\gamma(x)}|b(x)-b(y)||f(y)|\chi_{\mathbb{Q}_k}(y)dy,\,\,\,x\in
Q_k\,\,\,\rm{and}\,\,\,\varepsilon>0.
\end{eqnarray*}

From Lemma \ref{add1} we deduce that,
\begin{eqnarray}\label{VP2}
\int_{Q_k}\Big|\sup_{\varepsilon>0}
J_{1,k}(x,\varepsilon)\Big|^pdx&\le&C\int_{\mathbb{R}^n}\Big|T_{b,M}(f\chi_{\mathbb{Q}_k})(x)\Big|^pdx\nonumber\\
&\le&C\|b\|^p_{BMO_\theta(\gamma)}\int_{\mathbb{Q}_k}|f(y)|^pdy,\,\,\,k\in
\mathbb{N},
\end{eqnarray}
with a constant independent of $k$.

On the other hand, we have that, for $x\in Q_k$ and $\varepsilon>0$,
$$
J_{2,k}(x,\varepsilon)=\Big|\int_{|x-y|>\varepsilon}(b_k(x)-b_k(y))R_\ell(x-y)f(y)dy\Big|,
$$
where $b_k=b\chi_{\mathbb{Q}_k}$. According to Lemma \ref{add2},
there exists a function $\widetilde{b}_k\in BMO(\mathbb{R}^n)$ such
that $\widetilde{b}_k=b$, on $\mathbb{Q}_k$, and
$\|\widetilde{b}_k\|_{BMO(\mathbb{R}^n)}\le
C\|b\|_{BMO_\theta(\gamma)}$, where $C>0$ does not depend on $k\in
\mathbb{N}$. It follows that
$$
\sup_{\varepsilon>0} J_{2,k}(x,\varepsilon)=
C_{\widetilde{b}_k,\ell}^*(f\chi_{\mathbb{Q}_k})(x),\,\,\,x\in Q_k.
$$

Hence, we get
\begin{eqnarray}\label{VP3}
\int_{Q_k}\Big|\sup_{\varepsilon>0}
J_{2,k}(x,\varepsilon)\Big|^pdx&\le&C\int_{\mathbb{R}^n}\Big|C_{\widetilde{b}_k,\ell}^*(f\chi_{\mathbb{Q}_k}(x)\Big|^pdx\nonumber\\
&\le&C\|\widetilde{b}_k\|^p_{BMO(\mathbb{R}^n)}\int_{\mathbb{Q}_k}|f(y)|^pdy\nonumber\\
&\le&
C\|b\|^p_{BMO_\theta(\gamma)}\int_{\mathbb{Q}_k}|f(y)|^pdy,\,\,\,k\in
\mathbb{N}.
\end{eqnarray}

By combining (\ref{VP1}), (\ref{VP2}) and (\ref{VP3}) and using the
properties of the sequence $\{Q_k\}_{k\in \mathbb{N}}$ we conclude,
for every $1<p<\infty$,
\begin{eqnarray*}
\|C_{b,\ell}^{*,\rm{loc}}(f)\|_{L^p(\mathbb{R}^n)}&\le&\Big(\sum_{k=1}^\infty
\int_{Q_k}|C_{b,\ell}^{*,\rm{loc}}(f)(x)|^pdx\Big)^{1/p}\\
&\le& C\|b\|_{BMO_\theta(\gamma)}\Big(\sum_{k=1}^\infty
\int_{\mathbb{Q}_k}|f(y)|^pdy\Big)^{1/p}\\
&\le&C\|b\|_{BMO_\theta(\gamma)}\|f\|_{L^p(\mathbb{R}^n)},
\,\,\,f\in L^p(\mathbb{R}^n).
\end{eqnarray*}
\end{proof}

Suppose now that $\ell=1,\cdots,n$, $f\in C_c^\infty(\mathbb{R}^n)$
and $b\in BMO_\infty(\gamma)$. Then, $bf\in
L^1(\mathbb{R}^n)$. Moreover, we can write, for
$\varepsilon>0$ small enough,
$$
C_{b,\ell}^{\varepsilon,\rm{loc}}(f)(x)=C_{b,\ell}^{\varepsilon}(f)(x)-\int_{|x-y|\ge\gamma(x)}(b(x)-b(y))R_\ell(x-y)f(y)dy,\,\,\,
x\in \mathbb{R}^n.
$$
Hence, there exists the limit
$$
\lim_{\varepsilon\to
0^+}C_{b,\ell}^{\varepsilon,\rm{loc}}(f)(x),\,\,\,x\in \mathbb{R}^n.
$$

By using standard arguments, from Proposition \ref{add1} we deduce
that, for every $f\in L^p(\mathbb{R}^n)$, $1<p<\infty$, there exists
the limit
$$
\lim_{\varepsilon\to
0^+}C_{b,\ell}^{\varepsilon,\rm{loc}}(f)(x),\,\,\,a.e.\,\,\,x\in
\mathbb{R}^n.
$$
We define, for every $f\in L^p(\mathbb{R}^n)$, $1<p<\infty$,
$$
C_{b,\ell}^{\rm{loc}}(f)(x)=\lim_{\varepsilon\to
0^+}C_{b,\ell}^{\varepsilon,\rm{loc}}(f)(x),\,\,\,a.e.\,\,\,x\in
\mathbb{R}^n.
$$

The behavior in $L^p(\mathbb{R}^n)$ of the variation operators for
the family of truncations $\{C_{b,\ell}^{\varepsilon,
\rm{loc}}\}_{\varepsilon>0}$ associated with the local commutator
operator  $C_{b,\ell}^{\rm{loc}}$ are established in the following.

\begin{Th}\label{loccomm} Let $b\in BMO_\infty(\gamma)$ and
$\ell=1,\cdots,n$. Assume that $\rho>2$ . Then, the variation
operator $V_\rho(C_{b,\ell}^{\varepsilon, \rm{loc}})$ is bounded
from $L^p(\mathbb{R}^n)$ into itself, for every $1<p<\infty$.
\end{Th}

\begin{proof} Suppose that $f\in L^p(\mathbb{R}^n)$ with $1<p<\infty$. Let $k\in
\mathbb{N}$. As in the proof of Proposition \ref{master} we define
$\mathbb{Q}_k=B(x_k,L\gamma(x_k))$, where $L>0$ is such that
$B(x,\gamma(x))\subset \mathbb{Q}_k$, for every $x\in Q_k$. Morever,
$L$ does not depend on $k$.

For every $x\in Q_k$, we can write
\begin{eqnarray*}
&&V_\rho(C_{b,\ell}^{\varepsilon,
\rm{loc}})(f)(x)=\sup_{\{\varepsilon_j\}_{j\in \mathbb{N}}\downarrow
0}\Big(\sum_{j=1}^\infty \Big|\int_{\varepsilon_{j+1}\le
|x-y|<\varepsilon_j;|x-y|<\gamma(x)}(b(x)-b(y))R_\ell(x-y)f(y)dy\Big|^\rho\Big)^{1/\rho}\\
&\le&\sup_{\{\varepsilon_j\}_{j\in \mathbb{N}}\downarrow
0}\Big(\sum_{j=1}^\infty \Big|\int_{\varepsilon_{j+1}\le
|x-y|<\varepsilon_j;|x-y|<\gamma(x)}(b(x)-b(y))R_\ell(x-y)f(y)dy\\
&-&\int_{\varepsilon_{j+1}\le |x-y|<\varepsilon_j;y\in
\mathbb{Q}_k}(b(x)-b(y))R_\ell(x-y)f(y)dy\Big|^\rho\Big)^{1/\rho}\\
&+&\sup_{\{\varepsilon_j\}_{j\in \mathbb{N}}\downarrow
0}\Big(\sum_{j=1}^\infty \Big|\int_{\varepsilon_{j+1}\le
|x-y|<\varepsilon_j;y\in
\mathbb{Q}_k}(b(x)-b(y))R_\ell(x-y)f(y)dy\Big|^\rho\Big)^{1/\rho}\\
&\le&C\Big(\sup_{\{\varepsilon_j\}_{j\in \mathbb{N}}\downarrow
0}\Big(\sum_{j=1}^\infty \Big|\int_{\varepsilon_{j+1}\le
|x-y|<\varepsilon_j;y\in
\mathbb{Q}_k\setminus B(x,\gamma(x))}(b(x)-b(y))R_\ell(x-y)f(y)dy\Big|^\rho\Big)^{1/\rho}\\
&+&\sup_{\{\varepsilon_j\}_{j\in \mathbb{N}}\downarrow
0}\Big(\sum_{j=1}^\infty \Big|\int_{\varepsilon_{j+1}\le
|x-y|<\varepsilon_j}(b(x)-b(y))R_\ell(x-y)f(y)\chi_{\mathbb{Q}_k}(y)dy\Big|^\rho\Big)^{1/\rho}\Big)\\
&\le&C\Big(\sup_{\{\varepsilon_j\}_{j\in \mathbb{N}}\downarrow
0}\Big(\sum_{j=1}^\infty \int_{\varepsilon_{j+1}\le
|x-y|<\varepsilon_j;y\in
\mathbb{Q}_k\setminus B(x,\gamma(x))}|b(x)-b(y)|\frac{1}{|x-y|^{n}}|f(y)|dy\\
&+&\sup_{\{\varepsilon_j\}_{j\in \mathbb{N}}\downarrow
0}\Big(\sum_{j=1}^\infty \Big|\int_{\varepsilon_{j+1}\le
|x-y|<\varepsilon_j;}(\widetilde{b}_k(x)-\widetilde{b}_k(y))\frac{x_\ell-y_\ell}{|x-y|^{n+1}}f(y)\chi_{\mathbb{Q}_k}(y)dy\Big|^\rho\Big)^{1/\rho}\Big)\\
&\le&C \Big(\int_{\mathbb{Q}_k\setminus
B(x,\gamma(x))}|b(x)-b(y)|\frac{1}{|x-y|^{n}}|f(y)|dy+V_\rho(C_{\widetilde{b}_k,\ell}(f\chi_{\mathbb{Q}_k})(x)\Big),
\end{eqnarray*}
where $\widetilde{b}_k\in BMO(\mathbb{R}^n)$ satisfying that
$\widetilde{b}_k=b$, in $\mathbb{Q}_k$, and
$\|\mathfrak{b}_k\|_{BMO(\mathbb{R}^n)}\le
C\|b\|_{BMO_\theta(\gamma)}$, being $b\in BMO_\theta(\gamma)$. Here
$C>0$ is independent on $k$ (see Lemma \ref{add2}). Then, by using
Lemma \ref{add1} (as in the proof of Proposition \ref{master}) and
Theorem \ref{CVarBMO}, we conclude that
$$
\int_{Q_k}\Big|V_\rho(C_{b,\ell}^{\varepsilon,
\rm{loc}})(f)(x)\Big|^pdx\le
C\|b\|_{BMO_\theta(\gamma)}^p\int_{\mathbb{Q}_k}|f(x)|^pdx,
$$
where again $C>0$ does not depend on $k$.

Hence, according to the properties of the sequence $\{Q_k\}_{k\in
\mathbb{N}}$, we get
$$
\|V_\rho(C_{b,\ell}^{\varepsilon, \rm{loc}})(f)\|_{L^p(
\mathbb{R}^n)}\le C\|b\|_{BMO_\theta(\gamma)}\|f\|_{L^p(
\mathbb{R}^n)}.
$$
Thus the proof is finished.

\end{proof}

\section{Variation operators associated with the heat semigroup $\{W_t^\mathcal{L}\}_{t>0}$.}

In this section we present a proof of Theorem \ref{LVarLp}.
$L^p$-boundedness properties of the variation operators can be
showed in a similar way.

As the general procedure suggests we consider the following local operators
$$
W^{\mathcal{L}}_{t,{\rm loc}}(f)(x)= \int_{|x-y| < \gamma(x)} W^{\mathcal{L}}_t(x,y) f(y)dy,\;\; x \in \mathbb R^n,
$$
and
$$
W_{t,{\rm loc}}(f)(x)= \int_{|x-y| < \gamma(x)} W_t (x,y)f(y)dy,\;\; x \in \mathbb R^n,
$$
where $f \in L^p(\mathbb R^n)$, $1 \le p < \infty$.

Observe that
\begin{equation}\label{rel1}
V_\rho(W_t^{\mathcal{L}})(f) \leq V_\rho(W_{t,{\rm loc}}^{\mathcal{L}}-W_{t,{\rm loc}})(f) + V_\rho(W_{t,{\rm loc}})(f) + V_\rho(W_t^{\mathcal{L}}-W_{t,{\rm loc}}^{\mathcal{L}})(f).
\end{equation}

Assume that $\{t_j\}_{j\in \mathbb{N}}$ is a real decreasing sequence that converges to zero. We can write
\begin{multline*}
\left(\sum_{j=0}^\infty \left|W_{t_j,{\rm loc}}(f)(x) - W_{t_{j+1},{\rm loc}}(f)(x)\right|^\rho\right)^{\frac{1}{\rho}} \\
\shoveleft{\hspace{1.8cm}\leq\left(\sum_{j=0}^\infty \left|W_{t_j}(f)(x) - W_{t_{j+1}}(f)(x)\right|^\rho\right)^{\frac{1}{\rho}} }\\
\shoveleft{\hspace{2.1cm}+\left(\sum_{j=0}^\infty \left|\int_{|x-y|> \gamma(x)}\left(W_{t_j}(x,y) - W_{t_{j+1}}(x,y)\right)f(y)dy\right|^\rho\right)^{\frac{1}{\rho}}} \\
\leq V_\rho(W_t)(f)(x) + \sup_{\varepsilon > 0} \left\|\int_{|x-y|> \varepsilon}W_t(x,y)f(y) dy\right\|_{E_\rho},\quad x\in \mathbb{R}^n,\hspace{25mm}
\end{multline*}
where the space $E_\rho$ is defined in Section 1.

Then,
$$
V_\rho(W_{t,{\rm loc}})(f)(x) \leq V_\rho(W_t)(f) + \sup_{\varepsilon > 0}\left\|\int_{|x-y|> \varepsilon} W_t(x,y)f(y)dy \right\|_{E_\rho}, \quad x\in \mathbb{R}^n.
$$
We consider the operator defined by
$$
\begin{array}{rccl}
T:& L^2(\mathbb R^n) & \rightarrow & L^2_{E_\rho}(\mathbb R^n) \\
  & f  & \longrightarrow & Tf(x)=\displaystyle\int_{\mathbb R^n}W_t(x,y)f(y) dy.
  \end{array}
$$
According to \cite[Theorem 3.3]{JR} $T$ is bounded from $L^2(\mathbb R^n)$ into $L^2_{E_\rho}(\mathbb R^n)$. Moreover, $T$ is a Calder\'on-Zygmund operator associated with the $E_\rho$-valued kernel
$$
K(x,y;t)= W_t(x,y), \;\;x,y \in \mathbb R^n,\;\;t >0,
$$
that satisfies the following properties (see \cite{CMMTV}):
\begin{enumerate}
\item $ \displaystyle\|K(x,y;\cdot)\|_{E_\rho} \leq \frac{C}{|x-y|^n}$, $x,y \in \mathbb R^n$, $x \neq y$,
\item $ \displaystyle\left\|\frac{\partial}{\partial x} K(x,y;\cdot)\right\|_{E_\rho} + \left\|\frac{\partial}{\partial y} K(x,y;\cdot)\right\|_{E_\rho} \leq \frac{C}{|x-y|^{n+1}}$, $x,y \in \mathbb R^n$, $x \neq y$.
\end{enumerate}
Then, by proceeding as in the proof of \cite[Proposition 2, p. 34 and Corollary 2, p. 36]{Stein1}, we prove that the maximal operator $T^*$ defined by
$$
T^*f(x)= \sup_{\varepsilon > 0} \left\|\int_{|x-y|>\varepsilon} W_t(x,y) f(y) dy\right\|_{E_\rho}
$$
is bounded from $L^p(\mathbb R^n)$ into itself, for every $1<p<\infty$, and from $L^1(\mathbb R^n)$ into $L^{1,\infty}(\mathbb R^n)$. By combining this fact with \cite[Theorem 3.3]{JR} we conclude that the operator $V_\rho(W_{t,{\rm loc}})$ is bounded from $L^p(\mathbb R^n)$ into itself, for every $1<p<\infty$, and from $L^1(\mathbb R^n)$ into $L^{1,\infty}(\mathbb R^n)$.

We consider now the variation operator defined by
$$
V_\rho\left(W^{\mathcal{L}}_t - W^{\mathcal{L}}_{t,{\rm loc}}\right)(f)(x) = \sup_{\{t_j\}_{j\in \mathbb{N}} \searrow 0} \left(\sum_{j=0}^\infty \left|\int_{|x-y|>\gamma(x)}
\left(W^{\mathcal{L}}_{t_j}(x,y) - W^{\mathcal{L}}_{t_{j+1}}(x,y)\right) f(y) dy \right|^\rho \right)^{\frac{1}{\rho}}.
$$
Assume that $\{t_j\}_{j\in \mathbb{N}}$ is a real decreasing sequence that converges to zero. We can write
\begin{multline}\label{D1}
\left(\sum_{j=0}^\infty \left|\int_{|x-y|>\gamma(x)} \left( W^{\mathcal{L}}_{t_j}(x,y) - W^{\mathcal{L}}_{t_{j+1}}(x,y)\right)f(y)dy\right|^\rho\right)^{\frac{1}{\rho}} \\
\shoveleft{\hspace{20mm}\leq \sum^\infty_{j=0}\int_{|x-y|>\gamma(x)}|f(y)|\int_{t_{j+1}}^{t_j} \left|\frac{\partial}{\partial t} W^{\mathcal{L}}_{t}(x,y)\right|dt dy}\\
= \int_{|x-y|>\gamma(x)}|f(y)|\int_0^\infty \left|\frac{\partial}{\partial t} W^{\mathcal{L}}_{t}(x,y)\right|dt dy, \;\; x \in \mathbb R^n.\hspace{31mm}
\end{multline}
According to \cite[(2.7)]{DGMTZ}, for every $N \in \mathbb N$ there exist $c,C > 0$ such that
\begin{equation}\label{D2}
\left|\frac{\partial}{\partial t} W^{\mathcal{L}}_{t}(x,y)\right| \leq \frac{C}{t^{\frac{n}{2}+1}}\left(1+\frac{t}{\gamma(x)^2}+\frac{t}{\gamma(y)^2}\right)^{-N}e^{\frac{-c|x-y|^2}{t}}, \;\; x,y\in \mathbb R^n, \;\;t>0.
\end{equation}
Estimation (\ref{D2}) allows us to obtain
\begin{multline}\label{D3}
\int_{|x-y|>\gamma(x)}\int_{\gamma(x)^2}^\infty \left|\frac{\partial}{\partial t} W^{\mathcal{L}}_{t}(x,y)\right|dt|f(y)|dy \\
\shoveleft{\hspace{28mm}\leq C\int_{|x-y|>\gamma(x)}|f(y)|\int_{\gamma(x)^2}^\infty \frac{1}{t^{\frac{n}{2}+1}}\left(1+\frac{t}{\gamma(x)^2}\right)^{-n-2}e^{-\frac{c|x-y|^2}{t}}dt dy}\\
\shoveleft{\hspace{28mm}\leq \frac{C}{\gamma(x)^n}\int_{|x-y|>\gamma(x)}|f(y)|\int_1^\infty \frac{1}{u^{\frac{n}{2}+1}}\frac{1}{(1+u)^{n+2}}e^{-\frac{c|x-y|^2}{u\gamma(x)^2}}du  dy} \\
\shoveleft{\hspace{28mm}\leq  \frac{C}{\gamma(x)^n}\int_{|x-y|>\gamma(x)}|f(y)|\int_1^\infty \frac{1}{u^{\frac{n}{2}+1}}\frac{1}{(1+u)^{2+n}}\left(\frac{u\gamma(x)^2}{|x-y|^2}\right)^{\frac{n}{2}+1}dudy} \\
\shoveleft{\hspace{28mm}\leq \frac{C}{\gamma(x)^n}\int_{|x-y|>\gamma(x)}|f(y)|\left(\frac{\gamma(x)}{|x-y|}\right)^{n+2}dy} \\
\shoveleft{\hspace{28mm}= \frac{C}{\gamma(x)^n}\sum_{k=0}^\infty \int_{2^k\gamma(x) < |x-y| \leq 2^{k+1}\gamma(x)}|f(y)|\left(\frac{\gamma(x)}{|x-y|}\right)^{n+2}dy} \\
 \leq C\sum_{k=0}^\infty \frac{1}{2^{2k}(2^k\gamma(x))^n}\int_{|x-y| \leq 2^{k+1}\gamma(x)}|f(y)|dy \leq C\mathcal{M}(f)(x), \;\;x \in \mathbb R^n.
\end{multline}
Here $\mathcal{M}=\mathcal{M}_1$ denotes the Hardy-Littlewood maximal function.

Also we have that
\begin{multline}\label{D4}
\int_{|x-y|> \gamma(x)}\int_0^{\gamma(x)^2}\left|\frac{\partial}{\partial t}W_t^\mathcal{L}(x,y)\right|dt |f(y)|dy \\
\shoveleft{\hspace{45mm}\leq C\int_0^{\gamma(x)^2}\int_{|x-y|>\gamma(x)}\frac{1}{t^{\frac{n}{2}+1}}e^{-c\frac{|x-y|^2}{t}}|f(y)|dydt}\\
\shoveleft{\hspace{45mm}\leq C \int_0^{\gamma(x)^2}\frac{e^{-c\frac{\gamma(x)^2}{t}}}{t}\int_{\mathbb R^n}\frac{1}{t^{\frac{n}{2}}}e^{-c\frac{|x-y|^2}{t}}|f(y)|dydt}\\
\leq C \sup_{t>0}\frac{1}{t^{\frac{n}{2}}}\int_{\mathbb R^n} e^{-c\frac{|x-y|^2}{t}}|f(y)|dy \leq C \mathcal{M}(f)(x),\;\; x \in \mathbb R^n.
\end{multline}
From (\ref{D1}), (\ref{D3}) and (\ref{D4}) we conclude that
$$
V_\rho\left(W_t^{\mathcal{L}}-W_{t,{\rm loc}}^{\mathcal{L}}\right)(f)(x) \leq C\mathcal{M}(f)(x),\;\; x \in \mathbb R^n.
$$
Hence, the operator $V_\rho\left(W_t^{\mathcal{L}}-W_{t,{\rm loc}}^{\mathcal{L}}\right)$ is bounded from $L^p(\mathbb R^n)$ into itself, for every $1<p<\infty$, and from $L^1(\mathbb R^n)$ into $L^{1,\infty}(\mathbb R^n)$.
We now analyze the operator
\begin{multline*}
V_\rho\left(W_{t,{\rm loc}}^{\mathcal{L}}-W_{t,{\rm loc}}\right)(f)(x)= \sup_{\{t_j\}_{j\in \mathbb{N}} \searrow 0} \left(\sum_{j=0}^\infty \left|\int_{|x-y| < \gamma(x)} \left(W_{t_j}^{\mathcal{L}}(x,y)-W_{t_j}(x,y)\right.\right.\right.\\
\left.\left.\left.- \left(W_{t_{j+1}}^{\mathcal{L}}(x,y) - W_{t_{j+1}}(x,y)\right)\right)f(y) dy \right|^\rho\right)^{\frac{1}{\rho}},\quad x\in \mathbb{R}^n.
\end{multline*}
Let us take a real decreasing sequence $\{t_j\}_{j\in \mathbb{N}}$ that converges to zero. We have that
\begin{multline}\label{D5}
\left(\sum_{j=0}^\infty \left|\int_{|x-y| < \gamma(x)} \left(W_{t_j}^{\mathcal{L}}(x,y)-W_{t_j}(x,y)- \left(W_{t_{j+1}}^{\mathcal{L}}(x,y) -
W_{t_{j+1}}(x,y)\right)\right)f(y) dy \right|^\rho\right)^{\frac{1}{\rho}} \\
\shoveleft{\hspace{30mm}= \left( \sum_{j=0}^\infty \left|\int_{|x-y| < \gamma(x)}f(y) \int_{t_{j+1}}^{t_j}\frac{\partial}{\partial t}\left(W_t^{\mathcal{L}}(x,y)- W_t(x,y) \right)dt dy \right|^\rho\right)^{\frac{1}{\rho}}} \\
\shoveleft{\hspace{30mm}\leq \sum_{j=0}^\infty \int_{|x-y|<\gamma(x)}|f(y)|\int_{t_{j+1}}^{t_j}\left|\frac{\partial}{\partial t}\left(W_t^{\mathcal{L}}(x,y)- W_t(x,y)\right)\right|dtdy} \\
\shoveleft{\hspace{30mm}= \int_{|x-y|<\gamma(x)}|f(y)|\int_0^{\gamma(x)^2}\left|\frac{\partial}{\partial t}\left(W_t^{\mathcal{L}}(x,y) - W_t(x,y)\right)\right|dtdy}\\
\shoveleft{\hspace{34mm}+\int_{|x-y| < \gamma(x)}|f(y)|\int_{\gamma(x)^2}^\infty \left|\frac{\partial}{\partial t}\left(W_t^{\mathcal{L}}(x,y) - W_t(x,y)\right)\right|dt dy}  \\
 = I_1(f)(x) + I_2(f)(x), \;\;\;\;\; x \in \mathbb R^n. \hspace{51mm}
\end{multline}
Note firstly that, according to (\ref{D2}) and the known estimates for $\displaystyle \frac{\partial}{\partial t}W_t$, for certain $C,c >0$, we get
\begin{eqnarray}\label{D6}
I_2(f)(x) &\leq & C\int_{|x-y| < \gamma(x)} |f(y)|\int_{\gamma(x)^2}^\infty \frac{e^{-c\frac{|x-y|^2}{t}}}{t^{\frac{n}{2}+1}}dtdy \nonumber\\
&\leq & C\int_{|x-y|<\gamma(x)}|f(y)|\int_{\gamma(x)^2}^\infty \frac{dt}{t^{\frac{n}{2}+1}}dy \nonumber \\
&\leq & \frac{C}{\gamma(x)^n}\int_{|x-y|<\gamma(x)}|f(y)|dy \leq C\mathcal{M}(f)(x), \;\;\; x \in \mathbb R^n.
\end{eqnarray}
On the other hand, the perturbation formula (\cite[(5.25)]{DGMTZ}) leads to
\begin{eqnarray*}
\frac{\partial}{\partial t} \left(W_t(x,y)-W_t^{\mathcal{L}}(x,y)\right) &=& \int_{\mathbb R^n}V(z) W_{\frac{t}{2}}(x,z) W_{\frac{t}{2}}^{\mathcal{L}}(z,y)dz \\
&& + \int_0^{\frac{t}{2}} \int_{\mathbb R^n} V(z) \frac{\partial}{\partial t}W_{t-s}(x,z)W_s^{\mathcal{L}}(z,y) dzds \\
&& + \int_{\frac{t}{2}}^t\int_{\mathbb R^n}W_{t-s}(x,z)V(z) \frac{\partial}{\partial s}W_s^{\mathcal{L}}(z,y) dzds \\
&=& K_1 (x,y,t) + K_2(x,y,t) + K_3(x,y,t), \;\; x,y \in \mathbb R^n \;\; \mbox{and} \;\; t>0.
\end{eqnarray*}
Hence $I_1(f) = T_1(f) + T_2(f)+T_3(f)$, where, for $m=1,2,3$,
$$
T_mf(x)= \int_{|x-y|< \gamma(x)} |f(y)|\int_0^{\gamma(x)^2}|K_m(x,y,t)|dtdy,\;\;\; x \in \mathbb R^n.
$$
By \cite[(2.2) and (2.8)]{DGMTZ}, we obtain
\begin{eqnarray*}
\int_0^{\gamma(x)^2}|K_1(x,y,t)|dt & \leq &C\int_0^{\gamma(x)^2}\int_{\mathbb R^n}V(z) \frac{1}{t^{\frac{n}{2}}}e^{-\frac{|x-y|^2}{2t}}\frac{1}{t^{\frac{n}{2}}}e^{-\frac{|z-y|^2}{2t}}dz dt\\
&\leq & C \int_0^{\gamma(x)^2}\frac{1}{t^{\frac{n}{2}}}e^{-\frac{|x-y|^2}{4t}}\int_{\mathbb R^n}\frac{1}{t^{\frac{n}{2}}}e^{-\frac{|x-z|^2}{4t}}V(z)dzdt\\
&\leq& C\int_0^{\gamma(x)^2}\frac{1}{t^{\frac{n}{2}+1}}e^{-\frac{|x-y|^2}{4t}}\left(\frac{\sqrt{t}}{\gamma(x)}\right)^\delta dt,
\end{eqnarray*}
for a certain $\delta>0$. Then,
\begin{eqnarray}\label{D7}
|T_1(f)(x)| &\leq &C \int_{|x-y|< \gamma(x)}|f(y)|\int_0^{\gamma(x)^2}\frac{1}{t^{\frac{n}{2}+1}}e^{-\frac{|x-y|^2}{4t}}\left(\frac{\sqrt{t}}{\gamma(x)}\right)^\delta dt dy \nonumber\\
&\leq& C \int_0^{\gamma(x)^2}\frac{t^{-1+\frac{\delta}{2}}}{\gamma(x)^\delta}\frac{1}{t^{\frac{n}{2}}}\int_{\mathbb R^n}|f(y)|e^{-\frac{|x-y|^2}{4t}}dydt\nonumber\\
&\leq&C \sup_{t>0}\frac{1}{t^{\frac{n}{2}}}\int_{\mathbb R^n}|f(y)|e^{-\frac{|x-y|^2}{4t}}dy \leq C\mathcal{M}(f)(x),\;\;x \in \mathbb R^n.
\end{eqnarray}
Also, since $\frac{t}{2}< t-s < t$ provided that $0< s < \frac{t}{2}$, \cite[(2.2) and (2.8)]{DGMTZ} imply that, for some $0< c < \frac{1}{4}$,
\begin{eqnarray*}
\int_0^{\gamma(x)^2}|K_2(x,y,t)|dt & \leq & C \int_0^{\gamma(x)^2}\int_0^{\frac{t}{2}}\int_{\mathbb R^n}V(z) \frac{1}{(t-s)^{\frac{n}{2}+1}}e^{-c\frac{|x-z|^2}{t-s}}\frac{1}{s^{\frac{n}{2}}}e^{-\frac{|y-z|^2}{4s}}dzdsdt \\
&\leq & C \int_0^{\gamma(x)^2}\frac{1}{t^{\frac{n}{2}+1}}\int_0^{\frac{t}{2}}\int_{\mathbb R^n}V(z) e^{-c\frac{|x-z|^2}{t}}\frac{1}{s^{\frac{n}{2}}}e^{-\frac{|y-z|^2}{4s}}dzdsdt \\
& \leq & C \int_0^{\gamma(x)^2}\frac{1}{t^{\frac{n}{2}+1}}\int_0^{\frac{t}{2}}\int_{\mathbb R^n}
e^{-\frac{c}{2}\frac{|x-z|^2+|y-z|^2}{t}}V(z)\frac{1}{s^{\frac{n}{2}}}e^{-\frac{c}{2}\frac{|y-z|^2}{s}}dzdsdt\\
&\leq & C\int_0^{\gamma(x)^2}\frac{1}{t^{\frac{n}{2}+1}}\int_0^{\frac{t}{2}} e^{-\frac{c}{2}\frac{|x-y|^2}{t}} \int_{\mathbb R^n}
\frac{1}{s^{\frac{n}{2}}} e^{-\frac{c}{2}\frac{|y-z|^2}{s}} V(z) dzdsdt \\
& \leq & C \int_0^{\gamma(x)^2}\frac{1}{t^{\frac{n}{2}+1}}e^{-\frac{c}{2}\frac{|x-y|^2}{t}}\int_0^{\frac{t}{2}} \frac{s^{-1+\frac{\delta}{2}}}{\gamma(y)^\delta}dsdt \\
&\leq & \frac{C}{\gamma(x)^\delta}\int_0^{\gamma(x)^2}\frac{1}{t^{\frac{n}{2}+1-\frac{\delta}{2}}}e^{-\frac{c}{2t}|x-y|^2} dt, \;\;\; x,y \in \mathbb R^n,\,\,\,|x-y|\le \gamma(x).
\end{eqnarray*}
We have taken into account that $\gamma(x)\thicksim\gamma(y)$, provided that $|x-y|\le \gamma(x)$. Then,
\begin{eqnarray}\label{D8}
|T_2f(x)| &\leq & \frac{C}{\gamma(x)^\delta}\int_{|x-y|< \gamma(x)}|f(y)|\int_0^{\gamma(x)^2}\frac{e^{-\frac{c}{2t}|x-y|^2}}{t^{\frac{n}{2}+1-\frac{\delta}{2}}}dtdy \nonumber\\
&\leq & \frac{C}{\gamma(x)^\delta}\int_0^{\gamma(x)^2}\frac{dt}{t^{1-\frac{\delta}{2}}}\int_{\mathbb R^n}
\frac{1}{t^{\frac{n}{2}}}e^{-\frac{c}{2t}|x-y|^2}|f(y)|dydt \nonumber\\
& \leq & C \sup_{t>0}\frac{1}{t^{\frac{n}{2}}}\int_{\mathbb R^n}e^{-\frac{c}{2t}|x-y|^2}|f(y)|dy \leq  C\mathcal{M}(f)(x), \;\; x \in \mathbb R^n.
\end{eqnarray}
By proceeding in a similar way and using \cite[(2.7)]{DGMTZ}we see that
\begin{equation}\label{D9}
|T_3f(x)|\leq C\mathcal{M}(f)(x), \;\;\; x\in \mathbb R^n.
\end{equation}

From (\ref{D7}), (\ref{D8}) and (\ref{D9}), we get
\begin{equation}\label{D10}
I_1(f)(x) \leq C\mathcal{M}(f)(x), \;\;\; x \in \mathbb R^n.
\end{equation}
Inequalities (\ref{D6}) and (\ref{D10}) imply that
$$
V_\rho(W_{t,{\rm loc}}^{\mathcal{L}} - W_{t,{\rm loc}})(f)(x) \leq C\mathcal{M}(f)(x),\;\;\; x \in \mathbb R^n.
$$
Hence, the operator $V_\rho(W_{t,{\rm loc}}^{\mathcal{L}}-W_{t,{\rm loc}})$ is bounded from $L^p(\mathbb R^n)$ into itself, for every $1<p<\infty$, and from $L^1(\mathbb R^n)$ into $L^{1,\infty}(\mathbb R^n)$.

Finally, by \eqref{rel1} we conclude that $V_\rho(W_t^{\mathcal{L}})$ is a bounded operator from $L^p(\mathbb R^n)$ into itself, for every $1<p<\infty$, and from $L^1(\mathbb R^n)$ into $L^{1,\infty}(\mathbb R^n)$.

\section{Variation operators associated with Riesz transform $R^\mathcal{L}$}

In this section we prove Proposition \ref{Pvriesz} and Theorem \ref{VarRiesz}. As it was mentioned in Section 1, for every $\ell =1,\cdots,n$, the $\ell $-th Riesz transform associated with $\mathcal{L}$ is defined formally by
\begin{equation}\label{D11}
R^{\mathcal{L}}_\ell  = \frac{\partial}{\partial x_\ell }\mathcal{L}^{-\frac{1}{2}}.
\end{equation}
Here $\mathcal{L}^{-\frac{1}{2}}$ denotes the negative square root of the operator $\mathcal{L}$ given by
$$
\mathcal{L}^{-\frac{1}{2}}f(x) = -\frac{1}{2\pi}\int_{\mathbb R^n}\int_{-\infty}^{+\infty}(-i\tau)^{-\frac{1}{2}}\Gamma(x,y,\tau)d\tau f(y) dy,
$$
where $\Gamma(x,y,\tau)$ represents the fundamental solution for the
operator $\mathcal{L}+i\tau$, with $\tau \in \mathbb R$ (see
\cite[\S2]{Sh}).

We recall that, for every $\ell=1,\cdots,n$, the function $R_\ell^\mathcal{L}(x,y)$ is defined by
$$
R_\ell^\mathcal{L}(x,y)=   -\frac{1}{2\pi}\int_{-\infty}^{+\infty}(-i\tau)^{-\frac{1}{2}}\frac{\partial}{\partial x_\ell}\Gamma(x,y,\tau)d\tau,\,\,\,x,y\in \mathbb{R}^n,\,\,\,x\neq y.
$$

The following estimates for the kernels $R_\ell^\mathcal{L}$,
$\ell=1,\cdots,n$, were established in \cite[Sections 5 and 6]{Sh}
(see also \cite[Lemma 1]{BHS2}) and will be very useful in the
sequel.

\begin{Lema} \label{Lemaesti} Let $\ell=1,\cdots,n$ and $V\in B_q$.

(i) Assume that $q>n$. Then, for every $k\in \mathbb{N}$ there exists $C>0$ such that
\begin{equation}\label{L1}
|R_\ell^\mathcal{L}(x,y)|\le C\frac{1}{(1+|x-y|/\gamma(x))^k}\frac{1}{|x-y|^n}.
\end{equation}
Moreover, we have
\begin{equation}\label{L2}
|R_\ell^\mathcal{L}(x,y)-R_\ell(x-y)|\le
C\frac{1}{|x-y|^n}\Big(\frac{|x-y|}{\gamma(x)}\Big)^{2-n/q},\,\,\,0<|x-y|<\gamma(x).
\end{equation}

(ii) Suppose that $n/2<q<n$. Then, for every $k\in \mathbb{N}$ there exists $C>0$ such that
\begin{eqnarray}\label{L3}
|R_\ell^\mathcal{L}(x,y)|&\le& C\frac{1}{(1+|x-y|/\gamma(y))^k}\frac{1}{|x-y|^{n-1}}\nonumber\\
&\times&\Big(\frac{1}{|x-y|}+\int_{B(x,|x-y|/4)}\frac{V(z)}{|z-x|^{n-1}}dz\Big).
\end{eqnarray}
Also, we have
\begin{eqnarray}\label{L4}
|R_\ell^\mathcal{L}(x,y)-R_\ell(x-y)|&\le& C\frac{1}{|x-y|^{n-1}}\nonumber\\
&\times&\Big(\frac{1}{|x-y|}\Big(\frac{|x-y|}{\gamma(x)}\Big)^{2-n/q}+\int_{B(x,|x-y|/4)}\frac{V(z)}{|z-x|^{n-1}}dz\Big),\,\,\,0<|x-y|<\gamma(x).
\end{eqnarray}
\end{Lema}

\begin{Rm} Note that, according to \cite{G}, if $V\in B_n$, there exists $\varepsilon>0$ such that $V\in B_{n+\varepsilon}$. Then, the estimates in Lemma \ref{Lemaesti}, $(i)$, can be applied to $V\in B_n$ by taking $q=n+\varepsilon$, where $\varepsilon>0$ is small enough and it depends on $V$.
\end{Rm}

We consider, for every $\ell=1,\cdots,n$ and $\varepsilon>0$,
$$
R_\ell^{\mathcal{L},\varepsilon}(f)(x)=\int_{|x-y|>\varepsilon}R_\ell^\mathcal{L}(x,y)f(y)dy,\,\,\,x\in \mathbb{R}^n.
$$
According to Lemma \ref{Lemaesti}, it is not hard to see that, for every $\ell=1,\cdots,n$, $\varepsilon>0$, $f\in L^p(\mathbb{R}^n)$, $1\le p<\infty$, the integral defining $R_\ell^{\mathcal{L},\varepsilon}(f)(x)$ is absolutely convergent for each $x\in \mathbb{R}^n$.

Before proving Proposition \ref{Pvriesz} we establish the $L^p$-boundedness properties of the maximal operator $R_\ell^{\mathcal{L},*}$ defined by
$$
R_\ell^{\mathcal{L},*}(f)=\sup_{\varepsilon>0}|R_\ell^{\mathcal{L},\varepsilon}(f)|.
$$

\begin{Prop} \label{maxRiesz} Let $\ell =1,\cdots,n$ and $V\in B_q$, $q\ge n/2$. Then, if $1<p<p_0$, where $\frac{1}{p_0}=\Big(\frac{1}{q}-\frac{1}{n}\Big)_+$, $R_\ell ^{\mathcal{L},*}$ is bounded from $L^p(\mathbb R^n)$ into itself. Moreover, $R_\ell ^{\mathcal{L},*}$ is bounded from $L^1(\mathbb R^n)$ into $L^{1,\infty}(\mathbb R^n)$.
\end{Prop}

\begin{proof}
%(i) Assume that $q\ge n$. In \cite[Theorem 0.8]{Sh} it was established that the kernel
%$$
%R_\ell ^\mathcal{L}(x,y)=-\frac{1}{2\pi}\int_{-\infty}^\infty (-i\tau)^{-1/2}\frac{\partial}{\partial x_\ell }\Gamma(x,y,\tau)d\tau,\,\,\,x,y\in \mathbb{R}^n,\,\,\,x\neq y,
%$$
%is a standard Calder\'on-Zygmund one. Moreover, the Riesz transform $R_\ell ^\mathcal{L}$ is bounded from $L^2(\mathbb{R}^n)$ into itself. Then, according to \cite[Theorem 8.2.2]{Gra}, the maximal operator $R_\ell ^{\mathcal{L},*}$ is bounded from $L^p(\mathbb{R}^n)$ into itself, for every $1<p<\infty$, and from $L^1(\mathbb R^n)$ into $L^{1,\infty}(\mathbb R^n)$.

It is enough to prove the result when $n/2< q<n$. Indeed, according to \cite{G}, if $V\in B_{n/2}$ then $V\in B_{\varepsilon+n/2}$ for some $\varepsilon>0$. Moreover, $B_r\subset B_s$, when $r\ge s$. We split the operator $R_\ell ^{\mathcal{L},*}$, in the spirit of the general procedure described at the beginning of this section, as follows
\begin{eqnarray*}
R_\ell ^{\mathcal{L},*}(f)(x)&\le& \int_{|x-y|<\gamma(x)}|R_\ell ^\mathcal{L}(x,y)-R_\ell (x-y)||f(y)|dy\\
&+&\int_{|x-y|>\gamma(x)}|R_\ell ^\mathcal{L}(x,y)||f(y)|dy+\sup_{\varepsilon>0}\Big|\int_{\varepsilon<|x-y|<\gamma(x)}R_\ell (x-y)f(y)dy\Big|\\
&=&\tau_1(|f|)(x)+\tau_2(|f|)(x)+\tau_3(f)(x),\,\,\,x\in \mathbb{R}^n.
\end{eqnarray*}
It is clear that
\begin{eqnarray*}
\tau_3(f)(x)
%=&\sup_{\varepsilon>0}\Big|\Big(\int_{\varepsilon<|x-y|<\gamma(x)/M}+\int_{\gamma(x)/M\le|x-y|<\gamma(y)}\Big)
%R_\ell (x-y)f(y)dy\Big|\\
%&\le& \sup_{\varepsilon>0}\Big|\Big(\int_{\varepsilon<|x-y|}-\int_{|x-y|\ge\gamma(x)/M}\Big)R_\ell (x-y)f(y)dy\Big|\\
%&+&\int_{\gamma(x)/M\le |x-y|\le M\gamma(x)}|R_\ell (x-y)||f(y)|dy\\
%&\le&2\sup_{\varepsilon>0}\Big|\int_{|x-y|>\varepsilon}R_\ell (x-y)f(y)dy\Big|+\int_{\gamma(x)/M\le |x-y|\le M\gamma(x)}\frac{1}{|x-y|^n}|f(y)|dy\\
%&\le& C\Big(\sup_{\varepsilon>0}\Big|\int_{|x-y|>\varepsilon}R_\ell (x-y)f(y)dy\Big|+\frac{1}{\gamma(x)^n}\int_{ |x-y|\le M\gamma(x)}|f(y)|dy\Big)\\
&\le&\sup_{\varepsilon>0}\Big|\int_{|x-y|>\varepsilon}R_\ell (x-y)f(y)dy\Big|,\,\,\,x\in \mathbb{R}^n.
\end{eqnarray*}
Then, from well known results we infer that $\tau_3$ is bounded from $L^r(\mathbb{R}^n)$ into itself when $1<r<\infty$, and from $L^1(\mathbb{R}^n)$ into $L^{1,\infty}(\mathbb{R}^n)$.

For a certain $M>1$, $\frac{1}{M}\le\frac{\gamma(x)}{\gamma(y)}\le M$, when $|x-y|\le \gamma(y)$. Moreover, if $|x-y|>\gamma(x)$, $|x-y|>\gamma(y)/M$. Indeed, it is sufficient to observe that
\begin{eqnarray*}
\{y\in \mathbb{R}^n:|x-y|>\gamma(x)\}&\subset& \{y\in
\mathbb{R}^n:|x-y|<\gamma(y)\,\,\,and\,\,\,|x-y|>\gamma(x)\}\\
&\bigcup&\{y\in \mathbb{R}^N:\,|x-y|>\gamma(y)\,\,\,and\,\,\,|x-y|>\gamma(x)\}.
\end{eqnarray*}
We denote by $\tau_j^*$ the adjoint operator of $\tau_j$, $j=1,2$. We have that
\begin{eqnarray*}
\tau_1^*(|g|)(y)&=&\int_{|x-y|<\gamma(x)}|R_\ell ^\mathcal{L}(x,y)-R_\ell (x-y)|g(x)dx\\
&\le& \int_{|x-y|<M\gamma(y)}|R_\ell ^\mathcal{L}(x,y)-R_\ell (x-y)|g(x)dx,\,\,\,y\in \mathbb{R}^n,
\end{eqnarray*}
and
\begin{eqnarray*}
\tau_2^*(|g|)(y)&=&\int_{|x-y|\ge \gamma(x)}|R_\ell^\mathcal{L}(x,y)||g(x)|dx\\
&\le&\int_{|x-y|\ge
\gamma(y)/M}|R_\ell^\mathcal{L}(x,y)||g(x)|dx,\,\,\,y\in
\mathbb{R}^n.
\end{eqnarray*}
According to \cite[(5.9) and the proof of Lemma 5.8]{Sh} and the
arguments in the proof of \cite[Lemma 5.7]{Sh} we obtain that 
$\tau_i^*$, $i=1,2$, are bounded from $L^r(\mathbb{R}^n)$ into itself
when $p_0'<r\le \infty$. Then, $\tau_i$, $i=1,2$, are bounded from
$L^r(\mathbb{R}^n)$ into itself when $1\le r<p_0$.

By combining the above properties we conclude that $R_\ell ^{\mathcal{L},*}$ is a bounded operator from $L^p(\mathbb{R}^n)$ into itself, provided that $1<p<p_0$, and from $L^1(\mathbb{R}^n)$ to $L^{1,\infty}(\mathbb{R}^n)$.
\end{proof}

\vspace{5mm}

\begin{proof}[\bf Proof of Proposition \ref{Pvriesz}.]

 By Proposition \ref{maxRiesz}, in order to show Proposition \ref{Pvriesz} it is sufficient to see (\ref{repre}) for every $\phi\in C_c^\infty(\mathbb{R}^n)$.

Suppose that $\phi\in C_c^\infty(\mathbb{R}^n)$. By $\Gamma_0(x,y,\tau)$ we denote the fundamental solution of the operator $-\Delta +i \tau$ in $\mathbb R^n$, with $\tau \in \mathbb R$. We are going to see that, for every $\ell =1,\cdots,n$, $\mathcal{L}^{-\frac{1}{2}}\phi - (-\Delta)^{-\frac{1}{2}}\phi$ admits derivative with respect to $x_\ell $, for almost all $\mathbb R^n$, and that ${\rm a.e.}\; x \in \mathbb R^n$,
$$
\frac{\partial}{\partial x_\ell }\left(\mathcal{L}^{-\frac{1}{2}}\phi(x)-(-\Delta)^{-\frac{1}{2}}\phi(x)\right)=-\frac{1}{2\pi}\int_{\mathbb R^n}\phi(y)\int_{-\infty}^{+\infty}\frac{\partial}{\partial x_\ell }\left(\Gamma(x,y,\tau)-\Gamma_0(x,y,\tau)\right)(-i\tau)^{-\frac{1}{2}}d\tau dy.
$$
To simplify the notation we consider $\ell =1$. Also we assume that $q>n/2$ (\cite{G}). According to \cite[Lemma 4.5]{Sh}, for every $k \in \mathbb N$ there exists $C_k >0$ such that
\begin{equation}\label{D12}
\left|\Gamma(x,y,\tau)-\Gamma_0(x,y,\tau)\right| \leq \frac{C_k|x-y|^{2-n}}{(1+|\tau|^{\frac{1}{2}}|x-y|)^k}\left(\frac{|x-y|}{\gamma(x)}\right)^{2-\frac{n}{q}}
\end{equation}
provided that $\tau \in\mathbb R$, $x,y\in \mathbb R^n$, $|x-y| \leq \gamma(x)$. Moreover, by \cite[Theorem 2.7]{Sh}, for every  $k\in \mathbb N$ there exists $C_k >0$ such that
\begin{equation}\label{D13}
\left|\Gamma(x,y,\tau)\right| \leq \frac{C_k|x-y|^{2-n}}{(1+|\tau|^{\frac{1}{2}}|x-y|)^k\left(1+\frac{|x-y|}{\gamma(x)}\right)^k}, \;\;x,y\in \mathbb R^n,\;\;\tau \in \mathbb R ,
\end{equation}
and
\begin{equation}\label{D14}
\left|\Gamma_0(x,y,\tau)\right| \leq \frac{C_k|x-y|^{2-n}}{(1+|\tau|^{\frac{1}{2}}|x-y|)^k}, \;\;x,y\in \mathbb {R}^n,\;\;\tau \in \mathbb{R}.
\end{equation}
We will use repeatedly without reference the following equality
$$
\int_{-\infty}^{+\infty} \frac{1}{|\tau|^{\frac{1}{2}}}\frac{1}{(1+|x-y||\tau|^{\frac{1}{2}})^{k}}d\tau=\frac{C_k}{|x-y|},\,\,\,x,y\in \mathbb{R}^n,\,\,\,x\neq y\,\,\,\rm{and}\,\,\,k\ge 1.
$$
From (\ref{D13}) and (\ref{D14}) we deduce that
\begin{multline}\label{Abs}
\int_{\mathbb{R}^n}\left|\phi(y)\right|\int_{-\infty}^{+\infty} |\Gamma(x,y,\tau)-\Gamma_0(x,y,\tau)||\tau |^{-\frac{1}{2}}d\tau dy\\
\shoveleft{\hspace{33mm}\leq C\int_{\mathbb{R}^n}\frac{|\phi(y)|}{|x-y|^{n-2}}\int_{-\infty}^{+\infty} \frac{1}{|\tau|^{\frac{1}{2}}}\frac{1}{(1+|x-y||\tau|^{\frac{1}{2}})^{2}}d\tau dy}\\
\shoveleft{\hspace{33mm}\leq C \int_{\mathbb{R}^n}\frac{|\phi(y)|}{|x-y|^{n-1}}dy}\\
\leq C \frac{1}{1+|x|^{n-1}},\;\;\; x \in \mathbb{R}^n. \hspace{60mm}
\end{multline}

Then, the function $F=\mathcal{L}^{-\frac{1}{2}}\phi-(-\Delta)^{-\frac{1}{2}}\phi$ defines in $\mathbb R^n$ a distribution $S_F$ by
$$
\langle S_F, \psi \rangle = \int_{\mathbb R^n}F(y)\psi(y)dy,\;\;\; \psi \in C_c^\infty(\mathbb R^n).
$$
We have that, by writing $\overline x=(x_2,\cdots,x_n)$ and considering $ \psi \in C_c^\infty(\mathbb R^n)$,
\begin{multline}\label{DD}
\Big\langle \frac{\partial}{\partial x_1}S_F, \psi \Big\rangle = -\int_{\mathbb R^n} F(x)\frac{\partial}{\partial x_1}\psi(x)dx \\
\shoveleft{\hspace{4mm}=\frac{1}{2\pi}\int_{\mathbb R^n}\int_{\mathbb R^n}\phi(y)\int_{-\infty}^{+\infty}(-i\tau)^{-\frac{1}{2}}\left(\Gamma(x,y,\tau)-\Gamma_0(x,y,\tau)\right)d\tau dy \frac{\partial}{\partial x_1}\psi(x) dx} \\
\shoveleft{\hspace{4mm}=\frac{1}{2\pi}\int_{\mathbb R^{n-1}}\int_{\mathbb R^n}\phi(y)\int_{-\infty}^{+\infty}(-i\tau)^{-\frac{1}{2}}\int_{-\infty}^{+\infty}\left(\Gamma(x,y,\tau)-\Gamma_0(x,y,\tau)\right) \frac{\partial}{\partial x_1}\psi(x)dx_1d\tau dy d\overline x}.
\end{multline}
We now apply partial integration in the following way. Since the integral is absolutely convergent (see (\ref{Abs})) it follows that
\begin{eqnarray}\label{PI1}
\lefteqn{\int_{\mathbb R^{n-1}}\int_{\mathbb R^n}\phi(y)\int_{-\infty}^{+\infty}(-i\tau)^{-\frac{1}{2}}
\int_{-\infty}^{+\infty}\left(\Gamma(x,y,\tau)-\Gamma_0(x,y,\tau)\right) \frac{\partial}{\partial x_1}\psi(x)dx_1d\tau dy d\overline x}\nonumber\\
&=&\lim_{\varepsilon\to 0^+}\int_{\mathbb R^{n-1}}\int_{\mathbb R^n}\phi(y)\int_{-\infty}^{+\infty}(-i\tau)^{-\frac{1}{2}}\nonumber\\
&&\times\Big(\int_{-\infty}^{y_1-\varepsilon}+\int_{y_1+\varepsilon}^{+\infty}\Big)\left(\Gamma(x,y,\tau)-\Gamma_0(x,y,\tau)\right) \frac{\partial}{\partial x_1}\psi(x)dx_1d\tau dy d\overline x\nonumber\\
&=&\lim_{\varepsilon\to 0^+}\Big(-\int_{\mathbb R^{n-1}}\int_{\mathbb R^n}\phi(y)\int_{-\infty}^{+\infty}(-i\tau)^{-\frac{1}{2}}\\
&&\times\Big(\int_{-\infty}^{y_1-\varepsilon}+\int_{y_1+\varepsilon}^{+\infty}\Big)\frac{\partial}{\partial x_1}\left(\Gamma(x,y,\tau)-\Gamma_0(x,y,\tau)\right) \psi(x)dx_1d\tau dy d\overline x\nonumber\\
&&+\int_{\mathbb R^{n-1}}\int_{\mathbb R^n}\phi(y)\int_{-\infty}^{+\infty}(-i\tau)^{-\frac{1}{2}}\Big((\Gamma(y_1-\varepsilon,\overline{x},y,\tau)-\Gamma_0(y_1-\varepsilon,\overline{x},y,\tau))\psi(y_1-\varepsilon,\overline{x})\nonumber\\
&&-(\Gamma(y_1+\varepsilon,\overline{x},y,\tau)-\Gamma_0(y_1+\varepsilon,\overline{x},y,\tau))\psi(y_1+\varepsilon,\overline{x})\Big)d\tau dyd\overline{x}\Big).\nonumber
\end{eqnarray}
We have that
\begin{equation}\label{PI2}
\int_{\mathbb R^{n}}\int_{\mathbb R^n}\int_{-\infty}^{+\infty}|\phi(y)||\psi(x)||\tau|^{-1/2}\Big|\frac{\partial}{\partial x_1}\left(\Gamma(x,y,\tau)-\Gamma_0(x,y,\tau)\right)\Big|d\tau dxdy<\infty.
\end{equation}
Indeed, by \cite[Theorem 2.7]{Sh} (see also the proof of \cite[Lemma 5.7]{Sh}), for every $k\in \mathbb{N}$, there exists $C_k>0$ such that, for every $x,y\in \mathbb{R}^n$ and $\tau\in \mathbb{R}$,
\begin{equation}\label{PI3}
\Big|\frac{\partial}{\partial x_1}\Gamma(x,y,\tau)\Big|\le C_k\frac{|x-y|^{2-n}}{(1+|\tau|^{1/2}|x-y|)^k\Big(1+\frac{|x-y|}{\gamma(y)}\Big)^k}\Big(\int_{|z-x|<\frac{|x-y|}{4}}\frac{V(z)}{|z-x|^{n-1}}dz+\frac{1}{|x-y|}\Big).
\end{equation}
Moreover, for every $k\in \mathbb{N}$, there exists $C_k>0$ such that
\begin{equation}\label{PI4}
\Big|\frac{\partial}{\partial x_1}\Gamma_0(x,y,\tau)\Big|\le C_k\frac{|x-y|^{1-n}}{(1+|\tau|^{1/2}|x-y|)^k},\,\,\,x,y\in \mathbb{R}^n,\,\,\,\tau\in \mathbb{R}.
\end{equation}
By \cite[p. 541]{Sh}, for every $k\in \mathbb{N}$, there exists $C_k>0$ such that
\begin{eqnarray}\label{PI5}
&&\Big|\frac{\partial}{\partial x_1}\Gamma(x,y,\tau)-\frac{\partial}{\partial x_1}\Gamma_0(x,y,\tau)\Big|\le C_k\frac{|x-y|^{2-n}}{(1+|\tau|^{1/2}|x-y|)^k}\nonumber\\
&\times&\Big(\int_{|z-x|<\frac{|x-y|}{4}}\frac{V(z)}{|z-x|^{n-1}}dz+\frac{1}{|x-y|}\Big(\frac{|x-y|}{\gamma(y)}\Big)^{2-n/q}\Big),\,\,\,|x-y|<\gamma(y),\,\,\,\tau\in \mathbb{R},
\end{eqnarray}
where $q>n/2$.

Assume firstly that $V\in B_n$. Then $V\in B_q$, for some $q>n$ (\cite{G}). According to \cite[Lemma 1]{BHS2}, we deduce, from (\ref{PI3}) and (\ref{PI5}), that for every $k\in \mathbb{N}$, there exists $C_k>0$ such that
\begin{equation}\label{PI6}
\Big|\frac{\partial}{\partial x_1}\Gamma(x,y,\tau)\Big|\le C_k\frac{|x-y|^{1-n}}{(1+|\tau|^{1/2}|x-y|)^k\Big(1+\frac{|x-y|}{\gamma(y)}\Big)^k},\,\,\,x,y\in \mathbb{R}^n,\,\,\,\tau\in \mathbb{R},
\end{equation}
and, for every $\tau\in \mathbb{R}$ and $|x-y|<\gamma(y)$,
\begin{equation}\label{PI7}
\Big|\frac{\partial}{\partial x_1}\Gamma(x,y,\tau)-\frac{\partial}{\partial x_1}\Gamma_0(x,y,\tau)\Big|\le C_k\frac{|x-y|^{1-n}}{(1+|\tau|^{1/2}|x-y|)^k}\Big(\frac{|x-y|}{\gamma(y)}\Big)^{2-n/q}.
\end{equation}

According to \cite[Lemma 1.4]{Sh}, $\gamma$ and $1/\gamma$ are
bounded on any compact subset of $\mathbb{R}^n$. Since $\phi$ and
$\psi$ have compact support, there exists $A>0$ such that $|x-y|\le
A\gamma(y)$, $x\in \rm{supp}\psi$ and $y\in \rm{supp}\,\phi$. Then,
by using (\ref{PI4}), (\ref{PI6}) and (\ref{PI7}) we get
\begin{eqnarray*}
\,&\,&\int_{\mathbb R^{n}}\int_{\mathbb R^n}\int_{-\infty}^{+\infty}|\phi(y)||\psi(x)||\tau|^{-1/2}\Big|\frac{\partial}{\partial x_1}\left(\Gamma(x,y,\tau)-\Gamma_0(x,y,\tau)\right)\Big|d\tau dxdy\\
&\le&C\int_{\rm{supp}\,\psi}|\psi(x)|\int_{\rm{supp}\,\phi}|\phi(y)|\int_{-\infty}^\infty \frac{|x-y|^{3-n-n/q}}{|\tau|^{1/2}(1+|\tau|^{1/2}|x-y|)^2}d\tau dydx\\
&\le&C
\int_{\rm{supp}\,\psi}\int_{\rm{supp}\,\phi}\frac{1}{|x-y|^{n-2+n/q}}dydx<\infty.
\end{eqnarray*}

%Hence, by using (\ref{PI4}), (\ref{PI6}) and (\ref{PI7}) we get
%\begin{eqnarray*}
%\,&\,&\int_{\mathbb R^{n}}\int_{\mathbb R^n}\int_{-\infty}^{+\infty}|\phi(y)||\psi(x)||\tau|^{-1/2}\Big|\frac{\partial}{\partial x_1}\left(\Gamma(x,y,\tau)-\Gamma_0(x,y,\tau)\right)\Big|d\tau dxdy\\
%&\le& C\Big(\int_{\mathbb{R}^n}|\psi(x)|\int_{|x-y|>\gamma(x)}|\phi(y)|\int_{-\infty}^\infty \frac{\gamma(x)^{1-n}}{|\tau|^{1/2}(1+|\tau|^{1/2}\gamma(x))^2}d\tau dydx\\
%&&+\int_{\mathbb{R}^n}|\psi(x)|\int_{|x-y|\le \gamma(x)}|\phi(y)|\int_{-\infty}^\infty \frac{|x-y|^{3-n-n/q}}{|\tau|^{1/2}(1+|\tau|^{1/2}|x-y|)^2}d\tau dydx\Big)\\
%&\le&C\Big(\int_{\mathbb{R}^n}|\psi(x)|\gamma(x)^{-n}dx\int_{\mathbb{R}^n}|\phi(y)|dy\int_{0}^\infty \frac{1}{\sqrt{u}(1+u)}du\\
%&&+\int_{\mathbb{R}^n}|\psi(x)|\int_{|x-y|\le \gamma(x)}\frac{|\phi(y)|}{|x-y|^{n-2+n/q}}dydx\int_0^\infty\frac{1}{\sqrt{u}(1+u)}du\Big)<\infty .
%\end{eqnarray*}
%Note that $1-\frac{n}{q}>0$ and, by \cite[Lemma 1.4, (b) and (c)]{Sh}, $\gamma$ are $1/\gamma$ are bounded on any compact on $\mathbb{R}^n$.
%

We consider now $\frac{n}{2}<q<n$. We recall that the 1-th Euclidean fractional integral is bounded from $L^q(\mathbb{R}^n)$ into $L^{p_0}(\mathbb{R}^n)$, when $\frac{1}{p_0}=\frac{1}{q}-\frac{1}{n}$ (\cite{Stein1}). Since $q>n/2$ in order to establish (\ref{PI2}) we only have to see that
$$
\int_{\mathbb{R}^n}|\phi(y)|\int_{|x-y|\le \gamma(y)}|\psi(x)|\int_{-\infty}^\infty \frac{|x-y|^{2-n}}{|\tau|^{1/2}(1+|\tau|^{1/2}|x-y|)^2}\int_{|z-x|<\frac{|x-y|}{4}}\frac{V(z)}{|z-x|^{n-1}}dz  d\tau dydx<\infty.
$$
To do this we can proceed as follows. There exists $M>0$ for which
\begin{eqnarray*}
\,&\,&\int_{\mathbb{R}^n}|\phi(y)|\int_{|x-y|\le \gamma(y)}|\psi(x)|\int_{-\infty}^\infty \frac{|x-y|^{2-n}}{|\tau|^{1/2}(1+|\tau|^{1/2}|x-y|)^2}\int_{|z-x|<\frac{|x-y|}{4}}\frac{V(z)}{|z-x|^{n-1}}dz  d\tau dydx\\
&\le&C \int_{\rm{supp}\,\psi}\int_{\rm{supp}\,\phi}\frac{1}{|x-y|^{n-1}}\int_{|z|\le M}\frac{|V(z)|}{|z-x|^{n-1}}dzdydx\\
&\le&C \int_{\rm{supp}\,\phi}\Big(\int_{\rm{supp}\,\psi}\frac{1}{|x-y|^{p_0'(n-1)}}dx\Big)^{1/p_0'}\Big(\int_{\mathbb{R}^n}|I_1(\chi_{B(0,M)}V)(x)|^{p_0}dx\Big)^{1/p_0}dy\\
&\le&C
\int_{\rm{supp}\,\phi}\Big(\int_{\rm{supp}\,\psi}\frac{1}{|x-y|^{p_0'(n-1)}}dx\Big)^{1/p_0'}\Big(\int_{B(0,M)}|V(z)|^qdz\Big)^{1/q}dy<\infty,
\end{eqnarray*}
because $V\in L_{\rm{loc}}^q(\mathbb{R}^n)$, and $(n-1)(p_0'-1)<1$ when $q>n/2$.

Hence, (\ref{PI2}) holds and we can write
\begin{eqnarray*}
\lim_{\varepsilon\to 0^+}\int_{\mathbb R^{n-1}}\int_{\mathbb R^n}\phi(y)\int_{-\infty}^{+\infty}(-i\tau)^{-\frac{1}{2}}\Big(\int_{-\infty}^{y_1-\varepsilon}
+\int_{y_1+\varepsilon}^{+\infty}\Big)\frac{\partial}{\partial x_1}\left(\Gamma(x,y,\tau)-\Gamma_0(x,y,\tau)\right) \psi(x)dx_1d\tau dy d\overline x&&\\
&\hspace{-24cm}=&\hspace{-12cm}\int_{\mathbb R^{n}}\psi(x)\int_{\mathbb R^n}\phi(y)\int_{-\infty}^{+\infty}(-i\tau)^{-\frac{1}{2}}\frac{\partial}{\partial x_1}\left(\Gamma(x,y,\tau)-\Gamma_0(x,y,\tau)\right)d\tau dydx.
\end{eqnarray*}
Our next task is to see that $\lim_{\varepsilon\to 0^+}I(\varepsilon)=0$, where
\begin{eqnarray*}
I(\varepsilon)&=&\int_{\mathbb R^{n-1}}\int_{\mathbb R^n}\phi(y)\int_{-\infty}^{+\infty}(-i\tau)^{-\frac{1}{2}}\Big((\Gamma(y_1-\varepsilon,\overline{x},y,\tau)
-\Gamma_0(y_1-\varepsilon,\overline{x},y,\tau))\psi(y_1-\varepsilon,\overline{x})\\
&-&(\Gamma(y_1+\varepsilon,\overline{x},y,\tau)
-\Gamma_0(y_1+\varepsilon,\overline{x},y,\tau))\psi(y_1+\varepsilon,\overline{x})\Big)d\tau dyd\overline{x},\,\,\,\varepsilon>0.
\end{eqnarray*}
We have that
\begin{eqnarray*}
\lefteqn{(\Gamma(y_1-\varepsilon,\overline{x},y,\tau)-\Gamma_0(y_1-\varepsilon,\overline{x},y,\tau))
\psi(y_1-\varepsilon,\overline{x})
-(\Gamma(y_1+\varepsilon,\overline{x},y,\tau)-\Gamma_0(y_1+\varepsilon,\overline{x},y,\tau))
\psi(y_1+\varepsilon,\overline{x})}\\
&=&\hspace{-0.75cm}\Big((\Gamma(y_1-\varepsilon,\overline{x},y,\tau)-\Gamma(y_1+\varepsilon,\overline{x},y,\tau))
-(\Gamma_0(y_1-\varepsilon,\overline{x},y,\tau)-\Gamma_0(y_1+\varepsilon,\overline{x},y,\tau))\Big)
\psi(y_1-\varepsilon,\overline{x})\\
&\hspace{1cm}+&(\psi(y_1-\varepsilon,\overline{x})-\psi(y_1+\varepsilon,\overline{x}))
(\Gamma(y_1+\varepsilon,\overline{x},y,\tau)
-\Gamma_0(y_1+\varepsilon,\overline{x},y,\tau))\\
&=&\hspace{-0.75cm}J_1(\overline{x},y,\varepsilon,\tau)+J_2(\overline{x},y,\varepsilon,\tau),\,\,\,\overline{x}\in \mathbb{R}^{n-1},\,\,y\in \mathbb{R}^n,\,\,\varepsilon>0,\,\,\tau\in \mathbb{R}.
\end{eqnarray*}
According to \cite[p. 535]{Sh}, if $\overline{x}\in \mathbb{R}^{n-1}$, $y\in \mathbb{R}^n$, $0<\varepsilon<|\overline{x}-\overline{y}|/15$ and $\tau\in \mathbb{R}$,
\begin{equation}\label{PI8}
|\Gamma(y_1-\varepsilon,\overline{x},y,\tau)-\Gamma(y_1+\varepsilon,\overline{x},y,\tau)|\le C \frac{\varepsilon^\delta|\overline{x}-\overline{y}|^{2-n-\delta}}{(1+|\tau|^{1/2}|\overline{x}-\overline{y}|)^3},
\end{equation}
and
\begin{equation}\label{PI9}
|\Gamma_0(y_1-\varepsilon,\overline{x},y,\tau)-\Gamma_0(y_1+\varepsilon,\overline{x},y,\tau)|\le C \frac{\varepsilon^\delta|\overline{x}-\overline{y}|^{2-n-\delta}}{(1+|\tau|^{1/2}|\overline{x}-\overline{y}|)^3},
\end{equation}
for a certain $\delta>0$ that depends on $q$.

By using (\ref{D12}), (\ref{D13}) and (\ref{D14}) we can deduce that
$$
\lim_{\eta\to 0^+}\int_{\mathbb{R}^{n}}|\phi(y)|\int_{|\overline{x}-\overline{y}|<\eta}\int_{-\infty}^\infty |\tau|^{-1/2} |J_m(\overline{x},y,\varepsilon,\tau)|d\tau dyd\overline{x}=0,\,\,\,m=1,2,
$$
uniformly in $\varepsilon\in (0,1)$.

Indeed, let $\varepsilon\in (0,1)$. According to \eqref{D13} and \eqref{D14}, the mean value theorem leads to
$$
|J_2(\overline{x},y,\varepsilon, \tau)|\le C\varepsilon \frac{(\varepsilon+|\overline{x}-\overline{y}|)^{2-n}}{(1+|\tau|^{1/2}(\varepsilon+|\overline{x}-\overline{y}|))^3},\,\,\,\overline{x},\,\,\overline{y}\in \mathbb{R}^{n-1}\,\,\,\mbox{and}\,\,\,\tau\in \mathbb{R}.
$$
Then, we have
\begin{eqnarray*}
&&\int_{\mathbb{R}^n}|\phi(y)|\int_{|\overline{x}-\overline{y}|<\eta}\int_{-\infty}^\infty |\tau|^{-1/2}|J_2(\overline{x},y,\varepsilon, \tau)|d\tau d\overline{x}dy\\
&&\hspace{15mm}\le C\int_{\mathbb{R}^n}|\phi(y)|\int_{|\overline{x}-\overline{y}|<\eta}\int_{-\infty}^\infty |\tau|^{-1/2}\varepsilon \frac{(\varepsilon+|\overline{x}-\overline{y}|)^{2-n}}{(1+|\tau|^{1/2}(\varepsilon+|\overline{x}-\overline{y}|))^3}d\tau d\overline{x}dy\\
%&&\hspace{15mm}\le C\int_{\mathbb{R}^n}|\phi(y)|\int_{|\overline{x}-\overline{y}|<\eta}\int_0^\infty \frac{1}{u^{1/2}(1+u^{1/2})^3}\varepsilon(\varepsilon+|\overline{x}-\overline{y}|)^{1-n}dud\overline{x}dy\\
&&\hspace{15mm}\le C\int_{\mathbb{R}^n}|\phi(y)|\int_{|\overline{x}-\overline{y}|<\eta}\varepsilon(\varepsilon+|\overline{x}-\overline{y}|)^{1-n}d\overline{x}dy\\
%&&\hspace{15mm}\le C\int_{\mathbb{R}^n}|\phi(y)|\int_{|\overline{x}-\overline{y}|<\eta}(\varepsilon+|\overline{x}-\overline{y}|)^{2-n}d\overline{x}dy\\
&&\hspace{15mm}\le C\int_{\mathbb{R}^n}|\phi(y)|\int_{|\overline{x}-\overline{y}|<\eta}|\overline{x}-\overline{y}|^{2-n}d\overline{x}dy,\,\,\,\eta>0,
\end{eqnarray*}
where $C>0$ does not depend on $\varepsilon$. Hence,
$$
\lim_{\eta\to 0^+}\int_{\mathbb{R}^{n}}|\phi(y)|\int_{|\overline{x}-\overline{y}|<\eta}\int_{-\infty}^\infty |\tau|^{-1/2} |J_2(\overline{x},y,\varepsilon,\tau)|d\tau dyd\overline{x}=0,
$$
uniformly in $\varepsilon\in (0,1)$.

By using \eqref{PI5}, since ${\rm supp}\,\phi$ is compact, we get, for a certain $a>0$, that if $\varepsilon\in (0,1)$
\begin{eqnarray*}
\lefteqn{\int_{\mathbb{R}^n}|\phi(y)|\int_{|\overline{x}-\overline{y}|<\eta}\int_{-\infty}^\infty |\tau|^{-1/2}|J_1(\overline{x},y,\varepsilon, \tau)|d\tau d\overline{x}dy}\\
&&\le C\int_{\mathbb{R}^n}|\phi(y)|\int_{|\overline{x}-\overline{y}|<\eta}\int_{-\infty}^\infty |\tau|^{-1/2}
\Big|\int_{y_1-\varepsilon}^{y_1+\varepsilon}\frac{\partial}{\partial u}\Big(\Gamma(u,\overline{x},y,\tau)-\Gamma_0(u,\overline{x},y,\tau)\Big)du\Big|d\tau d\overline{x}dy\\
&&\le C\int_{\mathbb{R}^n}|\phi(y)|\int_{|\overline{x}-\overline{y}|<\eta}\int_{y_1-\varepsilon}^{y_1+\varepsilon}\int_{-\infty}^\infty |\tau|^{-1/2}
\Big|\frac{\partial}{\partial u}\Big(\Gamma(u,\overline{x},y,\tau)-\Gamma_0(u,\overline{x},y,\tau)\Big)\Big|d\tau dud\overline{x}dy\\
&&\le C\int_{\mathbb{R}^n}|\phi(y)|\int_{|\overline{x}-\overline{y}|<\eta}\int_{-a}^{a}\int_{-\infty}^\infty |\tau|^{-1/2}
\Big|\frac{\partial}{\partial u}\Big(\Gamma(u,\overline{x},y,\tau)-\Gamma_0(u,\overline{x},y,\tau)\Big)\Big|d\tau dud\overline{x}dy,\,\,\,\eta>0.
\end{eqnarray*}

It was showed in \eqref{PI2} that, for every $K$ compact subset of $\mathbb{R}^n$, we have that
$$
\int_K\int_K\int_{-\infty}^\infty |\tau|^{-1/2}
\Big|\frac{\partial}{\partial u}\Big(\Gamma(u,\overline{x},y,\tau)-\Gamma_0(u,\overline{x},y,\tau)\Big)\Big|d\tau dydx<\infty.
$$
Then, it follows that
$$
\lim_{\eta\to 0^+}\int_{\mathbb{R}^n}|\phi(y)|\int_{|\overline{x}-\overline{y}|<\eta}\int_{-a}^{a}\int_{-\infty}^\infty |\tau|^{-1/2}
\Big|\frac{\partial}{\partial u}\Big(\Gamma(u,\overline{x},y,\tau)-\Gamma_0(u,\overline{x},y,\tau)\Big)\Big|d\tau dud\overline{x}dy=0.
$$
Hence, we conclude
$$
\lim_{\eta\to 0^+}\int_{\mathbb{R}^n}|\phi(y)|\int_{|\overline{x}-\overline{y}|<\eta}\int_{-\infty}^\infty |\tau|^{-1/2}|J_1(\overline{x},y,\varepsilon, \tau)|d\tau d\overline{x}dy=0,
$$
uniformly in $\varepsilon\in (0,1)$.

Therefore, in order to achieve our purpose it is sufficient to see that, for every $\eta>0$,
$$
\lim_{\varepsilon\to 0^+}\int_{\mathbb{R}^{n}}\int_{|\overline{x}-\overline{y}|>\eta}|\phi(y)|\int_{-\infty}^\infty |\tau|^{-1/2} |J_m(\overline{x},y,\varepsilon,\tau)|d\tau d\overline{x}dy=0,\,\,\,m=1,2.
$$
Assume that $\eta>0$. By (\ref{PI8}) and (\ref{PI9}), we get, for certain $A_1, A_2>0$,
\begin{eqnarray*}
\int_{\mathbb{R}^{n}}\int_{|\overline{x}-\overline{y}|>\eta}|\phi(y)|\int_{-\infty}^\infty |\tau|^{-1/2} |J_1(\overline{x},y,\varepsilon,\tau)|d\tau d\overline{x}dy&&\\
&\hspace{-10cm}\le& \hspace{-5cm}C\varepsilon^\delta \int_{B(0,A_1)}\int_{\eta<|\overline{x}-\overline{y}|<A_2}\int_{-\infty}^\infty \frac{|\tau|^{-1/2}|\overline{x}-\overline{y}|^{2-n-\delta}}{(1+|\tau|^{1/2}|\overline{x}-\overline{y}|)^3}d\tau d\overline{x}dy\\
&\hspace{-10cm}\le&\hspace{-5cm}C\varepsilon^\delta \int_{B(0,A_1)}\int_{\eta<|\overline{x}-\overline{y}|<A_2}|\overline{x}-\overline{y}|^{1-n-\delta} d\overline{x}dy,\,\,\,0<\varepsilon<\eta/15.
\end{eqnarray*}
Hence
$$
\lim_{\varepsilon\to 0^+}\int_{\mathbb{R}^{n}}\int_{|\overline{x}-\overline{y}|>\eta}|\phi(y)|\int_{-\infty}^\infty |\tau|^{-1/2} |J_1(\overline{x},y,\varepsilon,\tau)|d\tau d\overline{x}dy=0.
$$
Finally, in order to see that
$$
\lim_{\varepsilon\to 0^+}\int_{\mathbb{R}^{n}}\int_{|\overline{x}-\overline{y}|>\eta}|\phi(y)|\int_{-\infty}^\infty |\tau|^{-1/2} |J_2(\overline{x},y,\varepsilon,\tau)|d\tau d\overline{x}dy=0,
$$
we use the mean value theorem, (\ref{D13}) and (\ref{D14}).

Thus we have proved that $lim_{\varepsilon\to 0^+}I(\varepsilon)=0$.

Hence, coming back to (\ref{PI1}) we obtain
\begin{eqnarray*}
\int_{\mathbb R^{n-1}}\int_{\mathbb R^n}\phi(y)\int_{-\infty}^{+\infty}(-i\tau)^{-\frac{1}{2}}\int_{-\infty}^{+\infty}
\left(\Gamma(x,y,\tau)-\Gamma_0(x,y,\tau)\right) \frac{\partial}{\partial x_1}\psi(x)dx_1d\tau dy d\overline x\nonumber\\
&\hspace{-20cm}=&\hspace{-10cm}-\int_{\mathbb R^{n}}\int_{\mathbb R^n}\phi(y)\psi(x)\int_{-\infty}^{+\infty}(-i\tau)^{-\frac{1}{2}}\frac{\partial}{\partial x_1}\left(\Gamma(x,y,\tau)-\Gamma_0(x,y,\tau)\right)d\tau dy dx,
\end{eqnarray*}
and from (\ref{DD}) it follows that
$$
\Big\langle \frac{\partial}{\partial x_1}S_F, \psi \Big\rangle =-\frac{1}{2\pi}\int_{\mathbb R^n} \psi(x)\int_{\mathbb R^n}\phi(y) \int_{-\infty}^{+\infty}(-i\tau)^{-\frac{1}{2}}\frac{\partial}{\partial x_1}\left(\Gamma(x,y,\tau)-\Gamma_0(x,y,\tau)\right)d\tau dy dx.
$$

Therefore we have proved that the distributional derivative $\frac{\partial}{\partial x_1}S_F$ of $S_F$ is
$$
\frac{\partial}{\partial x_1}S_F=-\frac{1}{2\pi}\int_{\mathbb R^n}\phi(y) \int_{-\infty}^{+\infty}(-i\tau)^{-\frac{1}{2}}\frac{\partial}{\partial x_1}\left(\Gamma(x,y,\tau)-\Gamma_0(x,y,\tau)\right)d\tau dy.
$$
Moreover, the above argument shows that the right hand side in the last equality defines a locally integrable function in $\mathbb{R}^n$.

Now, by invoking now \cite[\S5, Theorem V, part (2)]{Schw} we can conclude that $\mathcal{L}^{-\frac{1}{2}}\phi - (-\Delta)^{-\frac{1}{2}}\phi$ admits classical derivative with respect to $x_1$ for almost all $x\in \mathbb R^n$, and
$$
\frac{\partial}{\partial x_1}\left(\mathcal{L}^{-\frac{1}{2}}\phi(x) - (-\Delta)^{-\frac{1}{2}}\phi(x)\right) = -\frac{1}{2\pi}\int_{\mathbb R^n}\phi(y)\int_{-\infty}^{+\infty}(-i\tau)^{-\frac{1}{2}}\frac{\partial}{\partial x_1}(\Gamma(x,y,\tau)-\Gamma_0(x,y,\tau))d\tau dy,
$$
a.e. $x \in \mathbb R^n$. Moreover, the last integral is absolutely convergent. Then,
$$
\frac{\partial}{\partial x_1}\left(\mathcal{L}^{-\frac{1}{2}}\phi(x) - (-\Delta)^{-\frac{1}{2}}\phi(x)\right) = -\frac{1}{2\pi}\lim_{\varepsilon \rightarrow 0^+}\int_{|x-y|>\varepsilon}\phi(y)\int_{-\infty}^{+\infty}(-i\tau)^{-\frac{1}{2}}\frac{\partial}{\partial x_1}(\Gamma(x,y,\tau)-\Gamma_0(x,y,\tau))d\tau dy,
$$
a.e. $x \in \mathbb R^n$, and we obtain that
$$
\frac{\partial}{\partial x_1}\mathcal{L}^{-\frac{1}{2}}\phi(x) = \lim_{\varepsilon \rightarrow 0^+}\int_{|x-y|>\varepsilon}\phi(y)R^{\mathcal{L}}_1(x,y)dy, \;\;{\rm a.e.}\;\;x \in \mathbb R^n,
$$
since, as it is well known,
$$
\frac{\partial}{\partial x_1}(-\Delta)^{-\frac{1}{2}}\phi(x) = \lim_{\varepsilon \rightarrow 0^+}\int_{|x-y|>\varepsilon}\phi(y)\int_{-\infty}^{+\infty}(-i\tau)^{-\frac{1}{2}}\frac{\partial}{\partial x_1}\Gamma_0(x,y,\tau)d\tau dy, \;\;{\rm a.e.}\;\;x \in \mathbb R^n.
$$
\end{proof}

%We define, for every $\ell =1,\cdots,n$, the maximal operator $R_\ell ^{\mathcal{L},*}$ as follows
%$$
%R_\ell ^{\mathcal{L},*}(f)= \sup_{\varepsilon >0} \left|R_\ell ^{\mathcal{L},\varepsilon}(f)\right|,\,\,\,f\in L^p(\mathbb{R}^n)\,\,\,\mbox{and}\,\,\,1\le p<\infty,
%$$
%where, for every $\varepsilon > 0$,
%$$
%R_\ell ^{\mathcal{L},\varepsilon}(f)(x)=\int_{|x-y|> \varepsilon}f(y)R_\ell ^{\mathcal{L}}(x,y) dy,\,\,\,x\in \mathbb{R}^n.
%$$

\vspace{7mm}

We now prove $L^p$-boundedness properties for the variation
operators associated with the Riesz transforms $R_\ell
^{\mathcal{L}}$.
\begin{proof}[\bf Proof of Theorem \ref{VarRiesz}.] As in the proof of Proposition 4.1 it is enough to assume $n/2<q<n$. We consider the operators
$$
R_{\ell ,{\rm loc}}^{\mathcal{L}}(f)(x)= P.V. \int_{|x-y| < \gamma(x)} R_\ell ^{\mathcal{L}}(x,y) f(y) dy
$$
and
$$
R_{\ell ,{\rm loc}}(f)(x)= P.V.\int_{|x-y| < \gamma(x)} R_\ell (x-y)f(y) dy,
$$
where $f \in L^p(\mathbb R^n)$ for a suitable $1 \leq p < \infty$.

 Suppose that $\{\eta_j\}_{j\in \mathbb{N}}$ is a real decreasing sequence that converges to zero. Following the general procedure we may write
\begin{multline*}
\left(\sum_{j=0}^\infty\left|R_\ell ^{\mathcal{L},\eta_j}(f)(x) - R_\ell ^{\mathcal{L},\eta_{j+1}}(f)(x)\right|^\rho\right)^{\frac{1}{\rho}}\\
%\shoveleft{\hspace{10mm}\leq \left(\sum_{j=0}^\infty\left|R_{\ell ,{\rm loc}}^{\mathcal{L},\eta_j}(f)(x) - R_{\ell ,{\rm loc}}^{\mathcal{L},\eta_{j+1}}(f)(x)-\left(R_{\ell ,{\rm loc}}^{\eta_j}(f)(x) - R_{\ell ,{\rm loc}}^{\eta_{j+1}}(f)(x)\right)\right|^\rho\right)^{\frac{1}{\rho}}}\\
%\shoveleft{\hspace{13mm}+ \left(\sum_{j=0}^\infty\left|R_\ell ^{\mathcal{L},\eta_j}(f)(x) - R_\ell ^{\mathcal{L},\eta_{j+1}}(f)(x)-\left(R_{\ell ,{\rm loc}}^{\mathcal{L},\eta_j}(f)(x) - R_{\ell ,{\rm loc}}^{\mathcal{L},\eta_{j+1}}(f)(x)\right)\right|^\rho\right)^{\frac{1}{\rho}}}\\
%\shoveleft{\hspace{13mm}+ \left(\sum_{j=0}^\infty\left|R_{\ell ,{\rm loc}}^{\eta_j}(f)(x)-R_{\ell ,{\rm loc}}^{\eta_{j+1}}(f)(x)\right|^\rho\right)^{\frac{1}{\rho}}}\\
\shoveleft{\hspace{10mm}\leq \left(\sum_{j=0}^\infty \left| \int_{\eta_{j+1}<|x-y| <\eta_j,\,|x-y|<\gamma(x)}\left(R_\ell ^{\mathcal{L}}(x,y)-R_\ell (x-y)\right)f(y)dy\right|^\rho\right)^{\frac{1}{\rho}}}\\
\shoveleft{\hspace{13mm}+\left(\sum_{j=0}^\infty \left|\int_{\eta_{j+1} < |x-y| <\eta_j,\,|x-y|>\gamma(x)} R_\ell ^{\mathcal{L}}(x,y) f(y)dy\right|^\rho\right)^{\frac{1}{\rho}}}\\
\shoveleft{\hspace{13mm}+ \left(\sum_{j=0}^\infty \left|\int_{\eta_{j+1} < |x-y| <\eta_j,\,|x-y|<\gamma(x)} R_\ell (x-y) f(y)dy\right|^\rho\right)^{\frac{1}{\rho}}}\\
\shoveleft{\hspace{10mm}\leq \sum_{j=0}^\infty \int_{\eta_{j+1} <|x-y|< \eta_j,\,|x-y|<\gamma(x)}|R_\ell ^{\mathcal{L}}(x,y)-R_\ell (x-y)||f(y)|dy }\\
\shoveleft{\hspace{13mm}+ \sum_{j=0}^\infty \int_{\eta_{j+1}<|x-y|< \eta_j,\,|x-y|>\gamma(x)}|R_\ell ^{\mathcal{L}}(x,y)|f(y)|dy}\\
\shoveleft{\hspace{13mm}+\left(\sum_{j=0}^\infty\left|\int_{\eta_{j+1} <|x-y| <\eta_j,\,|x-y|<\gamma(x)}R_\ell (x-y) f(y) dy \right|^\rho\right)^{\frac{1}{\rho}}} \\
\shoveleft{\hspace{10mm}\leq \int_{|x-y| < \gamma(x)}|R_\ell ^{\mathcal{L}}(x,y)-R_\ell (x-y)||f(y)|dy + \int_{|x-y| > \gamma(x)}|R_\ell ^{\mathcal{L}}(x,y)||f(y)|dy}\\
+\sup_{\{t_j\}_{j\in \mathbb{N}}\searrow 0}\left(\sum_{j=0}^\infty \left|\int_{t_{j+1}< |x-y| <t_j}R_\ell (x-y)f(y)dy\right|^\rho\right)^{\frac{1}{\rho}}.\hspace{42mm}
\end{multline*}
Hence, we get
\begin{eqnarray}\label{For2}
V_\rho(R_\ell ^{\mathcal{L},\varepsilon})(f)(x) & \leq &\int_{|x-y| < \gamma(x)}\left|R_\ell ^{\mathcal{L}}(x,y) - R_\ell (x-y)\right||f(y)|dy\nonumber \\
&&+\int_{|x-y|> \gamma(x)}\left|R_\ell ^{\mathcal{L}}(x,y)\right||f(y)|dy
+ V_\rho(R_\ell ^\varepsilon)(f)(x)\nonumber\\
&=&\tau_1(|f|)(x)+\tau_2(|f|)(x)+V_\rho(R_\ell ^\varepsilon)(f)(x), \;\;\;
x \in \mathbb R^n.
\end{eqnarray}

Note that the operators $\tau_1$ and $\tau_2$ are the same ones that appeared in the proof of Proposition 4.1. Then, as it was proved there, $\tau_1$ and $\tau_2$ are bounded from $L^r(\mathbb{R}^n)$ into itself provided that $1\le r<p_0$.

By (\ref{For2}) and \cite[Theorem A and Corollary 1.4]{CJRW2} we
conclude the desired $L^p$-boundedness properties for $V_\rho(R_\ell
^{\mathcal{L},\varepsilon})$.

\end{proof}

\section{Variation operators associated with commutators $C_b^\mathcal{L}$}

Proposition \ref{PvComm} and Theorem \ref{VarComm} are proved in
this section. Assume that $b\in BMO_\theta(\gamma)$, for some
$\theta\ge 0$. Let us remind that, for every $\ell =1,\cdots,n$, the
commutator operator $C_{b,\ell }^{\mathcal{L}}$ for the Riesz
transform $R_\ell ^\mathcal{L}$  is given by
$$
C_{b,\ell }^{\mathcal{L}}f = b R_\ell ^{\mathcal{L}}f- R_\ell ^{\mathcal{L}}(bf),\;\;\; f \in C_c^\infty(\mathbb R^n).
$$
Some $L^p$-boundedness results for these operators were established in \cite[Theorem 1]{BHS3}.

From Proposition \ref{Pvriesz} we deduce that, for every  $f \in
C_c^\infty(\mathbb R^n)$,
$$
C_{b,\ell }^{\mathcal{L}}(f)(x) = \lim_{\varepsilon \rightarrow 0^+}\int_{|x-y|> \varepsilon} (b(x)-b(y))R_\ell ^{\mathcal{L}}(x,y) f(y) dy, \;\; x \in \mathbb R^n.
$$
In Proposition \ref{PvComm}, that is proved in the following, we extend this property for every $f\in L^p(\mathbb{R}^n)$, $1<p<p_0$, where $\frac{1}{p_0}=\Big(\frac{1}{q}-\frac{1}{n}\Big)_+$ and $V\in B_q$, $q\ge n/2$.

\vspace{5mm}

\begin{proof}[\bf Proof of Proposition \ref{PvComm}.]

It is enough to prove that the maximal operator $C_{b,\ell }^{\mathcal{L},*}$ defined by
$$
C_{b,\ell }^{\mathcal{L},*}(f) (x)=\sup_{\varepsilon >0}\left|\int_{|x-y|>\varepsilon }(b(x)-b(y))R_\ell ^{\mathcal{L}}(x,y)f(y)dy\right|, \quad x\in \mathbb{R}^n,
$$
is bounded from $L^p(\mathbb{R}^n)$ into itself when $V$ and $p$ satisfy the conditions specified in this proposition.

  As in the proof of Theorem \ref{VarRiesz} it suffices to take care of the case $n/2< q< n$. Suppose that $f\in L^p(\mathbb{R}^n)$, where $1<p<p_0$.

Let us consider the local operators
$$
C_{b,\ell}^{\mathcal{L},*,{\rm loc}}(f)(x)=\sup_{\varepsilon
>0}\left|\int_{\varepsilon<|x-y|<\gamma(x) }(b(x)-b(y))R_\ell
^{\mathcal{L}}(x,y)f(y)dy\right|, \quad x\in \mathbb{R}^n,
$$
and
$$
C_{b,\ell}^{*,{\rm loc}}(f)(x)=\sup_{\varepsilon
>0}\left|\int_{\varepsilon<|x-y|<\gamma(x) }(b(x)-b(y))R_\ell
(x-y)f(y)dy\right|, \quad x\in \mathbb{R}^n.
$$

With this notation we have that
\begin{eqnarray}\label{Co0}
C_{b,\ell }^{\mathcal{L},* }(f)(x)&=&C_{b,\ell}^{\mathcal{L},*,{\rm
loc}}(f)(x)-C_{b,\ell }^{*,{\rm loc}}(f)(x)+
C_{b,\ell }^{\mathcal{L}, *}(f)(x)-C_{b,\ell }^{\mathcal{L}, *,{\rm loc}}(f)(x)+C_{b,\ell }^{*,{\rm loc}}(f)(x)\nonumber\\
&\leq&\int_{|x-y|<\gamma(x)}|b(x)-b(y)||R_\ell ^{\mathcal{L}}(x,y)-R_\ell (x-y)||f(y)|dy\nonumber\\
&+&\int_{|x-y|>\gamma(x)}|b(x)-b(y)||R_\ell ^{\mathcal{L}}(x,y)||f(y)|dy+C_{b,\ell }^{*,{\rm loc}}(f)(x)\nonumber\\
&=&T_1(|f|)(x)+T_2(|f|)(x)+C_{b,\ell }^{*,\rm{loc}}(f)(x), \quad
x\in \mathbb{R}^n.
\end{eqnarray}

%It is wellknown that the maximal operator $C_{b,\ell }^*$ is bounded from $L^r(\mathbb{R}^n)$ into itself, for every $1<r<\infty$.

According to Proposition \ref{master} the operator $C_{b,\ell
}^{*,\rm{loc}}$ is bounded from $L^r(\mathbb{R}^n)$ into itself, for
every $1<r<\infty$.

We now analyze the $L^p$-boundedness properties for the operators $T_1$ and $T_2$ studying the behavior of their adjoints $T_1^*$ and $T_2^*$.

The operator $T _1^*$ adjoint of $T_1$ is defined by
$$
T_1^*(g)(y)=\int_{|x-y |<\gamma(x)}|b(y)-b(x)||R_\ell
^\mathcal{L}(x,y)-R_\ell (x-y)|g(x)dx.
$$
According to (\ref{L4}), since $\gamma(x)\sim\gamma(y)$ when
$|x-y|\le \gamma(x)$ there exists $A>0$ for which
\begin{eqnarray}\label{Co11}
|T_1^*(g)(y)|&\leq&C\left(\int_{|x-y|<A\gamma(y)}\frac{1}{|x-y|^n}\left(\frac{|x-y|}{\gamma(y)}\right)^{2-n/q}|b(y)-b(x)||g(x)|dx\right.\nonumber\\
%&+&\int_{\gamma(y)\le |x-y|<A\gamma(y)}\frac{1}{|x-y|^n}\left(\frac{|x-y|}{\gamma(y)}\right)^{2-n/q}|b(y)-b(x)||g(x)|dx\nonumber\\
&+&\int_{|x-y|<A\gamma(y)}\frac{1}{|x-y|^{n-1}}\int_{B(x,\frac{|x-y|}{4})}\frac{V(z)}{|z-x|^{n-1}}dz|b(y)-b(x)||g(x)|dx\Big)
\nonumber\\
%&+&\left.\int_{\gamma(y)\le
%|x-y|<A\gamma(y)}\frac{1}{|x-y|^{n-1}}\int_{B(x,\frac{|x-y|}{4})}\frac{V(z)}{|z-x|^{n-1}}dz|b(y)-b(x)||g(x)|dx
%\right)\nonumber\\
&=&C(T_{1,1}^*(|g|)(y)+T_{1,2}^*(|g|)(y)),\quad y\in \mathbb{R}^n.
\end{eqnarray}
%By proceeding as in the proof of (\ref{Co9}) we get
%\begin{equation}\label{Co12}
%||T _{1,1}^*(g)||_{L^{p'}(\mathbb{R}^n)}\leq
%C\|b\|_{\rm{BMO}_\theta(\gamma)}||g||_{L^{p'}(\mathbb{R}^n)}, \quad
%g\in L^{p'}(\mathbb{R}^n).
%\end{equation}

We have that
$$
T_{1,1}^*(|g|)(y)\leq C\sum_{j=0}^\infty \frac{2^{-j\delta
}}{(2^{-j}\gamma(y))^n}\int_{2^{-j-1}A\gamma(y) \leq |x-y|\leq
2^{-j}A\gamma(y)}|b(x)-b(y)||g(x)|dx,\quad y\in \mathbb{R}^n,
$$
where $\delta =2-\frac{n}{q}>0$.

Then,
\begin{equation}\label{Co5}
||T_{1,1}^*(|g|)||_{L^{p'}(\mathbb{R}^n)}\leq C\sum_{j=0}^\infty
||T_{1,1,j}^*(f)||_{L^{p'}(\mathbb{R}^n)},
\end{equation}
where, for every $j\in \mathbb{N}$,
$$
T_{1,1,j}^*(g)(y)=\frac{2^{-j\delta}}{(2^{-j}\gamma(y))^n}\int_{|x-y|\leq
2^{-j}A\gamma(y)}|b(x)-b(y)||g(x)|dx,\quad y\in \mathbb{R}^n.
$$

To deal with these operators we consider the covering $\{Q_k\}_{k\in \mathbb{N}}$ as given in Section 2. We know that there exists $C>0$ such that $\gamma(y)\le C\gamma(x_k)$, for every $y\in 2Q_k$. We choose the smaller $L\in \mathbb{N}$ such that $AC+1\le 2^L$. It is not hard to see that we can find $M\in \mathbb{N}$ such that, for every $k,j\in \mathbb{N}$, there exist $N_j\in \mathbb{N}$ and $x_{k,j}^i$, $i=1,...,N_j$, such that, by denoting $Q_{k,j}^i=B(x_{k,j}^i, 2^{-j}\gamma (x_k))$, $i=1,...,N_j$, the following properties hold:

(i) $Q_k\subset \cup_{i=1}^{N_j}Q_{k,j}^i\subset 2Q_k$;

(ii) $\mbox{card }\{l\in \mathbb{N}: 2^LQ_{k,j}^i\cap 2^LQ_{k,j}^l\not=\varnothing \}\leq M$, $i=1,...,N_j$.

Clearly, we have that  $B(y,2^{-j}A\gamma(y))\subset 2^LQ_{k,j}^i=\widetilde{Q_{k,j}^i}$, when $y\in Q_{k,j}^i$, $j,k\in \mathbb{N}$ and $i=1,...,N_j$.

We can write, for every $j\in \mathbb{N}$,
\begin{eqnarray}\label{Co6}
\int_{\mathbb{R}^n}|T_{1,1,j}^*(g)(y)|^{p'}dy&\leq& C\sum_{k=0}^\infty \int_{Q_k}\left(\frac{2^{-j\delta}}{(2^{-j}\gamma(y))^n}\int_{|x-y|\leq 2^{-j}A\gamma(y)}|g(x)||b(x)-b(y)|dx\right)^{p'}dy\nonumber\\
&\leq&C\sum_{k=0}^\infty \sum_{i=1}^{N_j}\int_{Q_{k,j}^i}\left(\frac{2^{-j\delta}}{(2^{-j}\gamma (x_k))^n}\int_{\widetilde{Q_{k,j}^i}}|g(x)||b(x)-b(y)|dx\right)^{p'}dy\nonumber\\
&\leq&C\left(\sum_{k=0}^\infty \sum_{i=1}^{N_j}\frac{2^{-j\delta p'}}{(2^{-j}\gamma (x_k))^{np'}}\int_{Q_{k,j}^i}|b(y)-b_{Q_{k,j}^i}|^{p'}\left(\int_{\widetilde{Q_{k,j}^i}}|g(x)|dx\right)^{p'}dy\right.\nonumber\\
&&+\left.\sum_{k=0}^\infty \sum_{i=1}^{N_j}\frac{2^{-j\delta p'}}{(2^{-j}\gamma (x_k))^{np'}}\int_{Q_{k,j}^i}\left(\int_{\widetilde{Q_{k,j}^i}}|b(x)-b_{Q_{k,j}^i}||g(x)|dx\right)^{p'}dy\right).
\end{eqnarray}

Each summand is estimated separately. For the first one, since
$\gamma(x_{k,j}^i)\sim \gamma(x_k)$, by \cite[Proposition 3]{BHS3},
we have that
\begin{multline}\label{Co7}
\sum_{j=0}^\infty \left(\sum_{k=0}^\infty \sum_{i=1}^{N_j}
\frac{2^{-j\delta p'}}{(2^{-j}\gamma (x_k))^{np'}}\int_{Q_{k,j}^i}|b(y)-b_{Q_{k,j}^i}|^{p'}\left(\int_{\widetilde{Q_{k,j}^i}}|g(x)|dx\right)^{p'}dy\right)^{1/{p'}}\\
\shoveleft{\hspace{10mm}\leq C\sum_{j=0}^\infty 2^{-j\delta}\left(\sum_{k=0}^\infty \sum_{i=1}^{N_j}
\frac{1}{(2^{-j}\gamma (x_k))^{n{p'}}}\int_{Q_{k,j}^i}|b(y)-b_{Q_{k,j}^i}|^{p'}dy\right.}\\
\shoveleft{\hspace{13mm}\left.\times\int_{\widetilde{Q_{k,j}^i}}|g(x)|^{p'}dx(2^{-j}\gamma (x_k))^{np'/p}\right)^{1/p'}}\\
\shoveleft{\hspace{10mm}\leq C\sum_{j=0}^\infty 2^{-j\delta}\left(\sum_{k=0}^\infty \sum_{i=1}^{N_j}
\frac{1}{(2^{-j}\gamma (x_k))^n}\int_{Q_{k,j}^i}|b(y)-b_{Q_{k,j}^i}|^{p'}dy\int_{\widetilde{Q_{k,j}^i}}|g(x)|^{p'}dx\right)^{1/{p'}}}\\
\shoveleft{\hspace{10mm}\leq C||b||_{{\rm BMO}_\theta(\gamma)}\sum_{j=0}^\infty 2^{-j\delta}\left(\sum_{k=0}^\infty \int_{2^{L+1}Q_k}|g(x)|^{p'}dx\right)^{1/{p'}}}\\
\leq C||b||_{{\rm
BMO}_\theta(\gamma)}||g||_{L^{p'}(\mathbb{R}^n)}.\hspace{82mm}
\end{multline}

Also, by using again \cite[Proposition 3]{BHS3},
\begin{multline}\label{Co8}
\sum_{j=0}^\infty \left(\sum_{k=0}^\infty \sum_{i=1}^{N_j}
\frac{2^{-j\delta p'}}{(2^{-j}\gamma (x_k))^{np'}}\int_{Q_{k,j}^i}\left(\int_{\widetilde{Q_{k,j}^i}}|b(x)-b_{Q_{k,j}^i}||g(x)|dx\right)^{p'}dy
\right)^{1/p'}\\
\shoveleft{\leq C\sum_{j=0}^\infty 2^{-j\delta}\left(\sum_{k=0}^\infty \sum_{i=1}^{N_j}
\frac{1}{(2^{-j}\gamma (x_k))^{np'-n}}\left(\int_{Q_{k,j}^i}|b(x)-b_{Q_{k,j}^i}|^{p}dx\right)^{p'/p}
\int_{\widetilde{Q_{k,j}^i}}|g(x)|^{p'}dx\right)^{1/p'}}\\
\shoveleft{\leq C||b||_{{\rm BMO}_\theta(\gamma)} \sum_{j=0}^\infty
2^{-j\delta}\left(\sum_{k=0}^\infty \sum_{i=1}^{N_j}
\int_{\widetilde{Q_{k,j}^i}}|g(x)|^{p'}dx\right)^{1/p'}}\\
\leq C||b||_{{\rm
BMO}_\theta(\gamma)}||g||_{L^{p'}(\mathbb{R}^n)}.\hspace{92mm}
\end{multline}
By combining (\ref{Co5}), (\ref{Co6}), (\ref{Co7}) and (\ref{Co8}), we obtain
\begin{equation}\label{Co9}
||T_{1,1}^*(|g|)||_{L^{p'}(\mathbb{R}^n)}\leq C||b||_{{\rm
BMO}_\theta(\gamma)}||g||_{L^{p'}(\mathbb{R}^n)}.
\end{equation}

%We have that
%$$
%|T_{1,2}^*(g)(y)|\le
%C\frac{1}{\gamma(y)^n}\int_{|x-y|<A\gamma(y)}|b(x)-b(y)||g(x)|dx,\,\,\,y\in
%\mathbb{R}^n.
%$$
%Then, by Lemma \ref{add1} we deduce
%\begin{equation}\label{Co121}
%||T _{1,2}^*(g)||_{L^{p'}(\mathbb{R}^n)}\leq
%C\|b\|_{\rm{BMO}_\theta(\gamma)}||g||_{L^{p'}(\mathbb{R}^n)}, \quad
%g\in L^{p'}(\mathbb{R}^n).
%\end{equation}

 On the other hand, we have
\begin{eqnarray}\label{Co13}
|T _{1,2}^*(g)(y)|&\leq&\sum_{j=0}^\infty \int_{2^{-j-1}A\gamma(y)\leq |x-y|<2^{-j}A\gamma(y)}|g(x)|\frac{|b(y)-b(x)|}{|x-y|^{n-1}}\int_{B(x,\frac{|x-y|}{4})}\frac{V(z)}{|z-x|^{n-1}}dzdx\nonumber\\
&\leq&C\sum_{j=0}^\infty \frac{1}{(2^{-j}\gamma(y))^{n-1}}\int_{|x-y|\leq 2^{-j}A\gamma(y)}\int_{B(y,2^{-j+1}A\gamma(y))}\frac{V(z)}{|z-x|^{n-1}}dz|b(y)-b(x)||g(x)|dx.
\end{eqnarray}

Since $\frac{1}{p_0}=\frac{1}{q}-\frac{1}{n}$, by using \cite[Lemma
1]{BHS2}, we get
\begin{eqnarray*}
||I_1(\chi _{B(y,2^{-j+1}A\gamma(y))}V)||_{L^{p_0}(\mathbb{R}^n)}&=&\left(\int_{\mathbb{R}^n}\left|\int_{B(y,2^{-j+1}A\gamma(y))}\frac{V(z)}{|z-x|^{n-1}}dz\right|^{p_0}dx\right)^{1/p_0}\\
&\hspace{-10cm}\leq&\hspace{-5cm}C\left(\int_{B(y,2^{-j+1}A\gamma(y))}|V(z)|^qdz\right)^{1/q}\\
&\hspace{-10cm}\leq&\hspace{-5cm}C(2^{-j}\gamma(y))^{n(-1+1/q)}\int_{B(y,2^{-j+1}A\gamma(y))}V(z)dz\\
&\hspace{-10cm}\leq&\hspace{-5cm}C(2^{-j}\gamma(y))^{n(-1+1/q)}2^{-j(2-n/q)}(2^{-j}\gamma(y))^{n-2}\\
&\hspace{-10cm}\leq&\hspace{-5cm}C\gamma(y)^{-2+n/q} ,\quad y\in
\mathbb{R}^n.
\end{eqnarray*}
Then, H\"older inequality implies that
\begin{eqnarray*}
|T _{1,2}^*(g)(y)|&\leq&C\sum_{j=0}^\infty \frac{1}{(2^{-j}\gamma(y))^{n-1}}\left(\int_{|x-y|\leq 2^{-j}A\gamma(y)}(|b(y)-b(x)||g(x)|)^{p_0'}dx\right)^{1/p_0'}\\
&&\times||I_1(\chi _{B(y,2^{-j+1}A\gamma(y))}V)||_{L^{p_0}(\mathbb{R}^n)}\\
&&\le C\sum_{j=0}^\infty
\frac{1}{(2^{-j}\gamma(y))^{n-1}}\left(\int_{|x-y|\leq
2^{-j}A\gamma(y)}(|b(x)-b(y)||g(x)|)^{p_0'}dx\right)^{1/p_0'}\gamma(y)^{n/q-2},\quad
y\in \mathbb{R}^n.
\end{eqnarray*}

We can write,
\begin{equation}\label{Co14}
||T _{1,2}^*(g)||_{L^{p'}(\mathbb{R}^n)}\leq C\sum_{j=0}^\infty ||T
_{1,2,j}^*(g)||_{L^{p'}(\mathbb{R}^n)},
\end{equation}
where, for every $j\in \mathbb{N}$,
$$
T_{1,2,j}^*(g)(y)=\frac{\gamma(y)^{n/q-2}}{(2^{-j}\gamma(y))^{n-1}}\left(\int_{|x-y|\leq
2^{-j}A\gamma(y)}(|b(x)-b(y)||g(x)|)^{p_0'}dx\right)^{1/p_0'},\quad
y\in \mathbb{R}^n.
$$
As before we have, for every $j\in \mathbb{N}$,
\begin{eqnarray}\label{Co15}
\int_{\mathbb{R}^n}|T_{1,2,j}^*(g)(y)|^{p'}dy\nonumber\\
&\hspace{-2cm}\leq&\hspace{-1cm}C\left(\sum_{k=0}^\infty \sum_{i=1}^{N_j}\frac{2^{-j\delta p'}}{(2^{-j}\gamma (x_k))^{(n-n/q+1)p'}}\int_{Q_{k,j}^i}|b(y)-b_{Q_{k,j}^i}|^{p'}\left(\int_{\widetilde{Q_{k,j}^i}}|g(x)|^{p_0'}dx
\right)^{p'/p_0'}dy\right.\nonumber\\
&\hspace{-1.5cm}+&\hspace{-0.5cm}\left.\sum_{k=0}^\infty \sum_{i=1}^{N_j}
\frac{2^{-j\delta p'}}{(2^{-j}\gamma(x_k))^{(n-n/q+1)p'}}
\int_{Q_{k,j}^i}\left(\int_{\widetilde{Q_{k,j}^i}}(|b(y)-b_{Q_{k,j}^i}||g(y)|)^{p_0'}dy
\right)^{p'/p_0'}dx\right),
\end{eqnarray}
where $\delta=2-n/q$. Here $Q_k$, $k\in \mathbb{N}$, and $Q_{k,j}^i$, and $\widetilde{Q}_{k,j}^i$ $k,j\in \mathbb{N}$, $i=1,...,N_j$, are the same balls that we considered above.

Also, since $p'> p_0'$, we get
\begin{multline}\label{Co16}
\sum_{j=0}^\infty \left(\sum_{k=0}^\infty \sum_{i=1}^{N_j}
\frac{2^{-j\delta p'}}{(2^{-j}\gamma(x_k))^{(n/q'+1)p'}}\int_{Q_{k,j}^i}|b(y)-b_{Q_{k,j}^i}|^{p'}
\left(\int_{\widetilde{Q_{k,j}^i}}|g(x)|^{p_0'}dx\right)^{p'/p_0'}dy\right)^{1/p'}\\
\shoveleft{\hspace{30mm}\leq C\sum_{j=0}^\infty 2^{-j\delta}
\left(\sum_{k=0}^\infty \sum_{i=1}^{N_j}
\frac{1}{(2^{-j}\gamma (x_k))^{(n/q'+1)p'}}\int_{Q_{k,j}^i}|b(y)-b_{Q_{k,j}^i}|^{p'}
\right.dy}\\
\shoveleft{\hspace{33mm}\times\left.\int_{\widetilde{Q_{k,j}^i}}|g(x)|^{p'}dx(2^{-j}\gamma (x_k))^{\frac{np'}{p_0'}\Big(1-\frac{p_0'}{p'}\Big)}\right)^{1/p'}}\\
\shoveleft{\hspace{30mm}\leq C||b||_{{\rm
BMO}_\theta(\gamma)}\sum_{j=0}^\infty 2^{-j\delta}
\left(\sum_{k=0}^\infty \sum_{i=1}^{N_j}\int_{\widetilde{Q_{k,j}^i}}|g(x)|^{p'}dx\right)^{1/p'}}\\
\leq C||b||_{{\rm
BMO}_\theta(\gamma)}||g||_{L^{p'}(\mathbb{R}^n)},\hspace{61mm}
\end{multline}
and
\begin{multline}\label{Co17}
\sum_{j=0}^\infty \left(\sum_{k=0}^\infty \sum_{i=1}^{N_j}
\frac{2^{-j\delta p'}}{(2^{-j}\gamma (x_k))^{(n/q'+1)p'}}\int_{Q_{k,j}^i}\left(\int_{\widetilde{Q_{k,j}^i}}
(|b(x)-b_{Q_{k,j}^i}||g(x)|)^{p_0'}dx\right)^{p'/p_0'}dy\right)^{1/p'}\\
\shoveleft{\hspace{30mm}\leq C\sum_{j=0}^\infty 2^{-j\delta}
\left(\sum_{k=0}^\infty \sum_{i=1}^{N_j}
\frac{1}{(2^{-j}\gamma (x_k))^{(n/q'+1)p'-n}}\right.}\\
\shoveleft{\hspace{33mm} \left.\left.\times\left (\int_{\widetilde{Q_{k,j}^i}}|b(x)-b_{Q_{k,j}^i}|^{p_0'p'/(p'-p_0')}dx\right)^{(p'-p_0')/p_0'}\int_{\widetilde{Q_{k,j}^i}}|g(x)|^{p'}dx\right)^{1/p'}\right.}\\
\shoveleft{\hspace{30mm}\leq C||b||_{{\rm
BMO}_\theta(\gamma)}\sum_{j=0}^\infty 2^{-j\delta}
\left(\sum_{k=0}^\infty \sum_{i=1}^{N_j}\int_{\widetilde{Q}_{k,j}^i}|g(x)|^{p'}dx\right)^{1/p'}}\\
\leq C||b||_{{\rm
BMO}_\theta(\gamma)}||g||_{L^{p'}(\mathbb{R}^n)}.\hspace{61mm}
\end{multline}

From (\ref{Co14}), (\ref{Co15}), (\ref{Co16}) and (\ref{Co17}) we
conclude that
\begin{equation}\label{Co18}
||T _{1,2}^*(g)||_{L^{p'}(\mathbb{R}^n)}\leq C||b||_{{\rm
BMO}_\theta(\gamma)}||g||_{L^{p'} (\mathbb{R}^n)}, \quad g\in
L^{p'}(\mathbb{R}^n).
\end{equation}

%Also, by arguing as in the proof of Lemma \ref{add1}, we can write
%\begin{eqnarray*}
%&&\int_{Q_k}|T_{1,4}^*(|g|)(y)|^{p'}\le
%C\int_{Q_k}\Big(\frac{1}{\gamma(y)^{n-1}}\int_{|x-y|<A\gamma(y)}\int_{B(y,\frac{5}{4}A\gamma(y))}\frac{V(z)}{|z-x|^{n-1}}dz|b(y)-b(x)||g(x)|dx\Big)^{p'}dy\\
%&\le&C\int_{Q_k}\Big(\frac{1}{\gamma(y)^{n+1-n/q}}\Big(\int_{|x-y|<A\gamma(y)}\Big(|b(y)-b(x)||g(x)|\Big)^{p_0'}dx\Big)^{1/p_0'}\Big)^{p'}dy\\
%&\le&
%C\int_{Q_k}\Big(\frac{1}{\gamma(y)^{n/p_0'}}\Big(\int_{|x-y|<A\gamma(y)}|g(x)|^{p'}dx\Big)^{1/p'}\Big(\int_{|x-y|<A\gamma(y)}|b(y)-b(x)|^{p'p_0'/(p'-p_0')}dx\Big)^{(p'-p_0')/p'p_0'}\Big)^{p'}dy\\
%&\le&
%C\|b\|_{BMO_\theta(\gamma)}^{p'}\int_{|x_k-y|<A\gamma(x_k)}|g(x)|^{p'}dx.
%\end{eqnarray*}
%Here $A>0$ can change in each occurrence but it does not depend on
%$k$. Hence, we conclude that
%\begin{equation}\label{Co181}
%\|T_{1,4}^*(g)\|_{L^{p'}(\mathbb{R}^n)}\leq C||b||_{{\rm
%BMO}_\theta(\gamma)}||g||_{L^{p'} (\mathbb{R}^n)}, \quad g\in
%L^{p'}(\mathbb{R}^n).
%\end{equation}

By invoking (\ref{Co11}), (\ref{Co9}), and (\ref{Co18})
 it follows that
$$
||T _1^*g||_{L^{p'}(\mathbb{R}^n)}\leq C||b||_{{\rm
BMO}_\theta(\gamma)}||g||_{L^{p'} (\mathbb{R}^n)}, \quad g\in
L^{p'}(\mathbb{R}^n).
$$
Then, $T _1$ is a bounded operator from $L^p(\mathbb{R}^n)$ into
itself.

The operator $T _2^*$ adjoint of $T _2$ is defined by
$$
T_2^*(g)(y)=\int_{|x-y|\ge \gamma(x)}|b(x)-b(y)||R_\ell
^{\mathcal{L}}(x,y)|g(x)dx,\quad y\in \mathbb{R}^n.
$$
According to (\ref{L3}), since for a certain $A>0$, $|x-y|\ge
A\gamma(y)$, when $|x-y|\ge \gamma(x)$, we have that
\begin{eqnarray}\label{Co19}
|T _2^*(g)(y)|&\leq&C_\alpha\left(\int_{|x-y|\ge A\gamma(y)}\frac{1}{|x-y|^n}\frac{1}{\Big(1+\frac{|x-y|}{\gamma(y)}\Big)^\alpha}|b(x)-b(y)||g(x)|dx\right.\nonumber\\
&&+\left.\int_{|x-y|\ge A\gamma(y)}\frac{1}{|x-y|^{n-1}}\frac{|b(x)-b(y)||g(x)|}{\Big(1+\frac{|x-y|}{\gamma(y)}\Big)^\alpha}\int_{B(x,\frac{|x-y|}{4})}\frac{V(z)}{|z-x|^{n-1}}dzdx\right)\nonumber\\
&=&C_\alpha(T_{2,1}^*(|g|)(y)+T _{2,2}^*(|g|)(y)),\quad y\in \mathbb{R}^n .
\end{eqnarray}
Here $\alpha >0$ will be sufficiently large  and it will be fixed
later.

By choosing $\alpha>\theta'+1$, where $\theta'$ is the one appearing in Lemma \ref{add1}, we can prove that
\begin{eqnarray*}
|T_{2,1}^*(|g|)(y)|
&\le&C\sum_{m=0}^\infty \int_{2^mA\gamma(y)<|x-y|\le
2^{m+1}A\gamma(y)}|b(x)-b(y)|\frac{1}{|x-y|^n}\Big(1+\frac{|x-y|}{\gamma(y)}\Big)^{-(\theta'+1)}|g(x)|dx\\
&\le&C\sum_{m=0}^\infty
\frac{1}{2^{m(\theta'+1)}}\frac{1}{(2^m\gamma(y))^n}\int_{|x-y|\le
2^{m+1}A\gamma(y)}|b(x)-b(y)||g(x)|dx\\
&\le& C\sum_{m=0}^\infty
\frac{1}{2^{m(\theta'+1+n)}}T_{b,2^{m+1}A}(|g|)(y),\,\,\,y\in
\mathbb{R}^n,
\end{eqnarray*}
where the operators $T_{b,M}$ are the ones introduced in Lemma \ref{add1}.

Then, by using Lemma \ref{add1} we get
\begin{eqnarray}\label{Co20}
\|T_{2,1}^*(|g|)\|_{L^{p'}(\mathbb{R}^n)}&\le&
C\sum_{m=0}^\infty\frac{1}{2^{m(\theta'+1+n)}}\|T_{b,2^{m+1}A}(|g|)\|_{L^{p'}(\mathbb{R}^n)}\nonumber\\
&\le&C\|b\|_{BMO_\theta(\gamma)}\|f\|_{L^{p'}(\mathbb{R}^n)},
\,\,\,f\in L_{p'}(\mathbb{R}^n).
\end{eqnarray}

On the other hand, since $\gamma(y)\sim \gamma(x_k)$, $y\in Q_k$ and $k\in \mathbb{N}$, we have, for a certain $\beta >0$,
\begin{eqnarray}
||T_{2,2}^*g||_{L^{p'}(\mathbb{R}^n)}^{p'} &\leq& C\sum_{k=0}^\infty \int_{Q_k}\left(\int_{|x-y|\ge\beta\gamma(x_k)}\frac{|b(x)-b(y)||g(x)|}{|x-y|^{n-1}}\right.\\
&&\left.\times\frac{1}{\Big(1+\frac{|x-y|}{\gamma(x_k)}\Big)^\alpha}\int_{B(x,\frac{|x-y|}{4})}\frac{V(z)dz}{|z-x|^{n-1}}dx\right)^{p'}dy\nonumber\\
%\shoveleft{\leq C\left(\sum_{k=0}^\infty \Big(\int_{Q_k}\left(\int_{\beta \gamma (x_k)<|x-y|\leq 2\gamma (x_k)}\frac{|b(x)-b(y)||g(x)|}{|x-y|^{n-1}}\right.\right.}\\
%\shoveleft{\hspace{30mm}\left.\times\frac{1}{(1+\frac{|x-y|}{\gamma (x_k)})^\alpha}\int_{B(x,\frac{|x-y|}{4})}\frac{V(z)dz}{|z-x|^{n-1}}dx\right)^{p'}dy}\\
%\shoveleft{\hspace{3mm}+ \left.\left.\int_{Q_k}\left|\int_{|x-y|>2\gamma (x_k)}\frac{|b(x)-b(y)||g(x)|}{|x-y|^{n-1}}\frac{1}{(1+\frac{|x-y|}{\gamma (x_k)})^\alpha}\int_{B(x,\frac{|x-y|}{4})}\frac{V(z)dz}{|z-x|^{n-1}}dx\right|^{p'}dy\right)\right)^{1/p'}}\\
&=& \sum_{k=0}^\infty
\int_{Q_k}J_k(y)^{p'}dy.\nonumber
\end{eqnarray}
%Let $k\in \mathbb{N}$. Since $|x_k-x|>\gamma (x_k)$ and $|x_k-x|\leq 2|x-y|$, when $y\in Q_k$ and $|x-y|>2\gamma (x_k)$, it follows that
 Let $k\in\mathbb{N}$. It follows that
\begin{eqnarray*}
J_k(y)
%\int_{|x-y|>A\gamma(y)}\frac{|b(x)-b(y)||g(x)|}{|x-y|^{n-1}}
%\frac{1}{\Big(1+\frac{|x-y|}{\gamma(y)}\Big)^\alpha}\int_{B(x,\frac{|x-y|}{4})}\frac{V(z)dz}{|z-x|^{n-1}}dx\\
%&\leq&C\int_{|x-y|>\beta\gamma (x_k)}\frac{|b(x)-b(y)||g(x)|}{|x-y|^{n-1}}\frac{1}{\Big(1+\frac{|x-y|}{\gamma (x_k)}\Big)^\alpha}\int_{B(x,\frac{|x-y|}{4})}\frac{V(z)dz}{|z-x|^{n-1}}dx\\
&\leq&C\sum_{j=0}^\infty \int_{2^j\beta\gamma (x_k)<|x-y|\le 2^{j+1}\beta\gamma (x_k)}\frac{|b(x)-b(y)||g(x)|}{(2^j\gamma (x_k))^{n-1}}\frac{1}{2^{j\alpha}}\int_{B(y,\beta2^{j+2}\gamma (x_k))}\frac{V(z)dz}{|z-x|^{n-1}}dx\\
&\leq&C\sum_{j=0}^\infty \frac{1}{2^{j\alpha }(2^j\gamma (x_k))^{n-1}}\left(\int_{\widetilde{Q}_j^k}(|b(x)-b(y)||g(x)|)^{p_0'}dx\right)^{1/p_0'}\\
&&\times\left(\int_{\mathbb{R}^n}\left(\int_{\widetilde{Q}_j^k}\frac{V(z)dz}{|z-x|^{n-1}}\right)^{p_0}dx\right)^{1/p_0},\,\,\,y\in
Q_k,
\end{eqnarray*}
where $p_0$ is such that $\frac{1}{p_0}=\frac{1}{q}-\frac{1}{n}$, and for every $j\in \mathbb{N}$ we consider $\widetilde{Q}_j^k=B(x_k,c2^j\gamma(x_k))$, where $c>0$ is independent of $j,k\in \mathbb{N}$ and such that $B(y,\beta 2^{j+2}\gamma(x_k))\subset \widetilde{Q}_j^k$, $y\in Q_k$.

Since the 1-th Euclidean fractional integral is bounded from $L^q(\mathbb{R}^n)$ into $L^{p_0}(\mathbb{R}^n)$ we deduce
$$
J_k(y)\leq C\sum_{j=0}^\infty \frac{2^{-j\alpha}}{(2^j\gamma  (x_k))^{n-1}}\left(\int_{\widetilde{Q}_j^k}(|b(x)-b(y)||g(x)|)^{p_0'}dx\right)^{1/p_0'}
\left(\int_{\widetilde{Q}_j^k}V(z)^qdz\right)^{1/q},\,\,\,y\in Q_k.
$$
Moreover, by using the doubling property of $V$ it follows that, for some
$\mu >0$,
\begin{eqnarray*}
\left(\int_{\widetilde{Q}_j^k}V(z)^qdz\right)^{1/q}
&\leq&C(2^j\gamma (x_k))^{-n/q'}2^{j\mu}
\int_{Q_k}V(z)dz\\
&\leq&C(2^j\gamma (x_k))^{-n/q'}2^{j\mu}\gamma (x_k)^{n-2},\quad y\in Q_k,
\end{eqnarray*}
where in the last inequality we just use the definition of $\gamma$.

Then,
$$
J_k(y)\leq C\sum_{j=0}^\infty \frac{2^{j(\mu-\alpha-n+1-n/q)}}{\gamma
(x_k)^{1+n/q'}}\left(\int_{\widetilde{Q}_j^k}(|b(x)-b(y)||g(x)|)^{p_0'}dx\right)^{1/p_0'},\,\,\,y\in Q_k.
$$
Since $p'>p_0'$, calling $\nu=p_0'(p'/p_0')'$, H\"older inequality and \cite[Proposition
3]{BHS3} imply that, for some $\theta'>\theta$,
\begin{eqnarray*}
J_k(y)&\leq&C\sum_{j=0}^\infty \frac{2^{j(\mu-\alpha-n+1-n/q)}}{\gamma
(x_k)^{1+n/q'}}
\left(\int_{\widetilde{Q}_j^k}|b(x)-b(y)|^
{\nu}dx\right)^{\frac{1}{\nu}}\left(\int_{\widetilde{Q}_j^k}|g(x)|^{p'}dx\right)^{1/p'}\\
&\leq&C\sum_{j=0}^\infty \frac{2^{j(\mu-\alpha-n+1-n/q)}}{\gamma
(x_k)^{1+n/q'}}
\left\{\left(\int_{\widetilde{Q}_j^k}|b(x)-b_{Q_k}|^
{\nu}dx\right)^{\frac{1}{\nu}}\right.\\
&&+\left.\left(\int_{\widetilde{Q}_j^k}|b(y)-b_{Q_k}|^
{\nu}dx\right)^{\frac{1}{\nu}}\right\}\left(\int_{\widetilde{Q}_j^k}|g(x)|^{p'}dx\right)^{1/p'}\\
&\leq&C\sum_{j=0}^\infty \gamma(x_k)^{-n/p'}2^{-j(n-2+\alpha-\mu+n/p')}\\
&&\times\Big((j+1)2^{j\theta'}||b||_{{\rm
BMO}_\theta(\gamma)}+|b(y)-b_{Q_k}|\Big)
\left(\int_{\widetilde{Q}_j^k}|g(x)|^{p'}dx\right)^{1/p'},\quad
y\in Q_k.
\end{eqnarray*}
%In a similar way we get
%$$
%I_k(y)\leq C\gamma(x_k)^{-n/q'-1+n(p'-q_0')/(p'q_0')}\Big(||b||_{{\rm BMO}(\mathbb{R}^n)}+
%|b(y)-b_{Q_k}|\Big)
%\left(\int_{|x_k-x|<4\gamma(x_k)}|g(x)|^{p'}dx\right)^{1/p'},
%$$
%$y\in Q_k$.

Then, by taking into account the properties of the sequence
$\{Q_k\}_{k\in \mathbb{N}}$ and Minkowski inequality, we can choose
$\alpha$ large enough such that
\begin{eqnarray*}
&&\Big(\sum_{k=0}^\infty \int_{Q_k}J_k(y)^{p'}dy\Big)^{1/p'}\\
&\leq& C||b||_{{\rm BMO}_\theta(\gamma)}\left(\sum_{k=0}^\infty \left(\sum_{j=0}^\infty 2^{-j(n-2+\alpha-\mu+\theta'+\frac{n}{p'})}(j+1) \left(\int_{\widetilde{Q}_j^k}|g(x)|^{p'}dx\right)^{1/p'}\right)^{p'}\right)^{1/p'}\\
&\leq& C||b||_{{\rm BMO}_\theta(\gamma)}\sum_{j=0}^\infty 2^{-j-j(n-2+\alpha-\mu+\theta'+\frac{n}{p'})}(j+1)
\left(\sum_{k=0}^\infty
\int_{\widetilde{Q}_j^k}|g(x)|^{p'}dx\right)^{1/p'}\\
&\leq&
C||b||_{{\rm BMO}_\theta(\gamma)}||g||_{L^{p'}(\mathbb{R}^n)}.
\end{eqnarray*}
Thus, we prove that
\begin{equation}\label{Co21}
||T _{2,2}^*(g)||_{L^{p'}(\mathbb{R}^n)}\leq
C||g||_{L^{p'}(\mathbb{R}^n)}, \quad g\in L^{p'}(\mathbb{R}^n).
\end{equation}
By combining (\ref{Co19}), (\ref{Co20}) and (\ref{Co21}) we obtain that
$$
||T _2^*(g)||_{L^{p'}(\mathbb{R}^n)}\leq
C||g||_{L^{p'}(\mathbb{R}^n)},\quad g\in L^{p'}(\mathbb{R}^n).
$$
Hence, the operator $T _2$ is bounded from $L^p(\mathbb{R}^n)$ into
itself.

Finally, the $L^p$-boundedness of $T _1$ and $T _2$  allows us to
conclude that the operator $C_{b,\ell }^{{\mathcal L},*}$ is bounded
from $L^p(\mathbb{R}^n)$ into itself.
\end{proof}

We now prove the $L^p$-boundedness properties of the variation
operators associated with $\{C_{b,\ell
}^{\mathcal{L},\varepsilon}\}_{\varepsilon>0}$ that are established
in Theorem \ref{VarComm}.

\vspace{5mm}

\begin{proof}[\bf Proof of Theorem \ref{VarComm}.] Assume that $\rho>2$. We
consider the operators
$$
C_{b,\ell }^{\mathcal{L},{\rm loc}}(f)(x)=\lim_{\varepsilon
\rightarrow 0^+}\int_{\varepsilon <|x-y|<\gamma (x)
}(b(x)-b(y))R_\ell ^\mathcal{L}(x,y)f(y)dy,
$$
and
$$
C_{b,\ell}^{ {\rm loc}}(f)(x)=\lim_{\varepsilon \rightarrow
0^+}\int_{\varepsilon <|x-y|<\gamma (x)}(b(x)-b(y))R_\ell
(x-y)f(y)dy,
$$
and we define the truncations $C_{b,\ell
}^{\mathcal{L},\varepsilon,{\rm loc}}(f)$ and $C_{b,\ell
}^{\varepsilon,{\rm loc}}(f)$, $\varepsilon>0$, in the usual way.

If $\{\varepsilon _j\}_{j\in \mathbb{N}}$ is a real decreasing sequence that converges to zero, we can write
\begin{eqnarray*}
\left(\sum_{j=0}^\infty|C_{b,\ell }^{\mathcal{L},\varepsilon _j}(f)(x)-C_{b,\ell }^{\mathcal{L},\varepsilon _{j+1}}(f)(x)|^\rho \right)^{1/\rho }&&\\
%&\hspace{-10cm}\leq&\hspace{-5cm}\left(\sum_{j=0}^\infty|C_{b,\ell ,{\rm loc}}^{\mathcal{L},\varepsilon _j}(f)(x)-C_{b,\ell ,{\rm loc}}^{\mathcal{L},\varepsilon _{j+1}}(f)(x)-(C_{b,\ell ,{\rm loc}}^{\varepsilon _j}(f)(x)-C_{b,\ell ,{\rm loc}}^{\varepsilon _{j+1}}(f)(x))|^\rho \right)^{1/\rho }\\
%&\hspace{-10cm}&\hspace{-5cm}+\left(\sum_{j=0}^\infty |C_{b,\ell ,{\rm loc}}^{\varepsilon _j}(f)(x)-C_{b,\ell ,{\rm loc}}^{\varepsilon _{j+1}}(f)(x)|^\rho\right)^{1/\rho }\\
%&\hspace{-10cm}&\hspace{-5cm}+\left(\sum_{j=0}^\infty \left|\int_{\varepsilon _{j+1}<|x-y|<\varepsilon _j, |x-y|\ge\gamma (x)}(b(x)-b(y))R_\ell ^\mathcal{L}(x,y)f(y)dy\right|^\rho \right)^{1/\rho}\\
&\hspace{-10cm}\leq&\hspace{-5cm}\left(\sum_{j=0}^\infty \left|\int_{\varepsilon _{j+1}<|x-y|<\varepsilon _j, |x-y|<\gamma (x)}(b(x)-b(y))(R_\ell ^\mathcal{L}(x,y)-R_\ell (x-y))f(y)dy\right|^\rho \right)^{1/\rho}\\
&\hspace{-10cm}&\hspace{-5cm}+\left(\sum_{j=0}^\infty \left|\int_{\varepsilon _{j+1}<|x-y|<\varepsilon _j, |x-y|\ge\gamma (x)}(b(x)-b(y))R_\ell ^\mathcal{L}(x,y)f(y)dy\right|^\rho \right)^{1/\rho}\\
&\hspace{-10cm}&\hspace{-5cm}+\left(\sum_{j=0}^\infty \left|\int_{\varepsilon _{j+1}<|x-y|<\varepsilon _j, |x-y|<\gamma (x)}(b(x)-b(y))R_\ell (x-y)f(y)dy\right|^\rho \right)^{1/\rho}\\
&\hspace{-10cm}\leq&\hspace{-5cm}\int_{|x-y|<\gamma (x)}|b(x)-b(y)||R_\ell ^\mathcal{L}(x,y)-R_\ell (x-y)||f(y)|dy\\
%&\hspace{-10cm}&\hspace{-5cm}+\sup_{\{t_j\}_{j\in \mathbb{N}}\searrow 0, t_1=\gamma (x)}\left(\sum_{j=0}^\infty \left|\int_{t_{j+1}<|x-y|<t_j}(b(x)-b(y))R_\ell (x-y)f(y)dy\right|^\rho \right)^{1/\rho}\\
&\hspace{-10cm}&\hspace{-5cm}+\int_{|x-y|\ge\gamma (x)}|b(x)-b(y)||R_\ell ^\mathcal{L}(x,y)||f(y)|dy+V_\rho(C_{b,\ell
}^{\varepsilon,{\rm loc}})(f)(x)\\
%&\hspace{-10cm}\leq&\hspace{-5cm}\int_{|x-y|<\gamma(x)}|b(x)-b(y)||R_\ell ^\mathcal{L}(x,y)-R_\ell (x-y)||f(y)|dy\\
%&\hspace{-10cm}&\hspace{-5cm}+V_\rho (C_{b,\ell }^\varepsilon)(f)(x)+\int_{|x-y|>\gamma (x)}|b(x)-b(y)||R_\ell ^\mathcal{L}(x,y)||f(y)|dy,\quad x\in \mathbb{R}^n.
\end{eqnarray*}
Hence,
$$
V_\rho (C_{b,\ell }^{\mathcal{L},\varepsilon})(f)\leq T_1(f)+T_2(f)+V_\rho
(C_{b,\ell }^{\varepsilon,\rm{loc}})(f),
$$
where the operator $T_1$ and $T_2$ are the ones defined in the proof of Proposition \ref{PvComm}.

According to the $L^p$-boundedness properties of the operators $T_1$
and $T_2$ (see the proof of Proposition \ref{PvComm}) and
Proposition \ref{master}, we conclude that the variation operator
$V_\rho (C_{b,\ell }^{\mathcal{L},\varepsilon })$ is bounded from
$L^p(\mathbb{R}^n)$ into itself.

\end{proof}

\begin{Rm} In \cite{Sh} (see also \cite{BHS1} and \cite{BHS2}) it is considered, for every $\ell =1,\cdots,n$, the adjoint $(R_\ell ^\mathcal{L})^*$ of $R_\ell ^\mathcal{L}$, when $V\in B_q$ with $\frac{n}{2}<q<n$. By proceeding as in the previous results  of this section we can prove the following properties.

Assume that $V\in B_q$, with $\frac{n}{2}<q$, $\ell =1,2,...,n$, and $p_0<p<\infty$, where $\frac{1}{p_0}=\Big(\frac{1}{q}-\frac{1}{n}\Big)_+$. For every $f\in L^p(\mathbb{R}^n)$ there exists the limit
$$
\lim_{\varepsilon \rightarrow 0^+}\int_{|x-y|>\varepsilon}R_\ell ^\mathcal{L}(y,x)f(y)dy,\quad \mbox{ a.e. }x\in \mathbb{R}^n,
$$
and defining the operator $\mathcal{R}_\ell ^\mathcal{L}$ on $L^p(\mathbb{R}^n)$ as
$$
\mathcal{R}_\ell ^\mathcal{L}f(x)=\lim_{\varepsilon \rightarrow 0^+}\int_{|x-y|>\varepsilon}R_\ell ^\mathcal{L}(y,x)f(y)dy,\quad \mbox{ a.e. }x\in \mathbb{R}^n,
$$
$\mathcal{R}_\ell ^\mathcal{L}$ is bounded from $L^p(\mathbb{R}^n)$ into itself.

Moreover, if $\rho >2$ the variation operator $V_\rho
(\mathcal{R}_\ell ^\mathcal{L})$ is bounded from $L^p(\mathbb{R}^n)$
into itself.

Suppose that $b\in BMO_\theta(\gamma)$. For every $f\in
L^p(\mathbb{R}^n)$, there exists the limit
$$
\lim_{\varepsilon \rightarrow 0^+}\int_{|x-y|>\varepsilon}(b(x)-b(y))R_\ell ^\mathcal{L}(y,x)f(y)dy,\quad \mbox{ a.e. }x\in \mathbb{R}^n,
$$
and the operator $\mathcal{C}_{b,\ell }^\mathcal{L}$ defined on $L^p(\mathbb{R}^n)$ by
$$
\mathcal{C}_{b,\ell }^\mathcal{L}(f)(x)=\lim_{\varepsilon \rightarrow 0^+}\int_{|x-y|>\varepsilon}(b(x)-b(y))R_\ell ^\mathcal{L}(y,x)f(y)dy,\quad \mbox{ a.e. }x\in \mathbb{R}^n,
$$
is bounded from  $L^p(\mathbb{R}^n)$ into itself.

Moreover, if $\rho >2$ the variation operator $V_\rho
(\mathcal{C}_{b,\ell }^{\mathcal{L},\varepsilon })$ is bounded from
$L^p(\mathbb{R}^n)$ into itself.
\end{Rm}

\begin{Rm} The fluctuations of a family $\{T_t\}_{t>0}$ of operators
when $t\to 0^+$ also can be analyzed by using oscillation operators
(see, for instance, \cite{Bou} and \cite{JKRW}). If
$\{t_j\}_{j\in \mathbb{N}}$ is a real decreasing sequence that
converges to zero, the oscillation operator $O(T_t;\{t_j\}_{j\in
\mathbb{N}})$ is defined by
$$
O(T_t;\{t_j\}_{j\in \mathbb{N}})(f)(x)=\Big(\sum_{j=0}^\infty \sup_{t_{j+1}\le \varepsilon_{j+1}<\varepsilon_j\le t_j}|T_{\varepsilon_j}f(x)-T_{\varepsilon_{j+1}}f(x)|^2\Big)^{1/2},\quad f\in L^p(\mathbb{R}^n).
$$
$L^p$-boundedness properties for the oscillation operators
associated with the heat semigroup, Riesz transforms and commutators
with the Riesz transforms in the Schr\"odinger setting can be
established by using the procedures developed in this paper.
\end{Rm}

\end{document}